\documentclass[11pt]{amsart}
\usepackage{fancyhdr}
\usepackage{txfonts}

\usepackage{amsmath,amsthm,amssymb,hyperref, mdframed}
\usepackage{mathrsfs} 
\usepackage{tikz}
\usepackage[all]{xy}
\usetikzlibrary{decorations.markings}
\usetikzlibrary{backgrounds,shapes}
\usetikzlibrary{patterns,hobby}
        \usepackage{pgfplots}
        \pgfplotsset{compat=1.6}

\usepackage{subfig}

\usepackage{color}
\usepackage{hyperref}
\hypersetup{
    colorlinks=true, 
    linktoc=all,     
    linkcolor=blue,  
    citecolor=blue,
    filecolor=blue,
    urlcolor=blue
}

\definecolor{aliceblue}{rgb}{0.9, 0.95, 1.0}
\definecolor{pallido}{RGB}{221,227,227}

\usepackage{geometry}
\newcommand\Z{{\mathbb Z}}

\newcommand{\C}{{\mathbb C}}
\newcommand{\R}{{\mathbb R}}

\newcommand{\Te}{Teichm\"{u}ller }

\newcommand{\pslc}{{\mathrm{PSL}_2 (\mathbb{C})}}
\newcommand{\pslr}{{\mathrm{PSL}_2 (\mathbb{R})}}

\newcommand{\slc}{{\mathrm{SL}_2 (\mathbb{C})}}

\newcommand{\affc}{\text{Aff}(\C)}
 
\newcommand{\cp}{\mathbb{C}\mathrm{P}^1}

\theoremstyle{plain}                    
\newtheorem{thm}{Theorem}[section]

\newtheorem{thma}{Theorem}

\newtheorem{cora}[thma]{Corollary}

\newtheorem{lem}[thm]{Lemma}

\newtheorem{prop}[thm]{Proposition}
\newtheorem{cor}[thm]{Corollary}

\theoremstyle{definition}
\newtheorem{defn}[thm]{Definition}

\theoremstyle{remark}
\newtheorem{rmk}[thm]{Remark}
\newtheorem{claim}{Claim}

\title[Monodromy of Schwarzian equations with regular singularities]{Monodromy of Schwarzian equations\\ with regular singularities}

\author{Gianluca Faraco}
\address{Max Planck Institute for Mathematics, Bonn, Germany}
\email{gianlucafaraco@mpim-bonn.mpg.de}

\author{Subhojoy Gupta}
\address{Department of Mathematics, Indian Institute of Science, Bangalore, India}
\email{subhojoy@iisc.ac.in}

\begin{document}

\begin{abstract} Let $S$ be a punctured surface of finite type and negative Euler characteristic. We determine all possible representations $\rho:\pi_1(S) \to \pslc$ that arise as the monodromy of the Schwarzian equation  on $S$ with regular singularities at the punctures. Equivalently, we determine the holonomy representations of complex projective structures on $S$, whose Schwarzian derivatives (with respect to some uniformizing structure) have poles of order at most two at the punctures. Following earlier work that dealt with the case when there are no apparent singularities, our proof reduces to the case of realizing a degenerate representation with apparent singularities. This mainly involves explicit constructions of complex affine structures on punctured surfaces, with prescribed holonomy. As a corollary, we determine the representations that arise as the holonomy of spherical metrics on $S$ with cone-points at the punctures. 
 \end{abstract}

\maketitle
\tableofcontents

\section{Introduction}

\noindent Consider the Schwarzian equation 
\begin{equation}\label{schwarz}
y^{\prime\prime} + \frac{1}{2} {q} y = 0
\end{equation} 
on a punctured Riemann sphere $X= \cp \setminus \{p_1,p_2,\ldots p_k\}$, where the prescribed meromorphic coefficient function $q$ has poles of order at most two at the punctures. This is a second-order complex linear differential equation with regular singularities  and thus admits meromorphic solutions that span a complex vector space of dimension two (see, for example \cite[\S15.3]{Ince}). The monodromy of the solutions determines a representation ${\rho}: \pi_1(X) \to \pslc$, and determining which monodromy groups appear has been a subject of classical work (\textit{e.g.} \cite{PActa}). 

\smallskip

\noindent More generally,  let $S$ be an arbitrary punctured Riemann surface of finite-type and negative Euler characteristic, and let $\Pi = \pi_1(S)$ be its fundamental group. When $S$ is equipped with a complex structure, the Schwarzian equation \eqref{schwarz} makes sense by passing to the universal cover; the coefficient $q$ is then the lift of a holomorphic quadratic differential on the surface that has poles of order at most two at the punctures. Equivalently, the solutions of \eqref{schwarz} also determine a \textit{complex projective (or $\cp$-) structure} on the surface, which is a geometric structure on the surface modelled on $\cp$; the monodromy of the solutions is then the holonomy or monodromy of the geometric structure (see \S\ref{prel} for a discussion). It has been a long-standing problem to determine which conjugacy-classes of representations $\rho:\Pi \to \pslc$ arise as the monodromy of such a $\cp$-structure, when one is allowed to vary the marked complex structure on $S$.  The work in \cite{GupMon1} provided an answer under the assumption that there is no \textit{apparent singularity} (see Definition \ref{appsing}). This clarified, in particular, a remark of Poincar\'{e} in \cite[pg.218]{PActa} concerning the case of the punctured Riemann sphere, where he wrote

\begin{quote}
    {``On peut \textit{en g\'{e}n\'{e}ral} trouver une \'{e}quation du $2^d$ ordre, sans points \`{a} apparence singuli\`{e}re qui admette un groupe donn\'{e}."}
\end{quote}

\noindent Indeed, we had showed in \cite{GupMon1} that the monodromy groups that do arise are exactly those that are  \textit{non-degenerate} as in Definition \ref{degen}. 

\smallskip

\noindent In this article, we drop the assumption that there are no apparent singularities, and solve the problem, providing a complete characterization of the monodromy groups of the Schwarzian equation with regular singularities. In other words, for a surface $S_{g,k}$ of genus $g$ and $k\geq 1$ punctures that has negative Euler characteristic, we determine the image of the monodromy map 
\begin{equation}\label{monmap} 
\Psi:  \mathcal{P}_g(k) \to \text{Hom}\big(\pi_1(S_{g,k}), \pslc\big)/\pslc 
\end{equation}
where $\mathcal{P}_g(k)$ is  the space of meromorphic projective structures on $S_{g,k}$, with respect to a choice of a marked complex structure, such that each puncture corresponds to a pole of order at most two, and the target space is the \textit{$\pslc$-representation variety}, the space of surface-group representations into $\pslc$ up to the action by conjugation. To be more precise, two representations are equivalent, i.e.\  $\rho_1 \sim \rho_2$ if the \textit{closures} of their $\pslc$-orbits intersect; this coincides with geometric invariant theory quotient (see, for example, \cite{Newstead}). We note here that with our condition on the orders of the poles, the dimensions of the two spaces in the domain and range of the monodromy map $\Psi$ coincide. The question of determining the image of the above monodromy map was  mentioned in \cite[Question 4, pg. 554]{Luo}. Our main result (Theorem \ref{thm1}) below answers this completely.  

\medskip 

\noindent To state our results, we recall the following definitions from \cite{GupMon1}: 

\begin{defn}\label{degen}   A representation $\rho:\Pi \to \pslc$ is said to be \textit{degenerate} if $\rho(\Pi)$ preserves a set $F$ on $\cp$ where either (a) $F = \{p\}$ (i.e.\ a global fixed point) and the monodromy around each puncture is a parabolic fixing $p$ or the identity element, or (b) $F = \{p,q\}$ and the monodromy around each puncture fixes $p$ and $q$. Otherwise, $\rho$ is said to be \textit{non-degenerate}.  
\end{defn}

\begin{defn}\label{appsing}
A representation $\rho:\Pi \to \pslc$ is said to have an \textit{apparent singularity} at a puncture of $S_{g,k}$ if $\rho(\gamma) = \text{Id}$ for the peripheral loop $\gamma$ around the puncture. We also say a complex projective structure has an \textit{apparent singularity} at a puncture, if the monodromy around it is trivial. In that case the puncture is either a branch-point or a regular point of the structure (see Definition \ref{brpt} and the discussion following it).
\end{defn}

\noindent We shall prove:

\begin{thma}\label{thm1}  Let $\Pi$ be the fundamental group of a surface $S_{g,k}$ of genus $g$ and $k\geq 1$ punctures, where $2-2g-k<0$. A representation $\rho:\Pi \to \pslc$ arises as the monodromy of a $\cp$-structure in $\mathcal{P}_g(k)$ if and only if one of the following hold:
\begin{itemize}
\item[(i)] $\rho$ is a non-degenerate representation, or
\item[(ii)] $\rho$ is a degenerate representation, with at least one apparent singularity, with the only exceptions being:
\begin{itemize}
    \item the trivial representation, when $g>0$ and $k=1$ or $2$, and 
    \item a representation whose image is a group of order two, when $g>0$ and $k=1$.
\end{itemize} 
\end{itemize}
\end{thma}

\medskip

\noindent The case of a \textit{closed} surface $S_g$ where $g\geq 2$ was handled by the work of Gallo-Kapovich-Marden in \cite{GKM}, who showed that non-elementary representations that lifts to $\slc$ are exactly those that arise as  monodromy representations of $\cp$-structures on $S_g$. Note that a non-elementary representation is automatically non-degenerate as defined above (for a comparison between these notions see \cite[\S2.4]{GupMon1}).  Gallo-Kapovich-Marden had also stated the problem of what happens for punctured surfaces, see \cite[Problem 12.2.1]{GKM}. This paper solves this --  the theorem above answers the ``existence" part of their problem, and for the ``nonuniqueness" part, we shall observe that the constructions in the paper imply the following:

\begin{cora}\label{cor:fiber} 
The non-empty fibers of the monodromy map \eqref{monmap} are infinite in cardinality. 
\end{cora}

\noindent We mention some further directions to ``explore the nonuniqueness" (quoting Gallo-Kapovich-Marden) at the end of this Introduction.

\medskip 

\noindent The case of a non-degenerate representation (case (i) in Theorem \ref{thm1}) follows from the work in \cite{GupMon1}; hence the present article, though independent of that paper, can be considered as its sequel. We provide a new and simplified discussion of the construction of that paper in \S\ref{proof1}, that used a geometric interpretation of certain cross-ratio coordinates introduced (in a more general context) by Fock-Goncharov in \cite{FG}. We also use one of the results of  Allegretti-Bridgeland in  \cite{AllBrid}  which implies that  in the case of no apparent singularities, the monodromy representation is necessarily non-degenerate (see Theorem \ref{ab61}). Indeed, together with Lemma \ref{trivex} concerning the case when the entire monodromy representation is trivial, and Lemma \ref{excase} concerning the case when the monodromy group, i.e.\ the image of the monodromy representation, has order two,  this establishes the ``only if" direction in Theorem \ref{thm1}.   
\medskip 

\noindent The main work in this paper is to handle the remaining case of degenerate representations, namely, to construct $\cp$-structures on the punctured surface with a specified monodromy representation as in case (ii) in Theorem \ref{thm1}. Note that this is specific to punctured surfaces, as when the surface is closed, the analogue of the degenerate case, i.e.\ elementary representations, does not arise.
In case (a) of Definition \ref{degen} of a degenerate representation, namely when the image of a degenerate representation has a global fixed point on $\cp$,  it can be conjugated into the subgroup $\text{Aff}(\C) = \{ z\mapsto a z + b\ \vert\ a \in \C^\ast,\ b \in \C\} $ of complex affine maps.  The strategy of the proof in this case is to construct a \textit{(complex) affine structure} on $S_{g,k}$, whose monodromy is the prescribed representation. Thus, in the course of the proof of Theorem \ref{thm1}, we prove the following result concerning the monodromy groups of affine structures on a punctured surface:

\begin{thma}\label{thm2} Let $\Pi$ be the fundamental group of $S_{g,k}$ as in Theorem \ref{thm1}, such that the number of punctures $k\geq 2$.  Then any non-trivial representation $\rho:\Pi \to \text{Aff}({\C})$ arises as the monodromy of a complex affine structure on $S_{g,k}$. If $k\ge3$, then every representation is realizable as the monodromy of a complex affine structure on $S_{g,k}$.
\end{thma} 

\noindent Our construction of an affine structure involves considering unbounded polygons on $\C$ with sides paired by affine maps, and certain gluing methods which we introduce (\textit{e.g.} Definition \ref{glue-new}). The resulting affine surface acquires branch-points that arise from the vertices of the polygon after the identification, that we delete to obtain an affine structure on a punctured surface. Since there is always an additional puncture at infinity, our proof of Theorem \ref{thm2} requires at least two punctures.

\smallskip

\noindent In the case that the surface has exactly one puncture, we rely on a refinement of the construction alluded to above. Our strategy then is to use the action of the mapping class group to choose a generating set of the fundamental group so that the corresponding polygonal curve bounds an immersed disk in $\cp$. When the representation is \textit{co-axial}, \textit{i.e.} when the pair of points in (b) of Definition \ref{degen} is globally fixed, we also introduce some further new operations (\textit{e.g.} Definition \ref{bubbhand}). In these cases, even when the monodromy is affine, we obtain a \textit{projective} structure with the desired monodromy, which is not necessarily affine. 

\smallskip

\noindent Indeed, the recent work of \cite{Fils2} shows that in the case of a once-punctured surface, there are some further exceptions to the existence of an affine structure (see \cite[Proposition 1.5]{Fils2}).  In fact, this case can also be handled using the main result of \cite{Fils2}; we describe this alternative approach in \S\ref{altproofk1}, which involves proving the existence of a projective structure  with the desired monodromy and a single branch-point, except when the monodromy group is of order two (\textit{c.f.} Lemma \ref{excase}). 
\noindent The remaining degenerate representations still fall in (b) of Definition \ref{degen}, but the pair of points is $\cp$ is preserved, not globally fixed by such a representation; in this case our argument relies on similar techniques. 


\medskip

\noindent One feature of our methods is that they are constructive, and can be potentially implemented in an algorithm. 
The constructions we introduce also yield results on other special classes of $\cp$-structures that are geometric structures in their own right. We mention some of these below, and leave the discussion of others (\textit{e.g.} half-translation structures, see Definition \ref{def_halftrans}) for future papers. 

\medskip
\noindent \textit{\textbf{Translation structures.}} A special case of a complex affine structure is a \textit{translation structure} which has monodromy in the (smaller) subgroup of complex {translations} $\{z\mapsto z + a \ \vert\ a \in \C\} \cong \C$. Such a structure acquires a holomorphic (abelian) differential $\omega$ that is the pullback of the differential $dz$ on $\C$ via the charts, and prescribing its monodromy is equivalent to prescribing the periods of $\omega$. For a punctured surface, in our recent work in \cite{CFG}  we determine the possible periods for meromorphic differentials in each strata with prescribed orders of zeroes and poles. In this paper, we include a proof of the translation-structure case of Theorem \ref{thm2}, in part to motivate the proof of the general case. This follows the general strategy employed in \cite{CFG}, although some arguments are different. 

\medskip 

\noindent \textit{\textbf{Spherical structures.}} Another class of projective structures that are interesting in their own right are \textit{spherical cone-metrics} on a surface. For such a structure, the developing map is to the round sphere $\mathbb{S}^2$, and the monodromy lies in $\text{PSU}(2) \cong \text{SO}(3,\mathbb{R})$. Note that the monodromy around any cone point of angle $\alpha$ is an elliptic rotation by that angle. The space $\mathcal{MS}ph_{g,k}(\vartheta)$ of such spherical cone-metrics on $S_{g,k}$ with a set of prescribed cone-angles $\vartheta$ at the punctures admits a forgetful projection to the moduli space $\mathcal{M}_{g,k}$; this has been a subject of much study --  see \cite{Mondello-Panov1} and \cite{Mondello-Panov2} and the references therein. What has been less studied is the {monodromy map} in this context, namely, the forgetful projection to the space of surface-group representations into $\text{SO}(3,\mathbb{R})$. 

\smallskip

\noindent Here, we provide the following corollary to our main theorem:

\begin{cora}\label{cor:spher} Let $\Pi$ be the fundamental group of $S_{g,k}$ as in Theorem \ref{thm1} and let $\rho:\Pi \to \text{\emph{SO}}(3,\mathbb{R})$ be a representation. Then $\rho$ is the monodromy of some spherical cone-metric on $S_{g,k}$ with cone-points only at the punctures, if and only if it satisfies conditions (i) and (ii) in the statement of Theorem \ref{thm1}. 
\end{cora}


\noindent It worth noting that prescribing the monodromy around the punctures determines the cone-angles only modulo $2\pi$; prescribing these cone-angles \textit{in addition to} the monodromy is a more delicate problem.  For a punctured sphere, the work of Eremenko in \cite{Eremenko} solves this for the case when the prescribed monodromy is co-axial (which, in this context, is equivalent to being degenerate).  


\medskip 

\noindent \textit{\textbf{Branched $\cp$-structures.}} 
Finally, $\cp$-structures on a surface which are allowed to have additional branch-points are of independent interest; see Definition \ref{brpt}. These were first studied in \cite{Mandelbaum}; see \cite{CDF} or  \cite{BDG} for more recent results. As a consequence of Theorem \ref{thm1} we obtain the following result concerning such branched projective structures on a punctured surface:

\begin{cora}\label{cor:branch}
Let $\Pi$ be the fundamental group of $S_{g,k}$ as in Theorem \ref{thm1}. Every representation $\rho:\Pi \to\pslc$ arises as the monodromy of some (possibly branched) $\cp$-structure on $S_{g,k}$, with at most two branch-points. 
\end{cora}

\noindent The recent work in \cite{Nas} also proves similar results, and develops completely different techniques for constructing branched projective structures with prescribed monodromy. 

\bigskip 

\noindent It remains to understand the fibers of the monodromy map \eqref{monmap} better, for example how projective structures in the same fiber are related, and we plan to address this in future work. The recent work \cite{BBCR} deals with the case when the surface is the thrice-punctured sphere $S_{0,3}$. For a closed surface, this was handled in the work of Baba in \cite{Baba2}, and it is conceivable that one can develop the analogues of the techniques there, in the punctured case. The main result of \cite{Luo} implies that  the fiber $\Psi^{-1}(\rho)$ is a discrete set when the monodromy representation $\rho$ is irreducible, and the monodromy around each puncture is  non-trivial and not parabolic. In contrast, in the case all the punctures are branch-points and $\rho$ is quasi-Fuchsian, the fibers can be connected, and the work in \cite{CDF} determines when that happens.  

\bigskip

\noindent \textbf{Plan of the paper.} The paper is organized as follows. In section \S\ref{prel} we begin by recalling the basic background of meromorphic projective structures on surfaces and their monodromy representations. In \S\ref{proof1}, we deal with the case of non-degenerate representations which essentially follows by \cite{GupMon1}, an earlier work by the second named author. The proof of the Theorem \ref{thm1} for degenerate representations is the main core of the present paper and it is developed along sections \S\ref{proof2}, \S\ref{saff2} and \S\ref{ssdih}. In section \S\ref{proof2} we provide a proof of Theorem \ref{thm2} which handles affine representations in the case of surfaces with at least two punctures. In section \S\ref{saff2}, instead, we deal with the complementary case of affine representations for once-punctured surfaces. Finally, in section \S\ref{ssdih} we deal with dihedral representations. The final section \S\ref{coros} concludes with the proofs of Corollaries \ref{cor:fiber}, \ref{cor:spher} and \ref{cor:branch}. 

\bigskip

\noindent \textbf{Acknowledgments.} Half of the present work was done while GF was affiliated with the IISc Bangalore. Despite the author could not fully enjoy his period there because of the ongoing pandemic, he is grateful to the Department of Mathematics and everyone who helped him. SG is grateful for the support of the Department of Science and Technology (DST) MATRICS Grant no. MT/2017/000706.

\section{Preliminaries}\label{prel}

\noindent Let $S$ be a surface of finite type (open or closed) and of negative Euler characteristic. 

\subsection{Meromorphic projective structures}\label{mps} A \textit{$\cp$-structure} or a \textit{(complex) projective structure} on a surface $S$ is a maximal atlas of charts to $\cp$ such that the transition functions are M\"{o}bius maps, \textit{i.e.} elements of $\pslc$; in other words, it is a $(G,X)$-structure where $G=\pslc$ and $X = \cp$. Such a geometric structure can be equivalently described in terms of the \textit{developing map} $f:\widetilde{S} \to \cp$ and the \textit{monodromy} (or holonomy)  representation $\rho:\pi_1(S) \to \pslc$; the developing map is a $\rho$-equivariant immersion,  satisfying  $f(\gamma \cdot x) = \rho(\gamma) \cdot f(x)$ for each $x\in \widetilde{S}$ and each $\gamma \in \pi_1(S)$.  Note that the developing map is determined up to a post-composition by a M\"{o}bius map $A$, and $\rho$ is defined up to conjugation, so that the pair $(\text{dev},\rho)$ and $(A\cdot \text{dev}, A\cdot \rho\cdot A^{-1})$ define the same $\cp$-structure.  

\smallskip 

\noindent An example of a $\cp$-structure on $S$ is a \textit{hyperbolic} structure, where the developing map develops into the upper hemisphere of $\cp$ (that can be identified with the hyperbolic plane $\mathbb{H}^2$), and the monodromy is a  discrete (Fuchsian) representation $\rho:\pi_1(S) \to \pslr$. 

\medskip 

\noindent We often consider a $\cp$-structure with an additional \textit{marking}, that is, a additional choice of a homeomorphism that identifies $S$ with our complex projective surface, such that two marked structures that differ by a homeomorphism homotopic to the identity are considered equivalent.  

\noindent A change of marking on a marked $\cp$-structure is effected by the action of an element of the mapping class group of $S$; this  changes the monodromy representation by  pre-composing with the corresponding automorphism of $\pi_1(S)$, which is the action of the mapping class group on the $\pslc$-representation variety. Note that for a punctured surface, an element of its mapping class group can permute the punctures.  Thus, while seeking a $\cp$-structure with a prescribed monodromy representation $\rho$, it suffices to find one whose monodromy lies in the mapping class group orbit of $\rho$. We shall use this to our advantage later in the paper.

\medskip

\noindent The relation with the Schwarzian equation  arises through the \textit{Schwarzian derivative} $\tilde{q} = S(f) dz^2$ of the developing map $f$, where 
\begin{equation}\label{schw-deriv}
S(f) = \left( \frac{f^{\prime\prime}}{f^\prime}\right)^{\prime}  - \frac{1}{2} \left( \frac{f^{\prime\prime}}{f^\prime}\right)^{2}. 
\end{equation}
The quadratic differential thus defined is invariant when $f$ is pre-composed by a M\"{o}bius transformation; since $f$ is $\rho$-equivariant, $\tilde{q}$ descends to a quadratic differential $q$ on $S$,  that is holomorphic with respect to the complex structure induced by the $\cp$-structure. 

\noindent Conversely, fix a complex structure on $S$ such that $S = \mathbb{H}^2/\Gamma$, where $\Gamma$ is a Fuchsian representation. Let $y_0$ and $y_1$ be two linearly independent solutions of  the Schwarzian equation 
\begin{equation}\label{schwarz1}
y^{\prime\prime} + \frac{1}{2} {\tilde{q}} y = 0
\end{equation} 
on the universal cover $\widetilde{S} = \mathbb{H}^2$, where $\tilde{q}$ is a $\Gamma$-invariant holomorphic quadratic differential. 
Then the ratio $f: = y_0/y_1$ defines a developing map to $\cp$, that is equivariant with respect to a monodromy representation $\rho:\Gamma \to \pslc$, defining a $\cp$-structure on $S$.  It is an exercise to show that then the Schwarzian derivative of $f$ equals $\tilde{q}$, so that this is indeed an inverse construction. \\

\noindent For a \textit{closed} surface $S$,  the deformation space  $\mathcal{P}(S)$ of marked $\cp$-structures on $S$ can be identified, via the correspondence sketched above, with the space of holomorphic quadratic differentials $\mathcal{Q}(S)$ that forms a bundle over \Te space $\mathcal{T}(S)$, where the fiber over a marked Riemann surface $X$ is the vector space $\mathcal{Q}(X)$ of holomorphic quadratic differentials on $X$. For details, see for example \cite{Dum} or \cite{Gunn}; for more on the geometry of these spaces, see \cite{Faraco}.

\noindent In the case that $S$ is not closed, the above correspondence still holds, except that now the space of holomorphic quadratic differentials on an open Riemann surface is infinite dimensional. However, we can restrict to \textit{meromorphic} quadratic differentials, where the punctures of $S$ are either removable singularities or poles of finite order: 

\begin{defn} A \textit{meromorphic projective structure} on a surface $S_{g,k}$ of negative Euler characteristic is a $\cp$-structure such that if we equip $S_{g,k}$ with a complete hyperbolic metric of finite volume, such that the universal cover $\widetilde{S_{g,k}} \cong \mathbb{H}^2$, then the Schwarzian derivative of the developing map $f:\mathbb{H}^2 \to \cp$ descends to a meromorphic quadratic differential on the surface. As described in the Introduction, in this paper, we shall consider the space $\mathcal{P}_g(k)$ of marked meromorphic projective structures on $S_{g,k}$ whose corresponding meromorphic quadratic differentials have poles of order at most two at the punctures.
\end{defn}

\textit{Remark.} The above definition, and the property that the poles have order at most two, is independent of the choice of complete hyperbolic structure, which merely serves as a ``reference" projective structure. See also \cite[Definition 3.1]{AllBrid}.

\smallskip 

\noindent Any $\cp$-structure on the closed surface $S_g$ becomes an example of a meromorphic projective structures in  $\mathcal{P}_g(k)$ after $k$ points are deleted; another set of examples include \textit{branched} projective structures on $S_g$ mentioned in the Introduction (see, for example, \cite{Mandelbaum}) with $k$ branch-points (as defined below) after they are deleted to form the $k$ punctures.  

\begin{defn}[Branch-point]\label{brpt} A \textit{branch-point} of a branched projective structure is a point around which  the developing map is of the form $z\mapsto z^n$ in local coordinates, where $n>1$. Note that its Schwarzian derivative, as computed by \eqref{schw-deriv} has a pole of order two at such a point. Away from the branch-points, a branched projective structure is a $\cp$-structure in the usual sense, i.e.\ the developing map is an immersion and each point is \textit{regular} or unbranched. 
\end{defn} 

\noindent For an account of the solutions of the Schwarzian equation \eqref{schwarz} around a pole of order two of $q$, see \cite[\S2.3]{GupMon1} or \cite[\S4.1]{BBCR} for a summary and Chapter IX of \cite{Saint-Gervais} for more details. In fact the eigenvalues of the monodromy around such a pole is determined by the \textit{quadratic residue} at the pole  (i.e.\ the coefficient of the $z^{-2}$ term in the Laurent series expansion of $q$) -- see Lemma 2.3 of \cite{GupMon1} and the remark following that, for a statement summarizing that relation. 
In particular, it turns out that the definition of an \textit{apparent singularity} as in Definition \ref{appsing} is equivalent to saying that the Schwarzian equation has meromorphic solutions around the pole -- see \S2.3 of \cite{GupMon1} for a discussion. In particular, as already noted in Definition \ref{appsing}, the $\cp$-structure will have a branch-point (as above) or a regular point at such a puncture; we shall use this later in the paper.

\medskip 

\noindent We mention here that one can also define spaces of meromorphic projective structures with poles of order \textit{greater}  than two, where the corresponding Schwarzian equation has \textit{irregular} singularities; see \cite{GM1} for an account, including a description of the asymptotic of the developing map around such a singularity. The (framed) monodromy representations of such structures (\textit{c.f.} the next section) were characterized in \cite{GM2}. 

\subsection{Framed representations} 

\noindent Following Fock-Goncharov (in \cite{FG}), a \textit{framed representation} of $\Pi = \pi_1(S_{g,k})$ to $\pslc$ is a representation $\rho:\Pi \to \pslc$ together with a \textit{framing}, which, roughly speaking, is an assignment of  a point in $\cp$ (which is a ``flag" in $\C^2$) to each puncture of $S_{g,k}$. More precisely,  let the \textit{Farey set} $F_\infty$ be the set of points on the ideal boundary of the universal cover of $S_{g,k}$, that  corresponding to the lifts of the punctures; a \textit{framing} of a $\cp$-structure on $S_{g,k}$ is a  a $\rho$-equivariant map  $\beta: F_\infty \to \cp$.  The \textit{moduli space of framed representations} $\widehat{\chi}(\Pi)$ is then the orbit space of such pairs under the action of $\pslc$ that identifies $(\rho,\beta) \sim (A\cdot \rho \cdot A^{-1}, A\cdot \beta)$ for each $A\in \pslc$.  The space $\widehat{\chi}(\Pi)$   is in fact a moduli stack -- see Lemma 1.1 and Definition 2.1  of \cite{FG}, or \S4.1 and Lemma 9.1 of \cite{AllBrid}.\\

\noindent In \cite{AllBrid} Allegretti-Bridgeland defined the following notion (see also \cite{GM2}) :

\begin{defn}\label{degen-f} A \textit{degenerate framed representation}  is a pair $(\rho,\beta) \in \widehat{\chi}(\Pi)$   that satisfies one of the following conditions:
\begin{itemize} 
\item[(i)] There is a set of two points $F = \{p_-,p_+\} \in \cp$ such that the image of $\beta$ lies in $F$, and for any peripheral loop $\gamma$, $\rho(\gamma)$ fixes the points in $F$.
\item[(ii)] There is a single point $p_0 \in \cp$ such that  the image of $\beta$ lies in $F$, and for any peripheral loop $\gamma$, $\rho(\gamma)$ is a parabolic element fixing $p_0$, or is the identity element.  
\end{itemize} 
A framed representation is \textit{non-degenerate} if it is not degenerate.
\end{defn} 

\begin{rmk}
Since properties (i) or (ii) above are invariant under the $\pslc$-action described above, the notion  is well-defined for elements of $\widehat{\chi}(\Pi)$. 
\end{rmk}


\noindent The following result from \cite{GupMon1} allows us to go back and forth between a non-degenerate framed representation and a non-degenerate representation in the sense of Definition \ref{degen}:

\begin{prop}[Proposition 3.1 of \cite{GupMon1}]\label{prop31gup}  Forgetting the framing of a degenerate (respectively, non-degenerate) framed  representation yields a representation $\rho:\Pi \to \pslc$ that is degenerate (respectively, non-degenerate). Moreover, a non-degenerate representation $\rho$ can be equipped with a framing $\beta$ such that the pair $(\rho,\beta)$ is a non-degenerate framed representation.
\end{prop}


\noindent The following result is a direct consequence of Theorem 6.1 of \cite{AllBrid} and Proposition \ref{prop31gup} :

\begin{thm}\label{ab61}  If $\rho:\Pi \to \pslc$ is the monodromy representation of a meromorphic $\cp$-structure in  $\mathcal{P}_g(k)$ with no apparent singularities, then $\rho$ is non-degenerate.
\end{thm}

\section{Proof of Theorem \ref{thm1}: Non-degenerate representations}\label{proof1}

\noindent In this section we shall fix an arbitrary non-degenerate representation $\rho:\Pi \to \pslc$, and show that there is a meromorphic projective structure in the space $\mathcal{P}_g(k)$ that has monodromy $\rho$. 
This proves the ``if" direction of Theorem \ref{thm1} in case (i) of the statement of the theorem. 
The proof is exactly the same as that in \cite{GupMon1}; the key observation is that our assumption in that paper that $\rho$ has no apparent singularity in fact plays no role in our construction. In what follows we provide a condensed (and simplified) account of this construction, and refer at times to \cite{GupMon1} for further discussion. 

\smallskip

\noindent The first observation is that by Proposition \ref{prop31gup} we can define a framing $\beta:F_\infty \to \cp$ such that the pair $(\rho,\beta)$ is a non-degenerate framed representation in the sense of Definition \ref{degen-f}. 

\subsection{Fock-Goncharov coordinates} In \cite{FG} Fock-Goncharov described coordinates on the moduli space of framed representations $\widehat{\chi}(\Pi)$ as follows.

\smallskip

\noindent Let $\hat{\rho}  = (\rho, \beta)$ be a framed representation. Choose an ideal triangulation $T$ of $S_{g,k}$; in the universal cover, this lifts to an ideal triangulation $\tilde{T}$ of the universal cover, with ideal vertices in the Farey set $F_\infty$. For each edge $e \in {T}$, choose a lift $\tilde{e} \in \tilde{T}$. Let  $p_1,p_2,p_3,p_4$ be the four ideal vertices of the two ideal triangles adjacent to $\tilde{e}$, in counter-clockwise order on the ideal boundary $\partial_\infty \widetilde{S_{g,k}} \cong \mathbb{S}^1$. Associated with $e$, we can then define a \textit{(complex) cross ratio} 
\begin{equation}\label{cross}
C(\hat{\rho}, e) = \frac{(z_1 - z_2) (z_3- z_4)}{(z_2- z_3) (z_1 - z_4)}
\end{equation} 
where $z_i = \beta(p_i)$ for $i=1,2,3,4$. 
Note that this is well-defined since by the equivariance of the framing $\beta$,  a different choice of lift of $e$ would yield a quadruple of points that differs by a M\"{o}bius transformation, that has the same cross-ratio. 

\noindent The set of such cross-ratios defines an element of $(\C^\ast)^N$ where $N$ is the number of edges of $T$, and this tuple is said to be the \textit{Fock-Goncharov coordinates} of the framed representation $\hat{\rho}$, with respect to our choice of ideal triangulation. 

\noindent Note, however, that the cross-ratio associated with $e$ above is well-defined only when the quadruple of points $z_1,z_2,z_3,z_4$ are \textit{distinct}, and hence for a fixed $T$, the Fock-Goncharov coordinates are only well-defined for a \textit{generic} framed representation.  We shall use the following theorem of Allegretti-Bridgeland: 

\begin{thm}[Theorem 9.1 of \cite{AllBrid}]  For any non-degenerate framed representation $\hat{\rho}$, there exists an ideal triangulation $T$ such that the Fock-Goncharov coordinates  of $\hat{\rho}$ with respect to $T$  are  well-defined.
\end{thm} 

\noindent Henceforth, in this section, we shall assume that we have fixed such a choice of ideal triangulation $T$, for the non-degenerate framed representation $\hat{\rho}= (\rho, \beta)$ we fixed at the beginning of this section. 

\subsection{Pleated planes in $\mathbb{H}^3$}\label{pleatedplane} Recall that $\cp$ is the ideal boundary of hyperbolic $3$-space $\mathbb{H}^3$.
The Fock-Goncharov coordinates of a framed representation $\hat{\rho}= (\rho,\beta)$ with respect to an ideal triangulation $T$ can be interpreted as defining a geometric object, namely a \textit{pleated plane} in $\mathbb{H}^3$. This pleated plane is a $\rho$-equivariant  map 
\begin{equation}\label{pplane}
\Psi: \widetilde{S} \to \mathbb{H}^3
\end{equation}
defined on the universal cover of our surface $S_{g,k}$, such that each ideal triangle $\Delta \in \tilde{T}$ with, say, ideal vertices $p_1,p_2,p_3$, maps to the totally-geodesic ideal triangle in $\mathbb{H}^3$ with ideal vertices $\beta(p_1),\beta(p_2),\beta(p_3)$. Here we assume that $\Psi$ preserves orientation, so that the image is an oriented surface in $\mathbb{H}^3$ that is ``piecewise totally-geodesic".

\subsection{Constructing a projective structure}\label{consprojstrucnd} Given a non-degenerate framed representation $\hat{\rho}= (\rho,\beta)$ as the beginning of the section, we now describe how to construct a projective structure in $\mathcal{P}_g(k)$ with monodromy $\rho$. This is exactly as in \cite{GupMon1}, however here we provide a  condensed and simplified discussion. In particular, here shall avoid the intermediate steps of ``straightening" the pleated plane and then "grafting". 

\smallskip

\noindent The main idea is that a pleated plane $\Psi$ constructed in the previous subsection  also defines a projective structure by considering its  "shadow"  at the boundary at infinity of $\mathbb{H}^3$. More precisely, on each totally-geodesic ideal triangle $\Psi(\Delta)$ we can consider the hyperbolic Gauss map $\mathcal{G}_{\Delta}$, in the normal direction consistent with the orientation of $\Psi$. On each ideal triangle $\Delta \in \tilde{T}$, define the map $\Psi^0_\infty = \mathcal{G}_\Delta \circ \Psi$. Note that the image of each ideal triangle is a triangle on $\cp$ with sides that are circular arcs that form ``cusps" at the three vertices. 

\noindent Note that this defines a map $\Psi^0_\infty: \widetilde{S} \to \cp$ that is $\rho$-equivariant. However it may not be continuous: suppose $\Delta_l,\Delta_r$ are two adjacent triangles, such that $\Psi(\Delta_l)$ and $\Psi(\Delta_r)$ lie in totally geodesic planes that intersect at an angle $\alpha \in (0, 2\pi)$ along a common geodesic line $l$ (which is the $\Psi$-image of the common edge $\tilde{e}$  between $\Delta_l$ and $\Delta_r$).  In that case the image of $\tilde{e}$ under $\mathcal{G}_{\Delta_l}$ and $\mathcal{G}_{\Delta_r}$ are arcs $\alpha_l,\alpha_r$  of great circles on $\cp$ with a common pair of endpoints,  that differ by an elliptic rotation of an angle $\alpha$.  We shall call a region in $\cp$  bounded by such a pair of arcs a "lune".

 \begin{figure}
  \centering
  \includegraphics[scale=0.34]{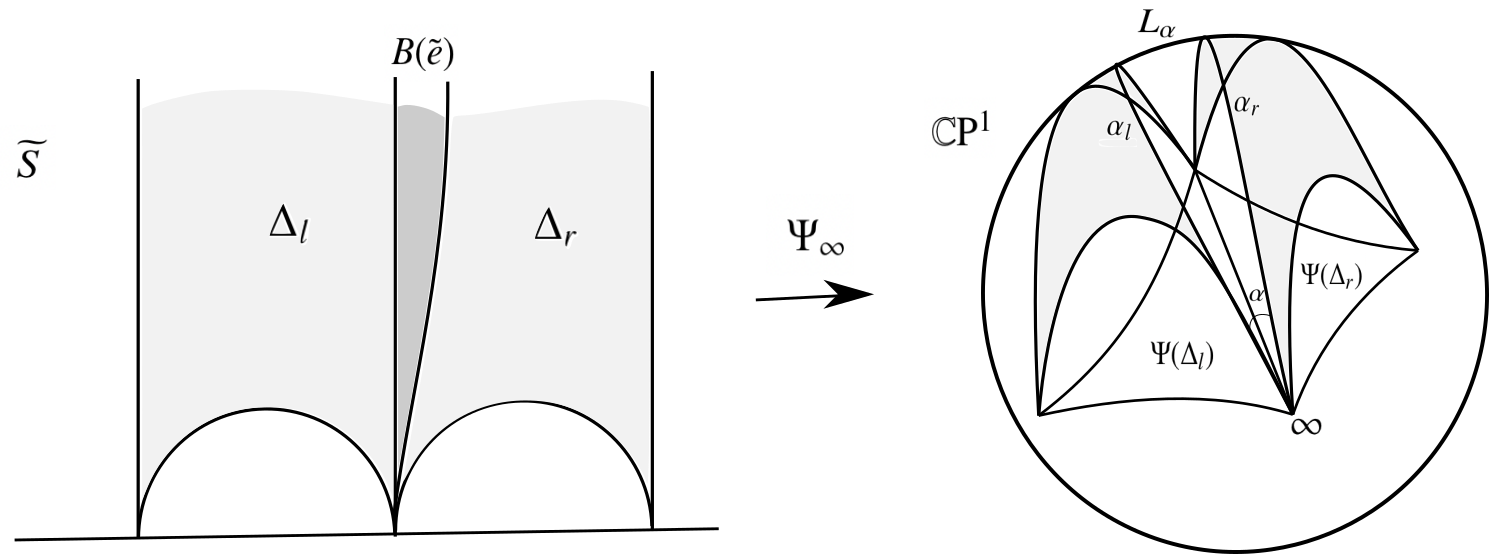}
  \caption{The map $\Psi_\infty$ maps the grafted region $B(\tilde{e})$ to the ``lune" $L_\alpha$ on $\cp$ bounded by the circular arcs $\alpha_l$ and $\alpha_r$. The shaded regions on the right are the images of ${\Delta_l}$ and ${\Delta_r}$ under $\Psi^0_\infty$ (see \S3.3).}
\end{figure}

\noindent We can, however modify the map $\Psi^0_\infty$ to obtain a \textit{continuous} map 
 \begin{equation}\label{psi-inf}
     \Psi_\infty: \widetilde{S} \to \cp
 \end{equation}
 as follows: for each pair of adjacent triangles $\Delta_l,\Delta_r$ in the domain $\widetilde{S}$ as above, cut along the common edge $\tilde{e}$, and glue in a bigon $B(\tilde{e})$ such that the sides of the bigon are identified with the resulting two sides $\tilde{e}_\pm$. Define a map from $B(\tilde{e})$ to $\cp$ such that the two sides map to the arcs $\alpha_l,\alpha_r$ described above, and the image is the lune  bounded by $\alpha_l \cup \alpha_r$, such that the resulting map from $\Delta_l \cup B(\tilde{e}) \cup \Delta_r $ is a smooth orientation-preserving immersion to $\cp$.

 \noindent We can define the maps on the "grafted" bigons in a $\Pi$-equivariant manner, so that the final map $\Psi_\infty$ 
 as in \eqref{psi-inf} is $\rho$-equivariant, and a smooth immersion. Hence, it defines a projective structure $P$ on the surface $S_{g,k}$ with monodromy $\rho$.

\subsection{Schwarzian derivative at the punctures}\label{sdpunct}
It remains to show that the projective structure $P$ we just constructed on $S_{g,k}$ in fact lies in the space $\mathcal{P}_g(k)$, that is, the Schwarzian derivative of the developing map with respect to a suitable reference projective structure, has a pole of order at most two. 

\noindent To see this, we uniformize the underlying Riemann surface structure on $S_{g,k}$ to obtain a hyperbolic metric of finite area, such that each puncture is a cusp. This will serve as a reference projective structure.  

\noindent Let $\mathbb{D}^\ast = \{ 0 <\lvert w\rvert < 1\} $ be a conformal punctured-disk neighborhood of a puncture.
Lifting to the universal cover $\widetilde{S}$, a neighborhood of an ideal point $p \in F_\infty$ would look exactly like a neighborhood of $\infty$ in the upper-half space model of $\mathbb{H}^2$, where the edges of the triangulation $\widetilde{T}$ are the vertical geodesics, and the deck-translation corresponding to the parabolic element around the cusp is  a (positive) translation. Moreover, can choose this conformal identification with the model $\mathbb{H}^2$ such that the translation is $z\mapsto z+1$, namely  $\mathbb{H}^2/\langle z\mapsto z+1 \rangle \cong \mathbb{D}^\ast$ via the map $z\mapsto w := e^{2\pi i z}$. If there are $n$ geodesic sides of the triangulation $T$ asymptotic to that cusp, then a fundamental domain $\Delta$ of the action will comprise  $n$ ideal triangles  $\Delta_1,\Delta_2,\ldots, \Delta_n$ together with $n$ ``bigons"  $B(\tilde{e}_1),B(\tilde{e}_1),\ldots B(\tilde{e}_n)$ that were grafting in. Here, $\tilde{e}_i$ is the ``right-hand" edge  of 
$\Delta_i$ for each $1\leq i\leq n$.

\noindent From our definition of the map $\Psi_\infty$ above, the images of $B(\tilde{e}_1),B(\tilde{e}_2),\ldots, B(\tilde{e}_n)$ are each a lune in $\cp$ with a common endpoint $\beta(p)$, and the images of $\Delta_1,\Delta_2,\ldots, \Delta_n$ in a neighborhood of $\infty$ are regions bounded by circular arcs that form a ``cusp" (of angle zero) at $\beta(p)$.  Each successive region shares a common circular arc with the preceding one. Hence the union of their images near $\beta(p)$ looks like a region bounded by two circular arcs that intersect at some angle $\alpha \in \mathbb{R}^+$ (if $\alpha>2\pi$ then the lune is immersed in $\cp$).  Here, the angle $\alpha$ is the sum of the angles at $\beta(p)$ of the lunes that are the images of the grafted bigons. 

\noindent The rest of the argument is exactly as in the proof of Proposition 3.5 in \cite{GupMon1}. 
Consider first the case when the total ``bending angle" $\alpha$ around $\beta(p)$ is positive. We can assume, by post-composing with an appropriate M\"{o}bius map, that $\beta(p) =\infty \in \cp$. 
The conformal developing map from a neighborhood of $\infty$ in $\Delta$ then maps to a neighborhood of $\infty$ of a  lune $L_\alpha \subset \cp$ that has vertices at $0,\infty$. Such a conformal map has the asymptotic form $f(z) = e^{-i\alpha z}$ for $\lvert z \rvert \gg 1$, and hence on the punctured disk $\mathbb{D}^\ast$, has the expression $\tilde{f}(w) = w^{-\alpha/2\pi} $. A computation of the Schwarzian derivative using \eqref{schw-deriv} then yields $S(\tilde{f})$ has a pole of order $2$ at the puncture, as desired. 

\noindent In the case when the total bending angle $\alpha=0$, the conformal developing map is the identity map $f(z)=z$, and the Schwarzian derivative thus yields the constant quadratic differential $dz^2$ on $\Delta$ that has the expression $-\frac{1}{4\pi^2}w^{-2} dw^2$ on $\mathbb{D}^\ast$, once again with a pole of order two.

\medskip

\noindent The discussion in this section thus proves:

\begin{prop}\label{prop:nondegen}  Let $\Pi$ be the fundamental group of a surface $S_{g,k}$ of genus $g$ and $k\geq 1$ punctures, where $2-2g-k<0$.  Any non-degenerate representation $\rho:\Pi \to \pslc$ arises as the monodromy of a $\cp$-structure in $\mathcal{P}_g(k)$.
\end{prop}

\section{Proof of Theorem \ref{thm2}: Affine holonomy with at least two punctures}\label{proof2}

\noindent In this section we start to deal with the complementary case when $\rho:\pi_1(S_{g,k}) \to \pslc$ is a degenerate representation (see Definition \ref{degen}). Here we shall consider surfaces $S_{g,k}$ with at least two punctures, that is $k\ge2$.  Our main goal in this section is to prove our Theorem \ref{thm2}, that is the case (ii) of Theorem \ref{thm1} for affine representations when $S_{g,k}$ has at least two punctures.

\begin{thm}
Let $\Pi$ be the fundamental group of $S_{g,k}$ as in Theorem \ref{thm1}, such that the number of punctures $k\geq 2$. Then any non-trivial representation $\rho:\Pi \to \text{Aff}({\C})$ arises as the monodromy of a complex affine structure on $S_{g,k}$. If $k\ge3$, then every representation is realizable as the monodromy of a complex affine structure on $S_{g,k}$.
\end{thm}
 

\medskip

\noindent In \S4.1 we shall handle the case when $\rho$ is the trivial representation, in \S4.2 the case when $\rho(\Pi)$ has a single global fixed point $p\in \cp$ and the monodromy of each element is a translation. In \S4.3 we deal with the more general case when the monodromy of each element is an \textit{affine} map. This would include also include the case when $\rho(\Pi)$ fixes a set $\{p,q\} \subset \cp$, and is "co-axial".  Note that the necessity of the presence of one apparent singularity follows from the work of Allegretti-Bridgeland (see Theorem \ref{ab61}).

\subsection{Trivial representation}\label{sstriv}

We prove the first of the exceptional cases mentioned in Theorem \ref{thm1}. To state the first result, we introduce the following terminology:

\begin{defn}[Handles, handle-generators]\label{handle} On a marked surface $S_{g,k}$ of some positive genus $g>0$, a \textit{handle} is an embedded subsurface $\Sigma$ that is homeomorphic to $S_{1,1}$, and a \textit{handle-generator} is a simple closed curve that is one of the generators of $H_1(\Sigma, \mathbb{Z})$. A \textit{pair of handle-generators} for a handle will refer to a pair of simple closed curves $\{\alpha, \beta\}$ that generate $H_1(\Sigma, \mathbb{Z})$; in particular, $\alpha$ and $\beta$ intersect once. 
\end{defn}

\begin{lem}\label{trivex} Suppose $S_{g,k}$ is a surface where $g>0$ and $k=1$ or $2$. Then the monodromy  $\rho$ of any $\cp$-structure in $\mathcal{P}_g(k)$ is non-trivial. Moreover, there is at least one loop $\gamma_0$ that is a handle-generator, such that $\rho(\gamma_0)$ is non-trivial. 
\end{lem}
\begin{proof} Suppose there is a $\cp$-structure on $S_{g,k}$ (where $k=1$ or $2$) such that the monodromy representation is trivial, i.e. $\rho(\gamma) = \text{Id}$ for all $\gamma \in \Pi$.  In particular, any puncture is an apparent singularity, and since this projective structure corresponds to the Schwarzian equation with regular singularities, is either a regular point or branch-point.  

\noindent Since the monodromy is trivial, the developing map (defined on the universal cover) in fact descends to a well-defined map $\overline{dev}: S_{g,k} \to \cp$ (where recall $k=1$ or $2$). Since the developing map is an immersion, this map can be thought of as a covering map from the closed surface $S_g$ to $\cp$ that is possibly branched at one or two points. That is, such a map has at most two critical values on $\cp$.  Since $g>0$, it follows from the Riemann-Hurwitz formula that there cannot be such a branched covering.

\noindent If we only assume that $\rho(\gamma) = \text{Id}$ whenever $\gamma$ is one of the $2g$ loops that are the generators of the handles, then we can in fact show that the representation $\rho$ is trivial: this is immediate if $k=1$, and if $k=2$, note that we already know that there is one apparent singularity; the monodromy around the other singularity is then the product of commutators of the handle-generators, and hence also trivial. 
\end{proof}

\noindent We now show that in all other cases, one can construct a projective structure with trivial monodromy:

\begin{lem}\label{triv} Let $S_{g,k}$ be a surface where $k\geq 3$. Then there is a $\cp$-structure in $\mathcal{P}_g(k)$ whose monodromy representation is trivial. 
\end{lem}
\begin{proof}
\noindent In the case that $g=0$, it is easy matter to define a projective structure with trivial monodromy on $S_{0,k}$, in fact it is sufficient to consider $\cp$ with $k$ punctures. Otherwise, we use the fact that the Hurwitz problem of the existence of branched coverings with prescribed branching data, that is solved for $g>0$ -- see Proposition 3.3. of \cite{E-K-S} and also \cite{Husemoller}. 

\noindent In particular, that implies that there is a branched covering $\Pi:S_g \to \cp$ of degree $2g+1$ that is branched over $0,1,\infty \in \cp$, each of ramification order $2g+1$ and has exactly three critical points, say $p,q,r$ on $S_g$, such that $p = \Pi^{-1}(0), q = \Pi^{-1}(1)$ and $r = \Pi^{-1}(\infty)$, see Figure \ref{branchcover}.

\begin{figure}[h]
  \centering
  \includegraphics[scale=0.42]{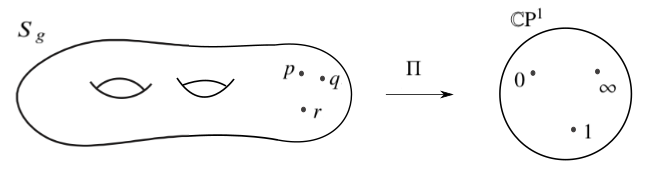}
  \caption{The map $\Pi$  in Lemma \ref{triv} is a branched cover over the sphere with exactly three critical points and three branch-points. }
  \label{branchcover}
\end{figure}

\noindent Then, the standard projective structure on $\cp \setminus \{0,1,\infty\}$ naturally pulls back to a projective structure on $S_g \setminus \{p,q,r\}$ that has trivial holonomy. Note that one can also equip  $\cp \setminus \{0,1,\infty\} = \mathbb{C} \setminus \{0,1\}$  with the standard  translation structure (see \S4.2) and pull it back to $S_g \setminus \{p,q,r\}$; the induced abelian differential $\omega$ extends to the closed surface $S_g$ and has two zeroes at $p$ and $q$, and a pole at $r$. 
\end{proof}

\subsection{Translation structures}\label{sstrans}

In this section we shall assume that  $\rho:\Pi \to \C$ is a non-trivial representation, where 
\begin{equation*}
 \C \cong \{\ z\mapsto z + c \ \vert\  c \in \C \ \}
\end{equation*}
is the subgroup of $\pslc$ comprising translations.
Note that since this subgroup is abelian, $\rho$ factors through the first homology group $\Gamma_{g,k} := H_1(S_{g,k}, \mathbb{Z})$, and can be thought of as a period character (or simply character) $\chi_{g,k}:\Gamma_{g,k} \to \C$. 

\medskip

\noindent Note that the residue theorem implies that:

\begin{lem} For any homomorphism  $\chi_{g,k}:\Gamma_{g,k} \to \C$ as above, then sum of the values around the peripheral curves is zero, namely 
\begin{equation}\label{sumC}
\sum\limits_{i=1}^k \chi_{g,k}(\gamma_i) =0
\end{equation}
where $\gamma_i$ is the simple closed curve around the $i$-th puncture.
\end{lem}
\medskip

\textit{Remark.} We shall say that a "puncture has trivial monodromy" if $\chi_{g,k}(\gamma) =0$ where $\gamma$ is the loop around that puncture. 

\bigskip 

\noindent In the case that $g>0$ and $k>1$  we shall also divide the character into $\chi_{g,1}$ and $\chi_{0, k+1}$ as follows:
choose a separating loop $\gamma$ that divides the surface $S_{g,k}$ into a subsurface $S_{g,1}$ containing all the handles and one boundary component, namely $\gamma$, and a subsurface $S_{0,k+1}$ that is topologically a punctured sphere, where one of the punctures is actually a boundary component $\gamma$, see Figure \ref{splitsurface}.   

\smallskip

\noindent By restricting to $S_{g,1}$ and $S_{0,k+1}$, the character $\chi_{g,k}$ determines homomorphisms 
\begin{equation}\label{chig1}
\chi_{g,1}:\Gamma_{g,1} \to \C
\end{equation}
and 
\begin{equation}\label{chi0n}
\chi_{0,k+1}:\Gamma_{0,k+1} \to \C
\end{equation}
 respectively, where $\Gamma_{g,1}$ and $\Gamma_{0,n+1}$ are the homology groups of the subsurfaces $S_{g,1}$ and $S_{0,k+1}$ respectively. Note that in either case,  the boundary loop $\gamma$ is trivial in homology.

\begin{figure}[h]
  \centering
  \includegraphics[scale=0.45]{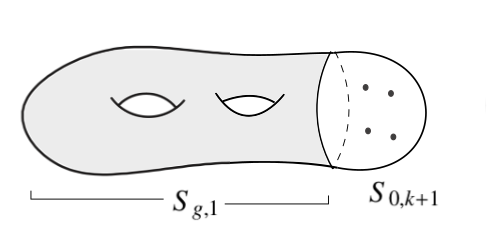}
  \caption{The surface $S_{g,k}$ can be divided into a subsurface homeomorphic to $S_{g,1}$ (shown shaded) and its complement, homeomorphic to $S_{0,k+1}$. }
  \label{splitsurface}
\end{figure}

\medskip

\noindent We shall need the following observation:

\begin{lem}\label{cbas} Let $\chi_{g,k}:\Gamma_{g,k} \to \C$ be a character and let its associated homomorphism  $\chi_{g,1}$ as in \eqref{chig1} be a non-trivial representation. Then there is a change of homology basis $A\in \mathrm{Sp}(2g, \mathbb{Z})$ such that $\chi^\prime_{g,k} = \chi_{g,k} \circ A$ has an associated restriction $\chi^\prime_{g,1}$ such that $\chi^\prime_{g,1}(\gamma) \neq 0 $ for each handle-generator $\gamma$ on $S_{g,1}$. 
\end{lem}

\begin{proof}
Since $\chi_{g,1}$ is non-trivial, there is some simple closed curve $\gamma_0$ that is a handle-generator on $S_{g,1}$, such that $\chi_{g,1}(\gamma_0) \neq 0 $. Let $\{\gamma_0, \gamma_1\}$ be the handle-generators of such a handle. Now suppose $\chi_{g,1}(\gamma_1) =0$, we can replace the handle-generator $\gamma_1$ with the curve $\gamma_1^\prime$ which is obtained by Dehn-twisting $\gamma_1$ around $\gamma_0$ (see Figure \ref{changeofbasis2}). In homology $\gamma_1^\prime = \gamma_0 + \gamma_1$, and hence we now have $\chi_{g,1}(\gamma_1) \neq 0$.  Thus, we have a basis of homology such that at least one of the handles has both its handle-generators with non-trivial monodromy.

\begin{figure}[h]
  \centering
  \includegraphics[scale=0.45]{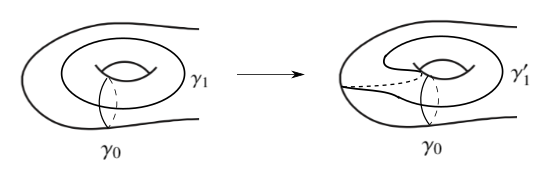}
  \caption{A Dehn-twist changes the pair of handle-generators $\{\gamma_0, \gamma_1\}$ to the pair $\{\gamma_0, \gamma_1^\prime\}$ where $\gamma^\prime_1 = \gamma_0 + \gamma_1$ in homology. }
  \label{changeofbasis2}
\end{figure}

\noindent Now suppose there is a handle with handle-generators $\{\alpha,\beta\}$ such that $\chi_{g,1}(\alpha) = \chi_{g,1}(\beta) =0$, then consider the element of $ \mathrm{Sp}(2g, \mathbb{Z})$ that changes the homology  basis  $\{\gamma_0, \gamma_1, \alpha,\beta\}$  of the two handles to the  new homology basis $\{\gamma_0 - \alpha, \gamma_1 + m\cdot \alpha, \alpha , \beta+m\cdot \gamma_0 + \gamma_1\}$, for an integer $m$, leaving the other generators unchanged. (See \cite{Martens} for this, and related elements of $ \mathrm{Sp}(2g, \mathbb{Z})$.)  The handle thus acquires a new pair of generators $\{\alpha,\beta+m\cdot \gamma_0 + \gamma_1\}$, where one of the new handle-generators has non-trivial monodromy for some $m$, \textit{i.e.}  $\chi_{g,1}(\beta+m\cdot \gamma_0 + \gamma_1) \neq 0$ for some $m$, since $\chi_{g,1}(\gamma_0)\neq 0$. The holonomy of the other handle remains unchanged since $\chi_{g,1}(\alpha)=0$.  By a further change of basis as in the previous paragraph, by Dehn-twisting $\alpha$ around the curve representing $\beta+ \gamma_0$, we can ensure there is a pair of generators of the handle which are \textit{both} non-trivial in monodromy. 

\noindent Applying either of these changes of bases to each of the handles, we can ensure that we obtain a homology basis $\{\alpha_1,\beta_1,\alpha_2,\beta_2,\ldots \alpha_g,\beta_g\}$ via a change of basis matrix $A \in  \mathrm{Sp}(2g, \mathbb{Z})$ such that the homomorphism $\chi^\prime_{g,k} = \chi_{g,k} \circ A$ satisfies  $\chi^\prime_{g,1}(\alpha_i) \neq 0 $  and $\chi^\prime_{g,1}(\beta_i) \neq 0 $ for each $1\leq i\leq g$. 
\end{proof} 

\textit{Remark.} Recall from the discussion in \S\ref{mps} that we only require to realize a monodromy representation up to the action of the mapping class group. Hence, since the change of basis above is realized by a mapping class, we can assume that the handle-generators are each non-trivial once we know that $\chi_{g,1}$ is non-trivial.

\medskip

\noindent In this section, our strategy would be to define a \textit{translation structure} on $S_{g,k}$ with the prescribed holonomy that lies in $\C$. Recall that it is a special case of a complex projective structure, comprising an atlas of charts to $\C$ such that the transition maps are translations. Note that such an atlas equips the resulting surface with a complex structure. and a translation structure is then  equivalent to  a non-vanishing holomorphic vector field on the Riemann surface, or equivalently, a non-vanishing holomorphic $1$-form $\omega$. Note that the punctures could be regular points, or zeroes or poles of $\omega$; a related problem when we require some punctures to be poles, and others to be  zeroes of $\omega$, with prescribed orders,  is dealt with in \cite{CFG}.

\noindent Such a translation structure can be defined by gluing sides of a polygon, as we now describe: 

\begin{defn}[Translation surface]\label{psurf}  A \textit{translation surface} is obtained by starting with a collection of (possibly non-compact) polygons in $\C$ bounded by straight lines and/or straight-line segments and/or rays, and identifying such sides pairwise by translations.  The resulting surface $\Sigma$  thus acquires a Euclidean metric, with possible cone-singularities (with cone-angles an integer multiple of $2\pi$)at points arising from the by identifications of the vertices of the polygons. In the complement of such cone-points, we then obtain a translation structure as defined above. The standard differential $dz$ on the polygons descend to  the holomorphic $1$-form $\omega$ on the surface, and the zeroes of $\omega$ are precisely at the cone-points, which are branch-points of the translation surface. The periods of $\omega$ define a representation $\chi:H_1(\Sigma, \mathbb{Z}) \to \C$ that is the \textit{holonomy} of the translation surface $\Sigma$; note that if the branch-points are removed from $\Sigma$, then each additional puncture has trivial monodromy around it.
\end{defn} 

\textit{Remark.} In the case that the polygons are non-compact, the translation structure will have at least one puncture ``at infinity", where the abelian differential $\omega$ has a pole. The order of such a pole can be determined from the flat geometry of the corresponding end: if the end is cylindrical, then the pole has order one, if it is a planar end (i.e.\ like that of $\C$) then the pole is of order two, and a pole of order $n>2$ has an end which is isometric to an $(n-1)$-fold cover of a planar end, branched at $\infty$. \\

\noindent The following construction allows to define a new translation surface by gluing together two translation surfaces with poles.

\begin{figure}
  \centering
  \includegraphics[scale=0.43]{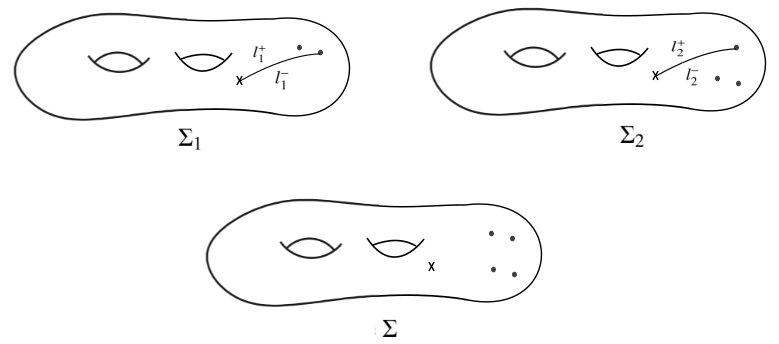}
  \caption{Gluing $\Sigma_1$ and $\Sigma_2$ along rays  results in a new surface $\Sigma$, see Definition \ref{glue-ray}. }
\end{figure}

\begin{defn}[Gluing along a ray]\label{glue-ray}  Suppose $\Sigma_1$ and $\Sigma_2$ are two translation surfaces, each with at least one pole. Let  $l_i \subset \Sigma_i$ for $i=1,2$ be an embedded straight-line ray that starts from a cone-singularity (or a regular point) and ends in a pole. Assume that $l_1$ and $l_2$  develop onto infinite rays on $\C$ that are parallel.  Then we can define a translation surface $\Sigma$ as follows: slit each ray $l_i$ and denote the resulting sides by $l_i^+$ and $l_i^-$; then identify $l_1^+$ with $l_2^-$ and $l_1^-$ and $l_2^+$ by a translation. If $\Sigma_i$ is homeomorphic to $S_{g_i,k_i}$ for $i=1,2$, then the resulting surface $\Sigma$ is homeomorphic to $S_{g_1+g_2, k_1+k_2-1}$. Note that the starting points of the rays are identified to a branch-point on $\Sigma$ with a cone-angle that is the sum of the corresponding angles on $\Sigma_1$ and $\Sigma_2$, and the other endpoints (at infinity) are identified to a higher order pole.  See Figure 5.
\end{defn}

\noindent We finally introduce the notion of \textit{algebraic volume} of a representation $\chi$ as follows:

\begin{defn}[Algebraic volume]\label{volumets} 
Let $\Sigma$ be a translation structure on once-punctured surface $S_{g,1}$ with holonomy $\chi:H_1(\Sigma, \mathbb{Z})\longrightarrow \C$. The \textit{algebraic volume} of $\chi$ is defined as the quantity
\begin{equation}
    \text{Vol}(\chi)=\sum_{i=1}^g \Im\big(\,\overline{\chi(\alpha_i)}\,\chi(\beta_i)\,\big)
\end{equation} where $\{\alpha_1,\beta_1,\dots,\alpha_g,\beta_g\}$ is any symplectic basis of $H_1(\Sigma,\mathbb{Z})$, and $\Im(c)$ denotes the imaginary part of a complex number $c$. If $\Sigma$ is a translation structure on a generic surface $S_{g,\,k}$ then we define the algebraic volume of $\chi$ as the algebraic volume of the sub-representation $\chi_{g,1}$ introduced in \eqref{chig1}. We have emphasized above the term \textit{algebraic} in order to distinguish this notion of volume from its geometric counterpart. In what follows we do not need to consider the geometric volume and henceforth for simplicity we abridge the notation to just ``volume".
\end{defn}

\noindent In our construction, we shall also need:

\begin{defn}[Volume of a quadrilateral]
Let $a$ and $b$ be two complex numbers and let $\mathcal{Q}$ be the (possibly degenerate) quadrilateral spanned by the corresponding vectors. The volume of $\mathcal{Q}$ is then defined as
\begin{equation}
  \text{Vol}(\mathcal{Q})=\text{Vol}(a,\,b)= \Im\big(\overline{a}\,b\big).
\end{equation} In particular, the volume is null if and only $a=\lambda b$ for some $\lambda\in\mathbb{R}$.
\end{defn}


\noindent We begin with the case that $k=2$, i.e.\ there are exactly two punctures; note that our assumption of negative Euler characteristic implies that $g>0$. We show:

\begin{prop}\label{trans1} Let $g>0$. Any non-trivial representation $\chi_{g,2}:\Gamma_{g,2} \to \C$ with at least one puncture with trivial monodromy, appears as the holonomy of some translation structure on $S_{g,2}$, where one of the punctures corresponds to a zero of the abelian differential. 
\end{prop}
\begin{proof}
Note that when there are exactly two punctures, and one of them has trivial monodromy, then from \eqref{sumC} it then follows that the other puncture also has trivial monodromy. Since $\chi_{g,2}$ is non-trivial, the associated representation $\chi_{g,1}$ as in \eqref{chig1} is non-trivial. By Lemma \ref{cbas} we can assume that each handle-generator maps to a non-zero complex number. 

\smallskip

\noindent Now, construct a translation surface $\Sigma$ homeomorphic to $S_{g,1}$ and having holonomy $\chi:=\chi_{g,1}$  as follows: First, we consider the case when $g=1$; let the pair of handle-generators be $\{\alpha,\beta\}$, and let $A,B:\C\to \C$ be the translations $z\mapsto z+ \chi(\alpha)$ and $z\mapsto z+ \chi(\beta)$ respectively. Note that $A,B$ commute, and neither is the identity map since the monodromy around each handle-generator is non-trivial. 

\noindent Define four closed directed straight-line segments $\{e_1,e_2,e_3,e_4\}$ in $\C$ as follows: let $p\in \C$ be any point and let $q := AB(p) = BA(p)$. Then define 
\begin{equation*}
    e_1 := \overline{A(p)\,q},\,\,e_2 := \overline{p\,A(p)},\,\,e_3 := \overline{p\,B(p)},\,\,e_4 := \overline{B(p)\,q}.
\end{equation*}
\noindent Note that $A(e_3) = e_1$ and  $B(e_2) = e_4$. If $\overline{e_i}$ denotes the line-segment $e_i$ with its direction reversed, then we see that $L := \overline{e_1}\cup \overline{e_2} \cup {e_3} \cup {e_4}$ is a closed  directed loop based at $q$. Notice that $L$ bounds a (possibly degenerate) quadrilateral. 

\smallskip

\noindent Now, we consider four embedded polygons $R_1,R_2,R_3,R_4$ in $\C$, where
\begin{itemize}
\item[\textit{(i)}] each $R_i$ is bounded by two infinite rays $r_i$ and $r_{i+1}$, and $e_i$ (where the indices $1\leq i\leq 4$  are cyclically ordered, such that $r_5$ is actually $r_1$), and 

\item[\textit{(ii)}] $R_1$ and $R_3$ lie on opposite sides of $e_1$ and $e_3$ respectively, and $R_2$ and $R_4$ lie on opposite sides of $e_2$ and $e_4$ respectively.
\end{itemize}

\begin{figure}
  \centering
  \includegraphics[scale=0.4]{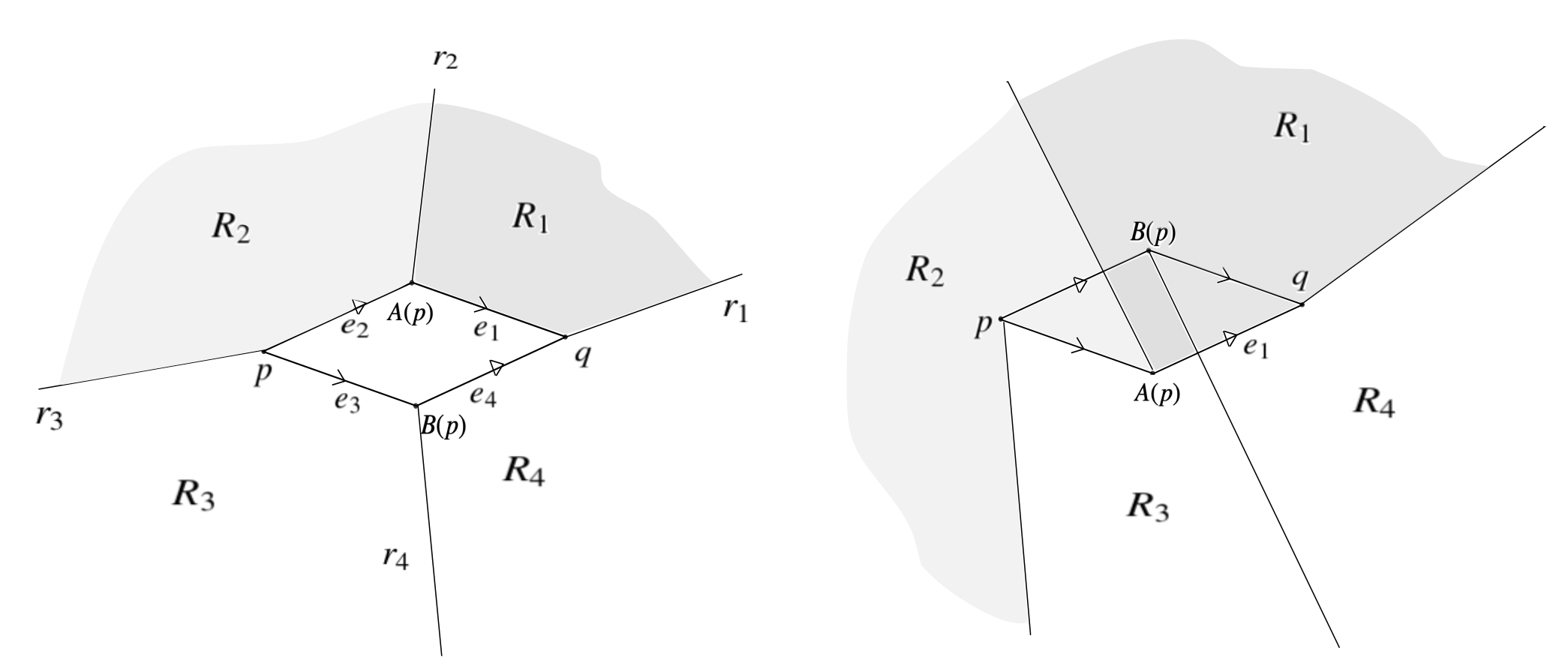}
  \caption{The identifications described in Proposition \ref{trans1} results in a translation surface homeomorphic to $S_{1,1,}$, a torus with a puncture (at infinity). The regions $R_1,R_2, R_3,R_4$ are chosen to lie on the left or right side of the closed polygonal curve $L = \overline{e_1}\cup \overline{e_2} \cup {e_3} \cup {e_4}$ depending on whether the volume of the handle is non-positive or positive (figures on left and right respectively). }
  \label{leftrightchoice} 
\end{figure}

\noindent Since $R_1$ and $R_2$ share the boundary ray $r_2$, either both lie on the left sides of the directed edges $e_1$ and $e_2$ respectively, or both lie on the right sides of those edges. Similarly, the pair of adjacent regions $R_3$ and $R_4$ either are both on the right sides of $e_3$ and $e_4$, or are both on their left sides. Thus, the choice of which side (left or right) of $e_1$ that $R_1$ should lie, determines which side of $e_i$ the region $R_i$ lies, for the remaining  $i=2,3,4$.  This choice is determined by the final requirement: 

\begin{itemize}
\item[\textit{(iii)}] $R_1$ lies on the left-hand side of $e_1$.
\end{itemize}

\smallskip

\noindent Note that if $\text{Vol}\big(\chi(\alpha), \chi(\beta)\big)=0$, then the four segments $e_1,e_2,e_3,e_4$ are collinear, and can be thought as lying on the boundary of a slit, the exterior of which is the union of regions $R_1\cup R_2\cup R_3 \cup R_4$. 

\smallskip

\noindent From requirement \textit{(i)} above we know  $R_i$ and $R_{i+1}$ already share a boundary ray $r_{i+1}$ (for the cyclically ordered indices $1\leq i\leq 4$),  and it follows that $R_1\cup R_2\cup R_3 \cup R_4$ is topologically a punctured disk immersed in $\C$, with boundary $L$ and a puncture at $\infty$. 

\noindent The translation surface $\Sigma$ is obtained by identifying the remaining boundary sides (the segments along the loop $L$) as follows: $e_3$ is identified with $e_1$ via the translation $A$, and $e_2$ is identified with $e_4$ via the translation $B$.  Note that requirement \textit{(ii)} above ensures that we obtain a surface. It is easy to see that $\Sigma$ is homeomorphic to the punctured torus $S_{1,1}$  (with the puncture at $\infty$), and the holonomy equals $\chi$.  Note that the holonomy is with respect to a fixed pair of oriented loops that are the handle-generators on $S_{1,1}$; the choice in \textit{(iii)} results in the desired orientation of these loops on $\Sigma$. 

\smallskip

\noindent The puncture at $\infty$ has trivial monodromy, and the induced abelian differential $\omega$ has a pole of order two at that point. There is one branch-point on $\Sigma$ (or equivalently, one zero of $\omega$) -- namely,  the one obtained from the endpoints of the segments $e_1,e_2,e_3,e_4$ after the identifications. Thus if we remove the branch-point, we obtain a surface homeomorphic to $S_{1,2}$, equipped with a translation structure, with holonomy $\chi$ where both punctures with trivial monodromy, as desired.

\smallskip

\noindent In the case that $g\ge2$, we shall proceed as follows. Let $\{\alpha_1,\beta_1,\dots,\alpha_g,\beta_g,\gamma_1,\gamma_2\}$ be a generating set of $\pi_1(S_{g,2})$ where $\{\alpha_i,\beta_i\}$ is a pair of handle generators for the $i$-th handle, where $1\le i\le g$. Let $A_i$ and $B_i$ denote the images of $\alpha_i$ and $\beta_i$ via $\chi_{g,2}$ respectively. We may assume that any handle generator is non-trivial as a consequence of Lemma \ref{cbas}.
For any point $p\in\C$, each pair $\{A_i,\,B_i\}$ determines a (possibly degenerate) quadrilateral $\mathcal{Q}_i\subset \C$. We can order these quadrilaterals in such a way the volume of $\mathcal{Q}_i$ is positive for $1\le i\le h\le g$ and non-positive for the remaining (possibly none) $g-h$ handles. Note that this notion does not depend on the choice of the base-point. In what follows we shall place the quadrilaterals on the complex plane according to the following rule:
\begin{quote}
    \textit{The right-most vertex of $\mathcal{Q}_i$ is identified with the left-most vertex of $\mathcal{Q}_{i+1}$. If $\mathcal{Q}_i$ has more than one right-most vertices, i.e. some  edges are vertical, then the top-most is chosen. If $\mathcal{Q}_{i+1}$ has more than one left-most vertices then the bottom-most is chosen. }
\end{quote}
Suppose there are $h \le g$ positive handles, and let $q$ the vertex that $\mathcal{Q}_h$ and $\mathcal{Q}_{h+1}$ have in common. Notice that such a point is unique because of the rule above. Let $\ell$ be a straight line passing through $q$ and such that it is not parallel to any edge of any quadrilateral. We introduce an orientation on $\ell$ in such a way the handles with non-positive volume are on the right of $\ell$ and the handles with positive volume are on the left of $\ell$. Moreover, according to this orientation, the point $q\in\ell$ divides the straight line in two rays, the upper one $\ell^+$ and the lower one $\ell^-$.
The right-hand side of $\ell$ is a halfplane $H$ in $\C$ containing $g-h$ quadrilaterals and let $R_0$ be the complement of $\mathcal{Q}_{h+1}\cup\cdots\cup\mathcal{Q}_g$ in $H$. If there are no handles with non-positive volume then $R_0$ is just the half-plane $H$. Similarly, if there are no handles with positive volume then the complement of $R_0$ in $\C$ is a half-plane $\C \setminus H$. 

\noindent Suppose $h\ge1$, on the left-hand side of $\ell$ we consider a chain of $h$ embedded quadrilaterals, obeying the rule above such that the $i$-th quadrilateral $Q_i$ has edges $\{e^i_1,e^i_2,e^i_3,e^i_4\}$ defined by 
\begin{equation}
e^i_1 := \overline{A_i(p_i)\,p_{i+1}} \quad e^i_2 := \overline{p_i\,A_i(p_i)} \quad e^i_3 := \overline{p_i\, B_i(p_i)}\quad e^i_4 := \overline{B_i(p_i)\,p_{i+1}}.
\end{equation}

\noindent where $p_i$ is the unique point such that $p_{i+1}=A_iB_i(p_i)$, where $p_{h+1}=q$ by definition. 

\smallskip 

\begin{figure}\label{fig:chain} 
  \centering
  \includegraphics[scale=0.48]{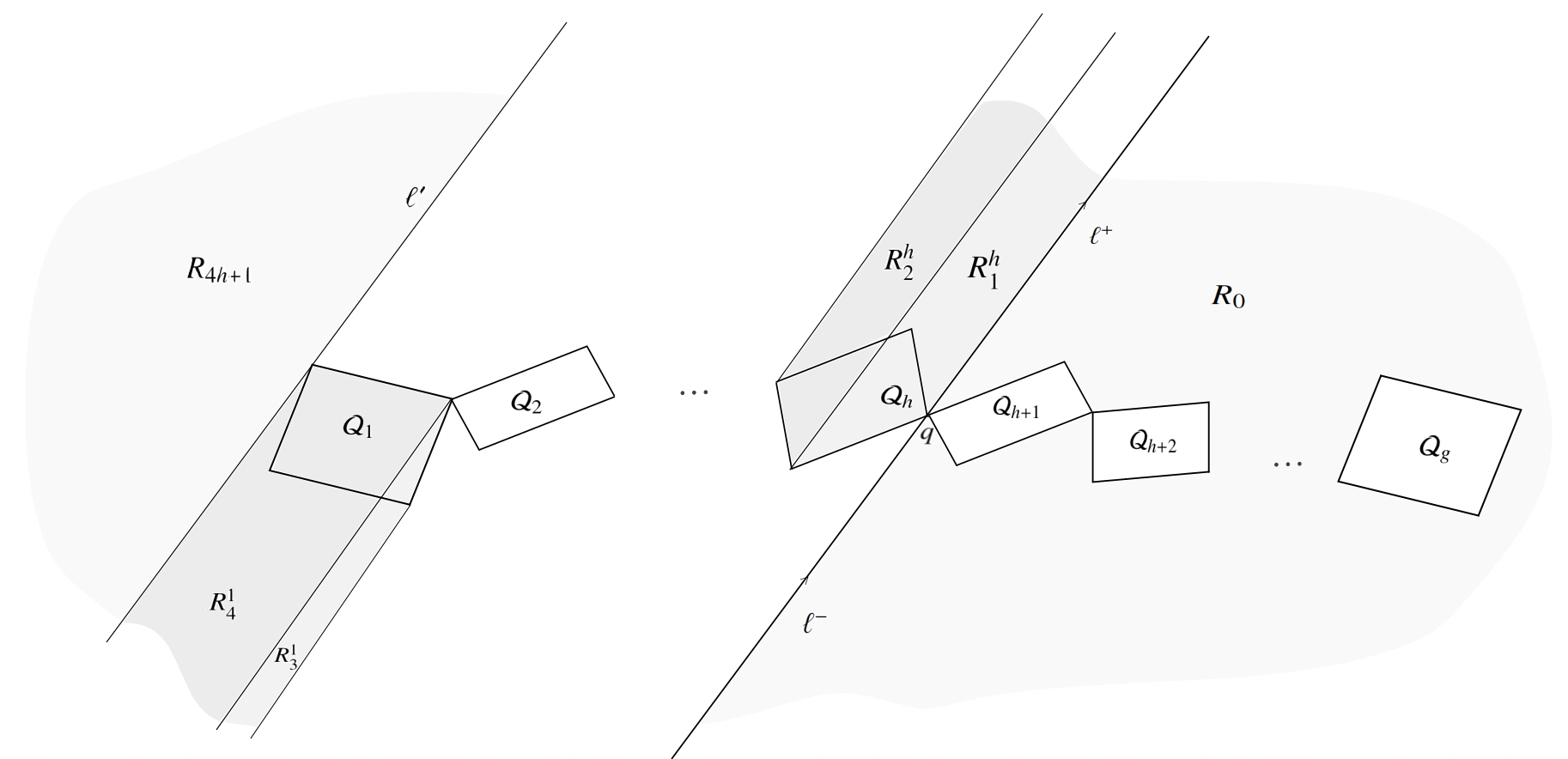}
  \caption{The chain of $g$ quadrilaterals corresponds to the positive handles ($1\leq i\leq h$) each like the right-hand figure in Figure \ref{leftrightchoice},   and the non-positive handles ($h+1 \leq i \leq g$) each like the left-hand figure in Figure \ref{leftrightchoice}. }
  \label{transhighergenus}
\end{figure}

\noindent For each $1\le i\le h$, consider the four polygonal regions $\{R_1^i, R_2^i,R_3^i,R_4^i\}$ in $\C$  satisfying \textit{(i)-(iii)} above, bounded by $e_1^i,e_2^i,e_3^i, e_4^i$ respectively together with infinite rays that are parallel to $\ell$ (i.e.\ a translated copy of either $\ell^+$ or $\ell^-$), see Figure \ref{transhighergenus}. 

\noindent By construction, $R_1^i$ and $R_3^i$ lie on opposite sides of $e_1^i$ and $e_3^i$ respectively, and $R_2^i$ and $R_4^i$ lie on opposite sides of $e_2^i$ and $e_4^i$ respectively (\textit{c.f.} Figure \ref{leftrightchoice}). Moreover, each $R_1^i$ lies on the left of the corresponding side $e_1^i$. 
The two  rays from the left-most point of $\mathcal{Q}_1$ forms a straight-line $\ell'$ to $\ell$. By introducing on $\ell'$ an orientation coherent to those of the straight-line $\ell$, it makes sense to say that $\ell'$ has the entire chain of quadrilaterals to its right. Finally, we define $R_{4h+1}$ the half-plane on the left of $\ell^\prime$. We can now proceed as above. The set 
\begin{equation}
    R_0\cup \left(\,\bigcup\limits_{i=1}^h \big(R^i_1 \cup R^i_2\cup R^i_3 \cup R^i_4\big) \,\right)\cup R_{4h+1}
\end{equation}
\noindent is topologically an immersed disk on the Riemann sphere, with its boundary being the chain of  quadrilaterals $\mathcal{Q}_1 \cup \mathcal{Q}_2 \cup \cdots \mathcal{Q}_g$.  The translation surface $\Sigma$ is obtained by identifying the remaining boundary sides $\{e_j^i\}$ as follows: $e_3^i$ is identified with $e_1^i$ via the translation $A_i$, and $e_2^i$ is identified with $e_4^i$ via the translation $B_i$. It is easy to see that $\Sigma$ is homeomorphic to a surface $S_{g,1}$ with the puncture at infinity, corresponding to a pole of order $2$, and one branch-point of order $2g$. We eventually delete the branch-point in order to get a translation surface on $S_{g,2}$ with the desired holonomy.
\end{proof}

\textit{Remark.} In the case that $g\ge2$ we  also have the following alternative construction which is easier. Namely, we construct a translation surface $\Sigma_j$  homeomorphic to $S_{1,1}$ exactly above, for each handle (so $1\leq j\leq g$), such that the holonomy of $\Sigma_j$ is precisely the holonomy of the $j$-th handle in the original character $\chi_{g,1}$.  Note that each $\Sigma_j$ has a pole of order two, and exactly one branch-point. Choose an embedded infinite ray in each, between the cone point and the pole such that each develops onto an infinite ray in $\C$ in the same direction. We can then glue $\Sigma_j$ to $\Sigma_{j+1}$ along these rays as in Definition \ref{glue-ray}, for each $1\leq j<g$. The resulting translation surface $\Sigma$ is homeomorphic to $S_{g,1}$, with holonomy $\chi_{g,1}$, and has one pole of order $(g+1)$ and one branch-point. As before, deleting the branch-point results in a surface homeomorphic to $S_{g,2}$, equipped with a translation structure having the desired monodromy $\chi_{g,2}$. However, notice that this construction results in a pole of order greater than $2$.  In Lemma \ref{transurfk1} below we shall need to consider a translation surface $\Sigma$ on a two-punctured genus $g$ surface with one pole of order \textit{exactly} two. This motivates the more complicated argument above.

\medskip 

\noindent In the case when $g=0$, we have the following construction of a translation structure realizing a prescribed monodromy:

\begin{prop}\label{trans2} If $k\geq 2$, any  representation $\chi_{0,k+1}:\Gamma_{0,k+1} \to \C$ with at least one puncture with trivial monodromy, is the holonomy of some translation structure on $S_{0,k+1}$,  where one of the punctures corresponds to a zero of the abelian differential. 
\end{prop}

\begin{proof}
If $\chi_{0,k+1}$ is trivial, then let $S_{0,k+1}$ be the complex plane $\C$ with $k$ punctures. This trivially defines a translation structure since the charts to $\C$ are given by the inclusion map (and the transition maps are all identity maps). The points removed are clearly apparent singularities, and so is the order-two pole at infinity. 

\smallskip

\noindent Now consider the case when $\chi_{0,k+1}$ is non-trivial. Since by removing regular points, it is easy to add punctures with trivial monodromy to a translation surface, we can assume without loss of generality, that each puncture except one, has non-trivial monodromy. We shall build a translation surface  $\Sigma$ that is homeomorphic to $S_{0,k}$ with one branch-point, such that after removing the branch-point,  the monodromy of the translation structure on the resulting surface is $\chi_{0,k+1}$; note that the monodromy around branch-point would be trivial.  

\smallskip

\noindent We now describe two such constructions for $k=2$; note that since one puncture has trivial monodromy, by \eqref{sumC} the other two punctures have monodromy $a$ and $ -a$ respectively, where $a$ is a non-zero complex number. 

\smallskip

\noindent (i) Choose a direction $\theta$ that is not parallel to the line passing through $0$ and $a$ (i.e. $\theta \neq \pm \arg(a)$). Consider the infinite strip $S$ in $\C$ in the direction $\theta$, with an orientation induced from the complex plane, such that the two boundary components differ by the translation $z\mapsto z + a$. Then let $A$ be the translation surface  obtained by identifying the two boundary components of $S$ via the translation. 

\smallskip

\noindent (ii) As before, choose a direction $\theta$ that is different from the direction determined by $a$. Consider the complex plane $\C$ with slits along two infinite rays $r_1$ and $r_2$, in the direction $\theta$, that start from $0$ and $a$ respectively. Let $r_i^-$ and $r_i^+$ be the upper and lower sides of the slits, respectively, where $i=1,2$. We then identify $r_1^-$ with $r_2^+$, and $r_1^+$ with $r_2^-$, each by the translation $z\mapsto z+a$. 

\smallskip

\begin{figure}
  \centering
  \includegraphics[scale=0.47]{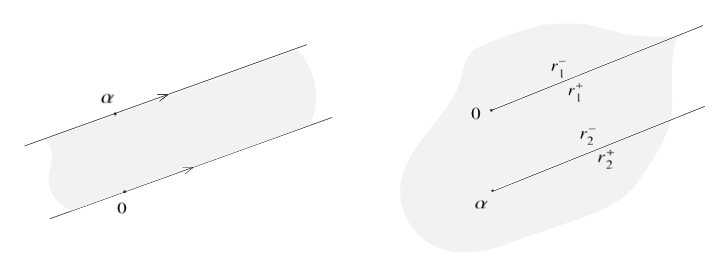}
  \caption{The construction of a translation surface homeomorphic to $S_{0,2}$, with prescribed holonomy. }
\end{figure}

\noindent In all these cases, the translation surface we obtain is homeomorphic to an annulus. Removing a regular point from this annulus, we obtain a puncture with trivial monodromy, and hence we obtain a translation structure on a surface homeomorphic to $S_{0,3}$ that has the desired monodromy. The abelian differential $\omega$ has poles at the two punctures with non-trivial monodromy: in construction (i), both poles are simple (\textit{i.e.} order one), whereas in construction (ii), one of them has order two (and the other is simple). 

\smallskip

\noindent The difference between constructions (i) and (ii) is the following: if you choose regular point on the resulting translation surface and consider an arc to a puncture that develops onto a ray in  $\C$ in the direction $\theta$, then if that puncture it is incident to is the one with monodromy $a$ in (i), it is the one with monodromy $-a$ in (ii), and vice versa. 

\smallskip

\noindent For $k>2$, let the translations around the $k$ punctures with non-trivial monodromy, as determined by $\chi_{0,k+1}$, be $A_1,A_2,\ldots A_k$. Let $A_i(z) = z + a_i$ where $a_i\in \mathbb{C}^\ast$,  for $1\leq i\leq k$.  Observe that by \eqref{sumC} we have 
\begin{equation*}
    a_k = - \sum\limits_{i=1}^{k-1} a_i.
\end{equation*}
 
\noindent Consider the cyclically ordered sequence of points $\{p_1,p_2,\ldots p_k\}$ in $\C$ where $p_1=0$ and $p_i = A_{i-1}(p_{i-1})$ for each $i=2,\ldots, k$. From our observation it follows that $p_1 = A_k(p_k)$. Choose a direction $\theta$ which is not parallel to $\overline{0\,a_i}$ for any $i$ and let $r_i$ be the infinite ray in $\C$ that starts from $p_i$ and has direction $\theta$. Now for each $1\leq i\leq k$, we construct a translation surface $A_i$ homeomorphic to an annulus, either by (i) or (ii) above, such that an arc from a regular point to the puncture with monodromy $A_i$ develops onto the ray $r_i$ on $\C$. We then glue $A_i$ to $A_{i+1}$ along the rays $r_i$ and $r_{i+1}$, as in Definition \ref{glue-ray}, where $i\in \{1,2,3,\ldots k\}$ is cyclically ordered so that $A_k$ is glued with $A_1$. The resulting translation surface is homeomorphic to $S_{0, k}$ with $k$ punctures having monodromy $A_1,A_2,\ldots A_k$ respectively, and a single branch-point of angle $2\pi k$ (which is the point corresponding to the $p_i$'s after identifications).  
Removing this branch-point, we obtain a translation structure on a surface homeomorphic $S_{0,k+1}$ with monodromy equal to $\chi_{0,k+1}$, as desired.
\end{proof}

\textit{Remark.} 
Although our main result concerns punctured surfaces with negative Euler characteristic, it is worth noting that in the case of the surface $S_{0,2}$ (i.e.\ $g=0, k=2$),  \textit{any} representation $\chi:\pi_1(S_{0,2})\cong\mathbb{Z}\longrightarrow \C$ can be realized as the monodromy of some translation structure on $S_{0,2}$. In case $\chi$ is trivial the desired structure is the complex plane punctured at any point. Otherwise, if $\chi(1)=\alpha\in\C^*$, we can proceed exactly as in the case $(i)$ above; the resulting Euclidean cylinder is the desired structure on $S_{0,2}$. Note that such a non-trivial representation $\chi$, though degenerate, is the monodromy of a complex projective structure \textit{without} apparent singularities. Theorem \ref{ab61} shows that this cannot happen in the case of negative Euler characteristic.

\medskip

\noindent We have thus been able to deal two cases -- one when the number of punctures $k=2$, but the genus $g>0$, and the other when the genus $g=0$, but the number of punctures is arbitrary, \textit{i.e.} $k\geq 2$. Using the gluing construction along rays once again, we can now prove the following more general statement:

\begin{prop}\label{trans} Let $\Gamma$ be the first  homology group of a surface $S_{g,k}$ where $k\geq 3$.  Let $\chi:\Gamma\to \C$ be a non-trivial representation such that there is at least one puncture with trivial monodromy. Then there is a translation structure on $S_{g,k}$ whose monodromy is $\chi$, such that the corresponding abelian differential $\omega$ on $S_{g,k}$ extends to a meromorphic abelian differential $\overline{\omega}$ on the closed surface $S_g$. 
\end{prop}

\begin{proof}
We have already seen the necessity of a puncture with trivial monodromy. Our construction of a translation structure with the prescribed monodromy $\chi$ splits into a few different cases. Note that the case for $g=0$ is already done in Proposition \ref{trans2}. Henceforth we shall assume that $g>0$. 

\smallskip

\noindent Let $\chi_{0,n+1}:\Gamma_{0,n+1} \to \C$ be the representation as in \eqref{chi0n}, obtained by restricting $\chi$ to the subsurface of $S_{g,k}$ that contains all the punctures with non-trivial monodromy.  Here we assume that there are $0<n<k$ such punctures, and hence the subsurface is homeomorphic to $S_{0,n+1}$. Note that  $\chi_{0,n+1}$ has exactly one puncture with trivial monodromy, corresponding to the boundary of the subsurface (which is trivial in homology).  Let the rest of the monodromies be represented by the corresponding translation vectors $a_1, a_2,\ldots a_{n} \in \C^\ast$. 
By Proposition \ref{trans2} there is a translation surface $\Sigma_0$ homeomorphic to $S_{0,n}$ with a single branch-point $p$, and at least one of the other punctures is at infinity (in the induced flat metric), such that the monodromy around the $n$ punctures are the translations by $a_1, a_2,\ldots a_{n}$. Then 
$\Sigma_0 \setminus \{p\}$ is homeomorphic to $S_{0, n+1}$ and carries a translation structure realizing the character $\chi_{0,n+1}$. 

\smallskip 

\noindent Now let $\chi_{g,1}:\Gamma_{g,1} \to \C$ be the representation as in \eqref{chig1}, obtained by restricting $\chi$ to the subsurface of $S_{g,k}$ that contains all the handles.  If $\chi_{g,1}$ is a non-trivial representation, then by Lemma \ref{cbas} we can assume that each handle-generator maps to a non-zero complex number. From the proof of Proposition \ref{trans1} there is a translation surface $\Sigma_1$  homeomorphic to $S_{g,1}$, with one branch-point and one puncture which is a pole of order two, with holonomy $\chi_{g,1}$. We then glue $\Sigma_0$ and $\Sigma_1$ along a suitable choice of rays from the branch-point on each surface to a pole, that develop to infinite rays in $\C$ that are parallel, see Definition \ref{glue-ray}. The resulting translation surface is then homeomorphic to $S_{g,n}$ and has one branch-point. Removing the branch-point, and $(k-n-1)$ additional regular points, we obtain the desired surface homeomorphic to $S_{g,k}$ equipped with a translation structure having monodromy $\chi$, as desired. 

\smallskip

\noindent If $\chi_{g,1}$ is the trivial representation, there are two cases:

\smallskip

\noindent \textit{Case 1: $\chi$ admits at least two punctures on $S_{g,k}$  with trivial monodromy.} We start with the translation structure on $S_{g,3}$ that realizes the trivial representation (see Lemma \ref{triv}). Recall from that construction, that such a translation structure is obtained from a translation surface $\Sigma_1$ with trivial holonomy that has two branch-points, and a single pole of higher order at infinity in the induced flat metric. We can now glue $\Sigma_0$ and $\Sigma_1$ along suitably chosen rays from a  branch-point to a pole at infinity, on either surface, as in Definition \ref{glue-ray}. Note that the resulting surface is homeomorphic to $S_{g,n}$ and has two branch-points. Removing these two branch-points, and an additional $(k-n-2)$ regular points if necessary, we obtain a surface homeomorphic to $S_{g,k}$ equipped with a translation structure having monodromy $\chi$.

\smallskip

\noindent \textit{Case 2:  There is exactly one puncture on $S_{g,k}$ with trivial monodromy.} By our assumption, there are exactly $k-1$ punctures with non-trivial monodromy, which we denote by  $a_1, a_2, \ldots, a_{k-1}$ as before. In this case, we first construct a translation structure on $S_{g,3}$ that realizes the representation that has all handle-generators have trivial monodromy, exactly one puncture (call it $p$) has trivial monodromy, and the other two punctures $q$ and $r$ have monodromy $a_1$ and $-a_1$ respectively: For this, we use the same covering map $\Pi:S_g\to\cp$ as in the proof of Lemma \ref{triv}, namely one that has three ramification points $0,1,\infty \in \cp$  and three critical points (the preimages of $0,1,\infty$) on the domain surface. This time, we equip $\cp$ with a translation structure for which the abelian differential $\omega$ has a simple pole at $0$ and $\infty$ (and residue $ a_1/2g$ and $-a_1/2g$ respectively), and $1$ is  regular point. Note that as a translation surface, the target is a just Euclidean cylinder with a distinguished regular point, and its pullback via $\Pi$ is then a translation surface $\Sigma_1$ homeomorphic to $S_{g,2}$ with one cone point (which is the pre-image of $1$) and two punctures which are two simple poles of $\Pi^\ast \omega$. Removing the branch-point, we obtain the desired translation structure on $S_{g,3}$. Note that in the case that $k=3$, the above construction completes the proof. 

\smallskip

\noindent We now assume that $k>3$; note that we can then assume without loss of generality that $a_1 \neq - a_2$. 

\noindent To $\Sigma_1$ we shall glue a translation surface $\Sigma_0^\prime$  homeomorphic to $S_{0,k-2}$ with a single branch-point,  where we construct $\Sigma_0^\prime$ such that the monodromy around the punctures are $a_1 + a_2, a_3, \ldots, a_{k-1}$. Such a translation surface exists by Proposition \ref{trans2}; in that construction, we can also ensure that the puncture $p^\prime$ having monodromy $a_1 + a_2$ is a simple pole. The gluing is now along a choice of a ray on $\Sigma_0^\prime$ from the branch-point to $p^\prime$, and of a ray on $\Sigma_1$ from the branch-point there to the puncture $r$ which has monodromy $-a_1$. Once again, we choose the rays so that they develop onto parallel rays on $\C$. The translation surface obtained after this gluing is homeomorphic to $S_{g,k-1}$; note that the pole obtained by identifying the end-points of the rays now has holonomy $-a_1 + (a_1 + a_2) = a_2$, and the other endpoints define a single branch-point after the identification. Removing this branch-point, we obtain a translation structure on $S_{g,k}$ with monodromy $\chi$, as desired. Thus, in all cases, we are able to construct the desired translation structure, and we are done.
\end{proof}

\subsection{Affine surfaces with at least two punctures}\label{ssaff1}

In this section (and the next) we shall assume that  $\rho:\Pi \to \text{Aff}(\C)$ is a non-trivial representation, where 
\begin{equation*}
 \text{Aff}(\C) = \big\{\ z\mapsto a z + b \ \vert\ a \in \C^\ast \text{ and } b \in \C \ \big\}
\end{equation*}
is the subgroup of $\pslc$ comprising affine transformations.  Here, recall that  $\Pi$ denotes the fundamental group of the surface $S_{g,k}$. Since $\text{Aff}(\C)$ is precisely the subgroup of $\pslc$ that stabilizes the point $\infty \in \cp$, any degenerate representation into $\pslc$ that has a global fixed point, can be conjugated to an affine representation as $\rho$ above. Note that this includes the case of co-axial monodromy, when the representation globally fixes \textit{two} points in $\cp$. 

\medskip 

\noindent In the language of geometric structures, an \textit{affine structure} on $S_{g,k}$ is an atlas of charts to $\C$ such that the transition maps are affine maps; notice that translation structures (see \S\ref{sstrans}) are a special case. Recalling our Definition \ref{psurf}, with the same spirit we describe an affine structure as follows:

\begin{defn}\label{def_affstruc}
An \textit{affine surface} to be one obtained by identifying sides of a (possibly disconnected, and possibly non-compact) polygon in $\C$ by affine maps. Note that the vertices after identification may result branch-points; a neighborhood of a branch-point on an affine surface develops to $\C$ as the map $z\mapsto z^n$ for some $n>1$. Unlike in Definition \ref{psurf}, however, the resulting surface $\Sigma$ may not have an induced Euclidean metric. Removing the set of branch-points $B$, we obtain an affine structure on the punctured surface with the punctures in $B$ having trivial monodromy (\textit{i.e.} apparent singularities of the affine structure).
\end{defn}

\medskip

\noindent In the course of this section we would need to also consider another particular type of affine structures, namely:

\begin{defn}\label{def_halftrans}
A \textit{half-translation structure} is obtained by starting with a collection of (possibly non-compact) polygons in $\C$ bounded by straight lines and/or rays and/or segments, and identifying such sides by half-translations, i.e.\ maps of the form $z\mapsto \pm z + c$. 
The resulting surfaces $\Sigma$ acquires an Euclidean metric, with possible cone-points with cone-angles $k\,\pi$, where $k\in\mathbb{Z}^+$. On the complement of the cone-points, we obtain an Euclidean structure locally modelled on $\C$ and such that the transition maps are affine maps of the form $z\longmapsto \pm z+c$. These structures naturally come equipped with a (possibly meromorphic) quadratic differential $q$, induced from the quadratic differential $dz^2$ on $\C$.  A zero of $q$ of order $m\geq 1$ corresponds to a cone-point of angle $(m+2)\pi$, and a simple pole is a cone-point $\pi$.    
\end{defn}


\medskip

\noindent In this section,  our strategy would be to construct affine structures with a given affine monodromy $\rho:\Pi \to \text{Aff}(\C)$. For this, we prove analogues of the results and  constructions in \S\ref{sstrans}. 

\medskip

\noindent Just like we did for translation surfaces in Definition \ref{glue-ray}, we can glue affine surfaces along rays as follows:

\begin{defn}[Gluing affine surfaces]\label{glue2} 
Let $\Sigma_1$ and $\Sigma_2$ be two affine surfaces, with embedded arcs $l_1$ and $l_2$ respectively from a branch-point or regular point to a puncture, that each develop onto an infinite ray in $\C$. Then we can define a new affine surface $\Sigma$ by making a slit along $l_1$ and $l_2$, and gluing cross-wise to obtain an affine surface $\Sigma$. This gluing is exactly as in Definition \ref{glue-ray}, except that now the 
the identifications between sides of the slit are by an affine map and its inverse. (In particular, the two rays that are the developed images of the lifts of $l_1$ and $l_2$ need not be parallel.)   As before, if $\Sigma_i$ is homeomorphic to $S_{g_i,k_i}$ for $i=1,2$, then $\Sigma$ is homeomorphic to $S_{g_1+g_2, k_1+k_2-1}$. Moreover, the starting points of the rays determine a branch-point on $\Sigma$.
\end{defn} 

\noindent However, the above gluing has a disadvantage: if the two rays in the developing image are not identical but related by an affine map $A$,  the holonomy of the resulting affine surface $\Sigma$ might be affected by $A$. For example, if one of the endpoints of the arc being slit is a puncture with non-trivial monodromy $M_1$ on one surface, and the corresponding puncture on the other surface has non-trivial monodromy $M_2$, then the monodromy around the puncture on $\Sigma$ will be $M_1 A M_2A^{-1}$.  Note that this issue does not arise in the case of a translation surface (\textit{c.f.} Definition \ref{glue-ray}) since the holonomy then is abelian.  

\smallskip

\noindent To handle this, we introduce the following variant of the gluing procedure that ensures that after gluing, the holonomy on the two constituent sub-surfaces remain unchanged. 

\begin{defn}[Gluing preserving holonomy]\label{glue-new} 
Let $\Sigma_0$ and $\Sigma_1$ be affine surfaces, as before, with rays $r_0$ and $r_1$ respectively, each from a (possibly branched) point of the surface to the puncture at infinity. The only requirement will be that the developing map of either surface takes the starting point of the ray to a common point $p\in \C$. 
In the gluing procedure we shall use the complex plane (thought of an affine surface) together with a choice of a ray $r_\star$ leaving from $p$, which we denote by  $(\mathbb{C},r_\star)$. The ray $r_0$ develops onto a ray $\overline{r}_0$ leaving from $p$. We slit $\Sigma_0$ along $r_0$ and $(\C,r_\star)$ along $\overline{r}_0$ and then we identify the resulting boundary rays cross-wise, as in the Definition \ref{glue2}, to obtain an affine surface $\Sigma_0'$ with holonomy $\rho_0$. In the same fashion, the ray $r_1$ develops on a ray $\overline{r}_1$ leaving from $p$ and hence we glue the affine surfaces $\Sigma_1$ and $(\mathbb{C},r_\star)$ along rays to obtain an affine surface $\Sigma_1'$ with holonomy $\rho_1$. By construction, the new surfaces $\Sigma_0'$ and $\Sigma_1'$ both contain a ray, say $r_0'$ and $r_1'$ respectively, from the branch-point to a puncture at infinity that develop onto the same ray $r_\star\subset\mathbb{C}$. We slit $\Sigma_0'$ and $\Sigma_1'$ along these two rays and glue as in Definition \ref{glue2} to obtain the final affine surface $\Sigma$. If $\Sigma_0 \cong S_{g_0,k_0}$ and  $\Sigma_1 \cong S_{g_1,k_1}$ then the resulting surface $\Sigma$ is homeomorphic to $S_{g_0 + g_1,k_0 + k_1-1}$. Note that the starting points of the arcs  get identified to a branch-point on $\Sigma$. In this construction the rays on the two surfaces being glued develop onto the same ray $r_\star$, therefore the resulting affine surface has monodromy $\rho$ that restricts to $\rho_0$ and $\rho_1$ on the sub-surfaces  corresponding to $\Sigma_0$ and $\Sigma_1$ respectively.
\end{defn} 

\begin{figure}
  \centering
  \includegraphics[scale=0.42]{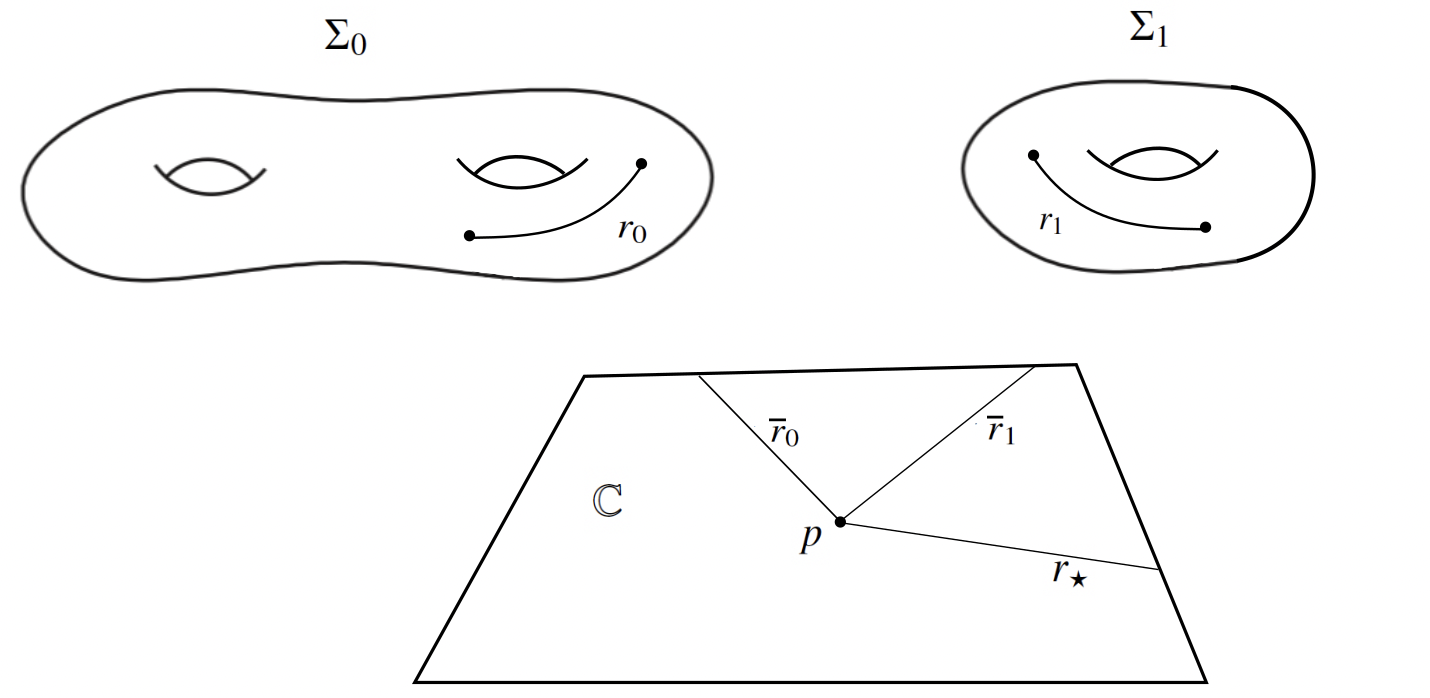}
  \caption{In the gluing preserving holonomy (Definition \ref{glue-new}) we introduce an intermediate copy of the complex plane. }
\end{figure}


\medskip

\noindent We start with the analogue of Proposition \ref{trans2}, that handles the case when $g=0$:

\begin{prop}\label{affg0}  If $k\geq 2$, any representation $\rho:\pi_1(S_{0,k+1}) \to \text{Aff}(\C)$ with at least one puncture with trivial monodromy, is the holonomy of some affine structure on $S_{0,\,k+1}$. 
\end{prop}

\begin{proof}
Let us start by assuming $k=2$ and let $\rho:\pi_1(S_{0,3}) \to \text{Aff}(\C)$ with at least one puncture with trivial monodromy. If all the punctures have trivial monodromy then $\C\setminus\{ q_1,\,q_2\}$ provides the desired structure for any pair of points $q_1,\,q_2\in\C$. We therefore assume the existence of at least one puncture with non-trivial monodromy, say $A\in\text{Aff}(\C)$. The remaining puncture  necessarily has monodromy $A^{-1}$. 

\smallskip

\noindent Let $p_0\in\C$ be any point. Let $r_0$ be a ray leaving from $p_0$ and let $A(r_0)$ be the image of $r_0$ leaving from $A(p_0)$. We can always choose $r_0$ such that the rays $r_0$ and $A(r_0)$ are not contained one in each other, \textit{i.e.} $r_0\not\subset A(r_0)$ nor $A(r_0)\not\subset r_0$. However, the rays $r_0$ and $A(r_0)$ may intersect at some point $s\in\,r_0\,\cap\,A(r_0)$. If this is the case, we replace $p_0$ with a point $p_\star\in r$ and such that $p_\star\notin\,\overline{p_0\,s}$. Notice that the segment $\overline{p_0\,s}$ may be degenerate, that is a point, if $p_0=\text{Fix}(A)$. Let $r_\star$ be the sub-ray of $r_0$, leaving from $p_\star$. By construction, $r_\star$ is disjoint from its image $A(r_\star)$.

\smallskip

\noindent We slit $\C$ along $r_\star$ and let $r_\star^+$ and $r_\star^-$ be the right and the left copy of $r_\star$ respectively. In the same fashion, we slit $\C$ along $A(r_\star)$ and let us denote $A(r_\star)^+$ and $A(r_\star)^-$ be the right and the left copy of $A(r_\star)$ respectively. We glue $r_\star^+$ with $A(r_\star)^-$ and similarly we glue $r_\star^-$ with $A(r_\star)^+$ by using the affine map $A$, see figure \ref{fig:3punctsphereaff}. 

\noindent The resulting surface is a two-punctured sphere which carries an affine structure with one branch-point of magnitude $4\pi$ arising from the identification of the points $p_\star$ and $A(p_\star)$. One of the punctures has monodromy $A$ and the other puncture has monodromy $A^{-1}$ by construction. We eventually delete the branch-point in order to get an affine structure $\Sigma$ on $S_{0,3}$ with the desired monodromy.

\begin{figure}[h!]
    \centering
    \begin{tikzpicture}[scale=0.7, every node/.style={scale=0.8}]
    \draw [thin, black] (3,-1)-- (8,5);
    \draw [thin, black] (3,-1)-- (7.75,5);
    \draw [thin, black] (1,0)-- (-2,5);
    \draw [thin, black] (1,0)-- (-1.75,5);
    \fill [white, pattern=north east lines, pattern color=pallido](-5,-3)--(11,-3)--(11,5)--(8,5)--(3,-1)--(7.75,5)--(-1.75,5)--(1,0)--(-2,5)--(-5,5)--(-5,-3);
    \draw plot [mark=*, smooth] coordinates {(3,-1)};
    \draw plot [mark=*, smooth] coordinates {(1,0)};
    \draw [-latex, thin, bend right] (5,3) to (0,3);
    \node at (2.5,4) {$A$};
    \node at (6.25,2) {$r_\star^-$};
    \node at (5,2) {$r_\star^+$};
    \node at (0.75,2) {$A(r_\star)^-$};
    \node at (-1,2) {$A(r_\star)^+$};
    \node at (3,-1.5) {$p_\star$};
    \node at (1,-0.5) {$A(p_\star)$};
    \node at (7,-1) {$\C$};
    \end{tikzpicture}
    \caption{}
    \label{fig:3punctsphereaff}
\end{figure}

\smallskip

\noindent We now consider the general case when $k\ge3$. Let $\rho:\pi_1(S_{0,k+1}) \to \text{Aff}(\C)$ be a representation such that at least one puncture has trivial monodromy. There is no loss of generality in assuming that \textit{exactly} one puncture has trivial monodromy. In fact, as already observed in Proposition \ref{trans2}, it is easy to add further punctures with trivial monodromy to an affine structure by deleting some regular points. The idea for the general case is to define $k-1$ affine structures on $S_{0,3}$ as above and then glue them together to obtain an affine structure on $S_{0,k+1}$ with one apparent singularity.

\smallskip

\noindent Let $A_1,\,A_2,\,\dots,A_k,\,A_{k+1}$ be the monodromy of the punctures. We may assume $A_{k+1}=\text{Id}$ and then we can notice that $A_k=\big(A_1A_2\cdots A_{k-1})^{-1}$. Let $p_0\in\C$ be any point and let $r_0$ be a ray leaving from $p_0$. For any $i=1,\dots,k-1$, we define $p_i=A_i(p_0)$ and $r_i=A_i(r_0)$. Clearly, $r_i$ is a ray leaving from $p_i$. Notice that $r_0$ can be chosen such that $r_i\not\subset r_0$ nor $r_0\not\subset r_i$ for $i=1,\dots,k-1$. However, the ray $r_i$ may still intersect $r_0$ at some point, say $s_i$, as observed above. We claim the existence of some good point $\overline{p}_0\in r_0$ and a sub-ray $\overline{r}_0\subseteq r_0$ leaving from $\overline{p}_0$ such that, upon setting $\overline{r}_i=A(\overline{r}_0)$, then $\overline{r}_i\subseteq r_i$ and the rays $\overline{r}_{0}$ and $\overline{r}_{i}$ are disjoint.

\smallskip

\noindent We briefly show why the claim above is true. Let $\xi_0:[0,\,\infty)\longrightarrow \C$ be a parametrization of $r_0$. We can easily note that any parameter $\tau\in[0,\,\infty)$ determines a sub-ray $r_\tau\subseteq r_0$. Moreover, given two parameters $\tau_1,\tau_2$ such that $\tau_1<\tau_2$, then the corresponding rays are such that $r_{\tau_2}\subset r_{\tau_1}$. We now define $\xi_i:[0,\,\infty)\longrightarrow \C$ to be the parametrization of $r_i$ satisfying the equation $\xi_{i}=A_i\circ \xi_{0}$. Upon setting $t_0=0$, as we showed above, it is possible to find a time $t_1$ such that the sub-ray $r_{t_1}\subset r_0$, leaving from $\xi_0(t_1)$, is disjoint from the ray $A_1(r_{t_1})$ leaving from $\xi_1(t_1)$. However, it may happen that the rays $r_{t_1}$ and $A_2(r_{t_1})$ still intersect. We apply again the same reasoning to these rays. There is a time $t_2\ge t_1$ such that $A_2(r_{t_2})\subseteq A_2(r_{t_1})\subset r_2$, leaving from $\xi_2(t_2)$, is disjoint from the sub-ray $r_{t_2}\subset r_0$ leaving from $\xi_0(t_2)$. By proceeding in the same fashion at most $k$ times, it is then possible to find a time $\overline{t}\ge t_0$ such that each sub-ray $\overline{r}_i=A_i(r_{\overline{t}})$, leaving from $\xi_{i}(\overline{t})$, is disjoint from $r_{\overline{t}}$ for any $i=1,\dots,k-1$. 

\smallskip

\noindent By replacing $p_0$ with $\overline{p}_0=\xi_0(\overline{t})$, we may assume without loss of generality that each ray $r_i$ is disjoint from $r_0$. For each $i=1,\dots,k-1$, we consider a copy of the $\C$ along with the rays $r_0$ and $r_i=A_i(r_0)$ that are disjoint. For the $i$-th copy of $\C$, we can proceed as above by slitting $\C$ along them and re-gluing to obtain a branched affine structure $\Sigma_i$ on a two-punctured sphere with one branch-point arising from the identification of the points $p_{0}$ and $p_i=A_i(p_{0})$. One of the punctures has monodromy $A_i$ and the other has monodromy $A_i^{-1}$. 

\noindent We now explain how to glue these $k-1$ surfaces. Let $\ell_i$ be an arc on $\Sigma_i$ from the branch-point to the puncture with monodromy $A^{-1}_i$ that develops onto an infinite ray in $\C$ starting from $p_0$.  We can now glue $\Sigma_1, \Sigma_2, \ldots \Sigma_{k-1}$ successively along these rays, as in Definition \ref{glue-new}. The resulting surface is homeomorphic to $S_{0,k}$ and carries an affine structure $\Sigma$. By construction, the punctures have monodromy $A_1,\,A_2,\dots A_{k-1},$ and $A_k=\big(A_1\cdots A_{k-1})^{-1}$.  The branch-points on each $\Sigma_i$ get identified to a single branch-point $P\in\Sigma$. By deleting that point, the surface $\Sigma\setminus\{P\}$ carries an affine structure with monodromy $\rho$ as desired. \qedhere
\end{proof}

\medskip

\noindent For the case when $g>0$, we begin with the following observation:

\begin{lem}\label{cbas2} Assume that $g>0$, and let $\rho: \Pi \to \text{Aff}(\C)$ be a representation such that $\rho(\gamma) \neq \text{Id}$ for at least one handle-generator $\gamma_0$ on $S_{g,\,k}$. Then there exists $\phi \in \text{MCG}(S_{g,\,k})$ with an associated outer automorphism $\phi_\ast:\Pi \to \Pi$ such that $\rho \circ \phi_\ast (\gamma) \neq \text{Id}$ for each handle-generator $\gamma$. 
\end{lem} 
\begin{proof}[Sketch of the proof]  The proof is exactly the same as that of Lemma \ref{cbas}: note that the standard basis of homology can be considered as a generating set for $\Pi$, and changes of homology basis used in the proof of Lemma \ref{cbas} can be realized by a mapping class. 
\end{proof}

\noindent Henceforth, we shall assume that in case $g>0$, the representation $\rho$ maps each handle-generator to a non-trivial affine map (\textit{c.f.} the remark following Lemma \ref{cbas}.) The following is the analogue of Proposition \ref{trans1}:

\begin{prop}\label{affk2}  Let $g>0$ and let $\rho: \pi_1(S_{g,2}) \to \affc$ be a non-trivial representation such that  there is at least one puncture with trivial monodromy.  
Then there is an affine structure on $S_{g,2}$ with monodromy $\rho$, obtained by puncturing an affine surface $\Sigma$  with a unique branch-point. 
\end{prop}

\begin{proof}
We shall now describe the construction of the affine surface $\Sigma$ by gluing of polygons, as in Proposition \ref{trans1}. A crucial difference from the construction there is that in the case of affine holonomy, the commutator of two elements (\textit{e.g.} a pair of handle-generators) need not map to the trivial (identity) element via $\rho$. Depending on whether the image of $\rho$ is finite of order two, we shall have to distinguish two cases. 

\smallskip

\noindent Assume $\text{Im}(\rho)$ is not finite of order two. We start with the case when $g=1$. Let $A,B$ be the affine maps that are the monodromies around the handle-generators $\alpha,\beta$ respectively, \textit{i.e.} $\rho(\alpha) = A$ and $\rho(\beta) = B$.  Note that by the assumption of non-triviality of $\rho$, and by Lemma \ref{cbas2}, we can assume that neither $A$ nor $B$ is the identity map. As already observed above, the commutator $C := [A, B] = ABA^{-1}B^{-1}$ need not be the identity map; however, it is easy to verify that $C$ is always a translation, and since $\rho:\pi_1(S_{1,2})\to \affc$ is a homomorphism,  the remaining puncture has monodromy $C^{-1}$ around it. 

\smallskip

\noindent Fix a point $p\in \C$ that is not a fixed point of $A$ or $B$, and let $q_1 := AB(p)$ and $q_2 :=BA(p)$. Note that $q_1 = q_2$ if $A$ and $B$ commute.  Consider the four directed line-segments $e_1 := \overline{A(p)q_1}$, $e_2 := \overline{pA(p)}$, $e_3 := \overline{pB(p)}$, and $e_4 := \overline{B(p)q_2}$. Note that $B(e_2) = e_4$ and $A(e_1) = e_3$. 
Consider two additional infinite rays $e_0$ and $e_5$  with starting points $q_1$ and $q_2$ respectively; if $q_1=q_2$ (when $A,B$ commute), we take $e_0=e_5$. The directed curve $L : = \overline{e_0} \cup \overline{e_1} \cup \overline{e_2} \cup e_3 \cup e_4 \cup e_5$ is then an immersed polygonal curve in $\C$.

\begin{figure}
  \centering
  \includegraphics[scale=0.44]{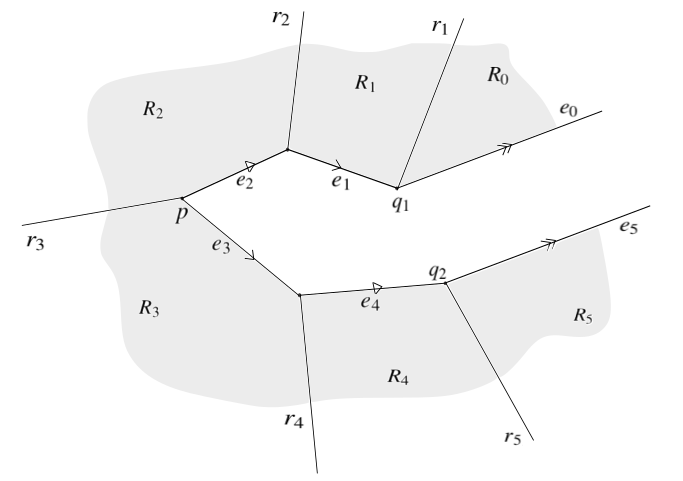}
  \caption{The construction of an affine surface  $\Sigma$ homeomorphic to $S_{1,1}$ where the puncture is at infinity, and $\Sigma$ has a single branch-point (see the proof of Proposition \ref{affk2}). }
  6
\end{figure}

\noindent As in the proof of Proposition \ref{trans1}, we then choose a collection of infinite rays $ \mathcal{R} = \{r_1,r_2,\ldots r_5\}$ with starting points at the vertices of the segments defined above, and consider embedded region $R_i$, for each $i \in \{0,1,\ldots, 5\}$, bounded by the segment $e_i$ and one or two infinite rays from the collection $\mathcal{R}$. As before, there are two choices of such a region, since the union of $e_i$ and the ray(s) from its endpoint(s) separates the complex plane; we choose the one that results in the correct orientation of the handle-generators $\alpha,\beta$ in the affine surface $\Sigma$ that we shall define below. Note that each region $R_i$ has one ideal vertex at $\infty$ on $\C$, and their union $R := R_0 \cup R_1 \cup \cdots \cup R_5$ is an immersed disk with a puncture at $\infty$ and boundary $\partial R = L$. 

\smallskip

\noindent Define the affine surface $\Sigma$ to be the quotient of $R$ obtained by identifying the boundary segments $e_1$ and $e_3$ via the affine map $A$, $e_2$ and $e_3$ via the affine map $B$, and $e_0$ and $e_5$ via the translation $[B,A]$ (which is the identity map if $A,B$ commute, compatible with our requirement that $e_0=e_5$ in such a case). The resulting surface $\Sigma$ is homeomorphic to a punctured torus, where the puncture is the point at $\infty$ on $\C$. From our construction, the pairs of segments $\{e_2,e_4\}$ and $\{e_1,e_3\}$ after identifications define the handle-generators $\alpha,\beta$; here note that the regions $R_i$ are chosen to lie on the correct "side" of these directed segments so that $\alpha$ and $\beta$ have the desired orientation on the punctured torus. Moreover, the handle-generators $\alpha$ and $\beta$ on $\Sigma$ have holonomy $A$ and $B$ respectively, and the monodromy around the puncture at $\infty$ is the translation $C$. Note that the vertices of the segments $\{e_i\}_{1\leq i\leq 4}$ get identified to a single branch-point $q$ on $\Sigma$. The surface $\Sigma \setminus \{q\}$ is thus homeomorphic to $S_{1,2}$ and acquires an affine structure with the desired monodromy $\rho$, in which $q$ is a puncture with trivial monodromy.

\smallskip

\noindent The higher genus case, \textit{i.e.} $g\ge2$ relies on the preceding construction. Let $\{\alpha_1,\beta_1,\dots,\alpha_g,\beta_g,\gamma_1,\gamma_2\}$ be a generating set of $\pi_1(S_{g,2})$, where $\{\alpha_i,\beta_i\}$ is a pair of handle-generators of the $i$-th handle, where $1\leq i\leq g$. For any $i=1,\dots,g$ there is an injection $\jmath_i:\langle\alpha_i,\beta_i\rangle\to \pi_1(S_{g,2})$ and define $\rho_i$ as $\rho\circ\jmath_i$. By Lemma \ref{cbas2}, we assume that each of these affine maps is non-trivial, \textit{i.e.} not the identity map.

\smallskip

\noindent Let $p\in \cp\setminus \big\{\,\text{Fix}(A)\,|\, A\in \text{Im}(\rho)\,\big\}$ be any point. For each $i$, let $\Sigma_i$ be the affine surface homeomorphic to the punctured torus, with monodromy $\rho_i$, obtained from the construction above based at $p$. In order to glue these $g$ affine surfaces together we have to find $g$ rays, one for each $\Sigma_i$, that all develop on the same ray on $\mathbb{C}$. For this, we employ the construction in Definition \ref{glue-new}, which we now spell out in more detail. 

\noindent Recall that $\mathbb{C}$ on its own right can be regarded as an affine structure with trivial monodromy on a disk. We single out on such a structure a ray, say $r_0$, leaving from the point $p$ above towards the infinity. The pair $(\C,\,r_0)$ is an affine structure with a marked $r_0$. On each surface $\Sigma_i$, we choose an infinite ray $r_i$ from the unique branch-point to the puncture at infinity and such that it develops on $\mathbb{C}$ along a ray, say $\overline{r}_i$ leaving from $p$ towards the infinity. Notice that a copy of the ray $\overline{r}_i$ is contained even in $(\C,\,r_0)$. For any $i$, we glue together the affine surfaces $\Sigma_i$ and $(\C,\,r_0)$ along the rays $r_0$ and $\overline{r}_i$ according to our Definition \ref{glue2}. The resulting surface is still homeomorphic to $S_{1,1}$ but it carries a new affine structure $\Sigma_i'$ with monodromy $\rho_i$. Moreover, on each surface $\Sigma_i'$ we can single out a copy of the ray $r_0$. 

\noindent We can now glue together affine surfaces $\Sigma_1',\ldots \Sigma_g'$ along these copies of $r_0$ (again according to our Definition \ref{glue2}) to obtain an affine surface $\Sigma$. This surface $\Sigma$ is homeomorphic to $S_{g,1}$ and has a unique branch-point $q$ which develops on $p\in\C$ and it is the starting-points of the rays, after the identifications. Removing $q$ from $\Sigma$, we obtain the desired surface homeomorphic to $S_{g,2}$ and equipped with an affine structure with monodromy $\rho$.

\smallskip

\noindent Let us finally assume $\text{Im}(\rho)\cong\mathbb{Z}_2$. Notice that we can always find a basis $\alpha_1,\beta_1,\dots,\alpha_g,\beta_g$ such that 
\begin{equation}
    \rho(\alpha_i)=1 \,\,\text{ and } \,\, \rho(\beta_i)=-1\,\, \text{ for any }1\le i \le g.
\end{equation}

\noindent Even in this case the proof is based on an inductive process, therefore we start with the case $g=1$. There exists a half-translation structure on $\cp$, see Definition \ref{def_halftrans}, associated to a meromorphic quadratic differential $\phi$ having one zero of order $2$ at $0$, two poles of order $-1$ at $\pm1$ and, finally, one pole of order $-4$ at $\infty\in\cp$, see Definition \ref{def_halftrans} and subsequent remark. Recall that, given a quadratic differential, a pole of order one corresponds to a cone point of angle $\pi$ and a zero of order two corresponds to a branch-point of magnitude $4\pi$. It is possible to verify that such a structure has non-trivial holonomy given by a representation $\chi:\pi_1\big(\cp\setminus\{\pm1\}\big)\longrightarrow \mathbb{Z}_2\cong\{ z\to \pm z\}$.
We now make use of this structure for realizing a half-translation structure $\Sigma$ with the desired holonomy. Take two copies of the structure $(\cp,\,\phi)$ and slit both along the segment $e_i=[-1,0]$ and the infinite ray $r_i=[1,\infty]$, with $i=1,2$. Denote the resulting sides $e_i^{\pm}$ and $r_i^{\pm}$. We define $\Sigma$ to be the half-translation structure on a torus obtained by identifying $e_1^+$ with $e_2^+$, $e_1^-$ with $e_2^-$ and, in the same fashion, $r_1^+$ with $r_2^+$, $r_1^-$ with $r_2^-$. Such a structure is naturally associated to a meromorphic quadratic differential $q$ having one zero and one pole of order $6$. By removing the singularities of $q$ we obtain an affine structure on $S_{1,2}$ with the desired monodromy. The general case $g\ge2$ now comes as follows. From our construction, there always exists an infinite ray $\overline{r}\subset \Sigma$ joining the two punctures. Then consider $g$ copies of $\Sigma$ slit along $\overline{r}$ and glue along rays (as in Definition \ref{glue-new}) in succession. The resulting surface is homeomorphic to $S_{g,2}$ and carries a half-translation structure with holonomy $\rho$ as desired.
\end{proof}

\noindent Using the previous two propositions, together the gluing construction as in Definition \ref{glue2}, we can now prove the analogue of Proposition \ref{trans}:

\begin{prop}\label{aff}  Let $g>0$ and $k>2$, and  let $\rho: \pi_1(S_{g,k}) \to \affc$ be a non-trivial representation such that  there is at least one puncture with trivial monodromy. Then there is an affine structure on $S_{g,k}$ with monodromy $\rho$, obtained by puncturing an affine surface $\Sigma$  with a unique branch-point. 
\end{prop} 

\begin{proof}
We shall follow the strategy of the proof of Proposition \ref{trans}. Our construction of the affine surface $\Sigma$ will split into two cases.  \\

\noindent \textit{Case 1. The representation $\rho$ has at least two  punctures with trivial monodromy.} Let us define $\rho_0:\pi_1(S_{g,2}) \to \affc$ as the restriction of $\rho$ to a subsurface of $S_{g,k}$ homeomorphic to $S_{g,2}$ that contains all the handles and one puncture with trivial monodromy. Let $\rho_1:\pi_1(S_{0,k}) \to \affc$ be the restriction of $\rho$ to the complementary subsurface, that contains all the other punctures. Note that by our assumption $\rho_1$ has at least one puncture with trivial monodromy. There are two sub-cases:

\smallskip

\noindent \textit{Sub-case (i). The representation $\rho_0$ is non-trivial.} We shall build $\Sigma$ by gluing together two affine surfaces $\Sigma_0$ and $\Sigma_1$ where 
\begin{itemize}
 
 \item $\Sigma_0$ is homeomorphic to $S_{g,1}$ and has exactly one branch-point $p$, and the monodromy of the affine structure on the surface $\Sigma_0 \setminus \{p\}$ is $\rho_0$. If $\rho_0$ is non-trivial, such an affine surface exists by Proposition \ref{affk2}. 
 \item $\Sigma_1$ is homeomorphic to $S_{0, k-1}$, and has holonomy $\rho_1$ and exactly one branch-point. Such an affine surface exists by Proposition \ref{affg0}. 
\end{itemize} 
 
\noindent Recall that both the affine surfaces $\Sigma_0$ and $\Sigma_1$ depend on the choice of an initial base-point. We choose the same base-point, say $p$ for both structures. It follows from the constructions in the proofs of Propositions \ref{affg0} and \ref{affk2} that the branch-points on both surfaces develop to $p$. For the gluing, we slit along rays, say $r_0$ and $r_1$, on $\Sigma_0$ and $\Sigma_1$ respectively, from the branch-point to a puncture at infinity. These rays may develop onto different rays on $\C$ with the same starting-point $p$. We glue along these rays as in Definition \ref{glue-new} to obtain a surface $\Sigma$ homeomorphic to $S_{g,k-1}$ and having a single branch-point. Removing the branch-point, we obtain an affine surface homeomorphic to $S_{g,k}$ that has holonomy $\rho$.

\smallskip

\noindent \textit{Sub-case (ii). The representation $\rho_0$ is trivial.}  Recall that at least two punctures are trivial; let $A_3,A_4,\ldots , A_k$ be the affine maps that are the monodromy around the remaining punctures. We can exclude the case that $k=3$ here, since then the triviality of $\rho_0$ would imply that $\rho$ is trivial, contradicting our assumption. According to our Lemma \ref{triv}, there is a branched projective structure on $S_g$ with three branch-points one of which develops at $\infty\in\cp$. We first construct a (branched) affine surface $\Sigma_0$ homeomorphic to $S_{g,1}$ with two branch-points by removing the branch-point at infinity. Let $r_0$ be a ray starting from one of the branch-points to the puncture at infinity and let $\overline{r}_0$ its developed image on $\C$. It is an infinite ray leaving from a point $p\in\mathbb{C}$.
Also, construct an affine surface $\Sigma_1$ homeomorphic to $S_{0,k-2}$ with exactly one branch-point, such that the monodromy around the punctures are $A_3,\ldots A_k$; such a surface exists by Proposition \ref{affg0}.  Recall the construction is subject to the choice of a base-point. By choosing $p$ as the base-point, it follows by construction that the branch-point of $\Sigma_1$ develops at $p$. Let $r_1\subset\Sigma_1$ be any ray from the branch-point to a puncture at infinity and let $\overline{r}_1$ be the developed ray leaving from $p$. As before, we now glue preserving holonomy, as in Definition \ref{glue-new}. Namely, we slit $\Sigma_1$ along $r_1$ and then glue a copy of the marked affine structure $(\C, \overline{r}_0)$ slit along $\overline{r}_1$. Notice that the gluing is possible because $r_1$ develops on $\overline{r}_1\subset\C$ by construction. The resulting surface is homeomorphic to $S_{0,k-2}$ but carries a new branched affine structure $\Sigma_1'$ containing a whole copy of $\C$ with the embedded ray $\overline{r}_0$. We slit $\Sigma_0$ along $r_0$ and $\Sigma_1'$ along $\overline{r}_0$ and then we identify the resulting boundary rays cross-wise to obtain an affine surface $\Sigma$ homeomorphic to $S_{g,k-2}$ and two branch-points. Removing the branch-points we obtain the desired affine structure on $S_{g,k}$ with monodromy $\rho$.

\smallskip

\noindent \textit{Case 2. The representation $\rho$ has exactly one puncture with trivial monodromy.} Consider the subsurface of $S_{g,k}$ that contains all the handles, and the puncture with trivial monodromy; note that such a surface is homeomorphic to $S_{g,2}$. Let $\rho_0:\pi_1(S_{g,2}) \to \affc $ be the restriction of $\rho$ to that surface. Note that $\rho_0$ has trivial monodromy for one of the punctures, and the other puncture has monodromy $C$ that is the product of the commutators of the holonomies around the handle-generators, for each handle. If this product is the identity map, then we can use the same constructions as in \textit{Case 1} to finish the construction of the desired affine surface $\Sigma$. In what follows, we shall assume that $C$ is not the identity element (and is therefore some non-trivial translation).

\begin{figure}
  \centering
  \includegraphics[scale=0.4]{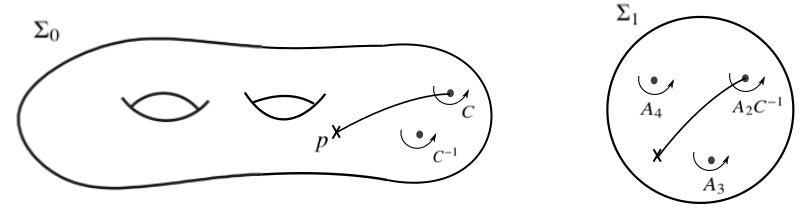}
  \caption{The construction of an affine surface  $\Sigma$ by gluing $\Sigma_0$ and $\Sigma_1$ along rays,  in Case 2 in the proof of Proposition \ref{aff}. }
\end{figure}

\noindent Either using Proposition \ref{affk2} if $\rho_0$ is non-trivial, or \textit{Case 2} of the proof of Proposition \ref{trans} if $\rho_0$ is trivial, we can then build an affine surface $\Sigma_0$ such that 
\begin{itemize}
    \item[(a)] it is homeomorphic to $S_{g,1}$ and a single branch-point, say $q$, 
    \item[(b)] on removing the branch-point, the affine structure on $\Sigma\setminus \{q\}$ has monodromy $\rho_0$, and
    \item[(c)] there is a ray $r_0$ from $q$ to the puncture at infinity which has monodromy $C$ by construction. This ray develops on an infinite ray $\overline{r}_0$ leaving from a point $p\in\C$.
\end{itemize}
 Let the monodromy around the remaining punctures be $A_2,A_3,\ldots A_k$. Given $p$ as above, by Proposition \ref{affg0} there is an affine surface $\Sigma_1$ that is homeomorphic to $S_{0,k-1}$ such that
 \begin{itemize}
    \item[(a)] there is exactly one branch-point which develops to the point $p$, and
    \item[(b)] the monodromy around the punctures are $A_2C^{-1}, A_3, A_4,\ldots, A_k$. 
\end{itemize}

\noindent Let $r_1$ be a ray from the branch-point to the puncture with holonomy $A_2C^{-1}$ and let $\overline{r}_1$ its developed image. We glue $\Sigma_0$ and $\Sigma_1$ along the rays $r_0$ and $r_1$, as in Definition \ref{glue-new}.  Recall that in that gluing preserving holonomy, we in fact first attach copies of the affine surface $\C$ to $r_0$ and $r_1$ respectively, and then glue along the same ray $r_\star$ in these copies via the identity map. The resulting affine surface $\Sigma$ is homeomorphic to $S_{g,k-1}$ and has one branch-point $p$ where the starting points of the rays get identified. Since the final gluing (along the ray $r_\star$) is by the identity map,  the other endpoints of the rays get identified to a puncture with holonomy $A_2 C^{-1} \cdot C = A_2$  (\textit{c.f.} the discussion just before Definition \ref{glue-new}). Hence the monodromy of the affine structure on $\Sigma \setminus \{q\} $ is precisely $\rho$, as desired. \qedhere
\end{proof}

\section{Affine holonomy and a single puncture}\label{saff2}

\noindent In this section we deal with the case when the representation $\rho$ is into the affine group $\text{Aff}(\C)$, as in the previous section, for once-punctured surfaces of positive genus, that is $k=1$ and $g>0$. For this, we need to modify the construction in Proposition \ref{affk2} such that the "puncture at infinity"  for $\Sigma$ is a regular point when viewed as a \textit{projective} structure. We can then "fill in" that puncture to obtain a surface with equipped with a projective structure (away from a single branch-point), like we did in the proof of Lemma \ref{transurfk1}. 

\subsection{Necessary conditions} We start by showing the necessity of assuming the image of $\rho$ is not a finite group of order two in Theorem \ref{thm1}. 

\begin{lem}\label{excase}
Let $\rho:\pi_1(S_{g,1}) \to \affc$ be a non-trivial representation such that the puncture has trivial monodromy and the image of $\rho$ is finite of order two. Then $\rho$ does not appear as the monodromy of any projective structure $S_{g,1}$.
\end{lem}

\begin{proof}[Proof of Lemma \ref{excase}]
Let $G=\text{ker}(\,\rho\,)$ and let $\widehat{S}_{g,1}$ be the covering of $S_{g,1}$ associated to $G$. The group $G$ is a subgroup of $\pi_1(S_{g,1})$ of index two and hence the covering map $f:\widehat{S}_{g,1}\longrightarrow S_{g,1}$ turns out a Galois covering map of degree two. In particular, $\widehat{S}_{g,1}$ is homeomorphic to $S_{2g-1,2}$. Let us now assume the existence of a complex projective structure on $S_{g,1}$. Then, we may lift this structure to a complex projective structure on $S_{2g-1,2}$ with monodromy determined by the composition $\rho\circ f_*$, where $f_*:\pi_1(S_{2g-1,2})\to \pi_1(S_{g,1})$. Since the image of $f_*$ is nothing but $\text{ker}(\,\rho\,)$, the representation $\rho\circ f_*$ is just the trivial one. Therefore, by our Lemma \ref{trivex}, such a structure does not exist and, in turn, there is no complex projective structure on $S_{g,1}$ with monodromy $\rho$.
\end{proof}

\subsection{Once-punctured translation surfaces}
The case of translation structures on once-punctured surfaces is actually subsumed by the construction in the proof of Proposition \ref{trans1}, provided we only require a projective structure, and not a translation structure, on the surface. 


\begin{prop}\label{transurfk1} Let $S_{g,1}$ be a surface of genus $g>0$ and exactly one puncture, and let $\Gamma_{g,1}$ be its first homology group. Any non-trivial representation $\chi:\Gamma_{g,1} \to \C$ is the monodromy of some projective structure on $S_{g,1}$. 
\end{prop} 

\begin{proof}
From \eqref{sumC}, we know that the puncture must have trivial monodromy. We can then construct a translation surface $\Sigma$ homeomorphic to $S_{g,1}$ exactly as in the proof of Proposition \ref{trans1}, that has holonomy $\chi$, one branch-point $p$ and one pole of order two. Note that a pole of order two is the point at $\infty$ in a standard planar end of $\C$; thus in particular, $\infty$ is a regular point of $\cp$. Thus, we can consider $\widehat{\Sigma} = \Sigma \cup \{\infty\}$ to be a surface equipped with a projective structure, with exactly one branch-point (namely, $p$); the surface $\widehat{\Sigma}\setminus \{p\}$ is then the desired surface homeomorphic to $S_{g,1}$ equipped with a projective structure having monodromy $\chi$.   \end{proof}

\textit{Remark.} If a non-trivial representation  $\chi:\Gamma_{g,1} \to \C$ is the monodromy of a \textit{translation} structure, then the corresponding abelian differential $\omega$ must extend to an abelian differential with exactly one zero on the closed surface $S_g$.   The recent work of \cite{Fils} and \cite{BJJP} generalizing Haupt's theorem (see \cite{Haupt}, \cite{Kapovich})   provides necessary and sufficient conditions  on $\chi$ for the existence of such a structure.

\subsection{Once-punctured affine torus} 
For the once-punctured torus, the problem of finding an projective structure with prescribed affine holonomy is handled by the following result: 

\begin{prop}\label{affg1k1} Let $\rho: \pi_1(S_{1,1}) \to \affc$ be a non-trivial representation such that the puncture has trivial monodromy. Assume $\rho\big(\pi_1(S_{1,1})\big)$ is not finite of order two. Then there is a projective structure on $S_{1,1}$ with monodromy $\rho$.
\end{prop} 


\begin{proof}[Proof of Proposition \ref{affg1k1}] Let $\rho:\pi_1(S_{1,1})\longrightarrow \affc$ be a non-trivial affine representation, let $\alpha$ and $\beta$ denote two handle-generators and let $\gamma=[\alpha,\beta]$ be a curve enclosing the puncture. The monodromy around the puncture is assumed to be trivial, \emph{i.e.} $\rho(\gamma)=I$. This implies in particular that $\rho(\alpha)=A$ and $\rho(\beta)=B$ commute and hence the representation $\rho$ is abelian. Up to conjugation, we may assume without loss of generality that:
\begin{equation}\label{standform}
A=\begin{pmatrix}
a & 0\\
0& 1
\end{pmatrix} \qquad
B=\begin{pmatrix}
b & 0\\
0& 1
\end{pmatrix}
\end{equation}
where $a,b\notin\{0,1\}$. In fact, $a$ and $b$ cannot be both equal to one as the representation is assumed to be non-trivial and, whenever $A$ or $B$ is the identity matrix, a suitable change of basis put the matrices in the desired form. Given any point $p_0\in\C$ we define the points $p_i$, where $i=1,\dots,3$ as follows: $p_1=A(p_0)$, $p_2=AB(p_0)=BA(p_0)$, and finally $p_3=B(p_0)$. The polygon
\begin{equation}\label{pol}
 p_0\mapsto p_1\mapsto p_2\mapsto p_3\mapsto p_0
\end{equation} bounds a possibly self-intersecting and possibly degenerate quadrilateral $\mathcal{Q}$ on the complex plane. As already done before, we shall denote the directed edges as follows: $e_1=\overline{p_1\,p_2}$, $e_2=\overline{p_0\,p_1}$, $e_3=\overline{p_0\,p_3}$ and, finally, $e_4=\overline{p_3\,p_2}$. The edges of this polygon are related by the maps $A,B$ as follows: $A(e_3)=e_1$ and $B(e_2)=e_4$. Given the matrices $A$ and $B$ as in the equation \eqref{standform}, it is convenient to choose $p_0=1$. As a consequence $p_1=a$, $p_2=ab$ and $p_3=b$. We can notice that the points $\{1,a,b,ab\}\subset \C$ are all aligned if and only if they are all real and the quadrilateral $\mathcal{Q}$ degenerates to a segment. According to this property, we shall divide the discussion in two cases. \\

\noindent \textit{Case 1. The points $1,a,b,ab$ are not reals.} In this case, the points $1,a,b,ab$ are the vertices of some possibly self-intersecting quadrilateral $\mathcal{Q}$. As done before in Proposition \ref{trans1}, we can choose a collection of infinite rays $\mathcal{R} = \{r_0, r_1, r_2, r_3\}$ with starting points at the vertices $p_i$ of $\mathcal{Q}$ and consider embedded region $R_i$, for each $i \in \{0, \dots , 3\}$, bounded by the segment $e_i$ and two infinite rays from the collection $\mathcal{R}$. Even in this case there are two choices of each such a region, since the union of $e_i$ and the ray from its endpoints separates the complex plane; we choose the one that results in the correct orientation of the handle-generators $\alpha,\beta$ in the affine surface $\Sigma$ that we shall define below, see figure \ref{leftrightchoice}. Each region $R_i$ has one ideal vertex at $\infty\in\cp$ and the union of the regions  determines an immersed disc $R$ on the Riemann sphere with boundary $\partial R=\overline{e_1}\cup\overline{e_2}\cup e_3\cup e_4$. We define $\Sigma$ to be quotient of the region $R$ by identifying the boundary segments $e_1$ and $e_3$ via the affine map $A$, and the segments $e_2$ and $e_4$ via the affine map $B$. The resulting surface is homeomorphic to a punctured torus endowed with an affine structure on a punctured torus with one branch-point of magnitude $6\pi$ and one pole of order two. We can fill up the puncture by adding a complex projective chart locally modelled at $\infty\in\cp$ and eventually remove the (only) branch-point. The final surface is a punctured torus endowed with a complex structure - but not affine - having monodromy $\rho$.

\medskip 

\noindent \textit{Case 2. The complex numbers $1,a,b,ab$ are reals.} In this case the four points ${1,a,b,ab}$ are aligned, and a similar construction works. Recall that, in this case $a,b\notin\{\pm1\}$ in the light of Lemma \ref{excase} above. The main different from the case (1) is that the quadrilateral $\mathcal Q$ degenerates to a segment on the real axis. Whenever either $a$ or $b$ is greater than zero, then we can still find a collection of rays and regions $R_i$ with the desired properties and thence one can proceed as above in a similar fashion. However, when both $a,b$ are negative it turns out to be impossible to find out rays and regions as desired regardless of the choice of the base-point $p_0$. In this case, we first need to change the handle-generators in order to make either $a$ or $b$ a positive real. For instance we may replace $\{\alpha,\beta\}$ with $\{\alpha, \alpha\beta\}$. Then we can proceed as above.
\end{proof}

\textit{Remark.} Here is a construction, inspired by \cite[Lemma 2.2]{Mondello-Panov2}, of a projective structure (in fact, a spherical structure) on $S_{1,1}$ such that the image of the monodromy representation  is a finite cyclic group of order $k\geq 3$.  Let $C$ be a great circle in $\cp$ and let $\alpha$ be the ``orthogonal" geodesic line in $\mathbb{H}^3$ (thought of as the unit ball, with $\partial_\infty \mathbb{H}^3 = \cp$) passing through the origin. On $C$ we can single out two adjacent segments, say $l_1,\,l_2$, each of length $\frac{2\pi}{k}$ in the spherical metric. Of course, $l_1$ and $l_2$ are related by the elliptic element $E$ that is  a rotation of angle $\frac{2\pi}{k}$ around the axis $\alpha$. Slit $\cp$ along $l_1,l_2$. The resulting space is a bigon with two vertices each of angle $2\pi$. Then re-glue $l_1^+$ with $l_2^-$ and $l_1^-$ with $l_2^+$. The final surface is a torus equipped with a spherical structure and a single branch-point of angle $6\pi$. By deleting the branch-point we end up with the desired structure on $S_{1,1}$ having the desired monodromy, since the monodromy of each handle-generator is $E^{\pm 1}$.

\subsection{Higher genus affine surfaces}\label{hgask1} Let us finally consider the general case of punctured surfaces with genus $g\ge2$.  Our goal would be to realize the given representation to the affine group as the monodromy of some branched projective structure with one single branch-point. By deleting such a point, we end up with a complex projective structure on $S_{g,1}$ as desired.  Note that although the monodromy is into the affine group $\affc$, the projective structure obtained might not be an affine structure. Namely, we prove the following:

\begin{prop}\label{affk1} Let $g\ge2$ and let $\rho: \pi_1(S_{g,1}) \to \affc$ be a non-trivial representation such that the puncture has trivial monodromy. Assume $\rho\big(\pi_1(S_{g,1})\big)$ is not finite of order two. Then there is a $\cp$-structure on $S_{g,1}$ with monodromy $\rho$. 
\end{prop}

\noindent Our proof shall deal with the co-axial case and non co-axial case separately. In the final subsection we also provide an alternative proof using Le Fils' results from \cite{Fils2}. Before moving to the proof of Proposition \ref{affk1}, we shall need some technical results.
 
\smallskip 
 
\subsubsection{Some technical lemmata} In order to state and prove those lemmata we shall need, we begin by introducing the following definitions.

\begin{defn}[Unitary part and linear part]\label{up} Given a co-axial representation $\rho$, its \textit{unitary part} $\rho_u:\Pi \to U(1)$ is defined by $\rho_u(\gamma) = \exp{(i\arg({\rho(\gamma)}))}$, for each $\gamma \in \Pi$. Note that if $\rho$ is unitary then $\rho = \rho_u$. This notion easily extends to any affine representation as follows. In fact, there is a natural projection $\text{Li}:\affc\longrightarrow \C^*$ that associates to any mapping $A(z)=az+b$ its linear part, \textit{i.e.} $\text{Li}(A)=az$. Notice that, if $\rho$ is co-axial then $\text{Li}\circ\rho=\rho$. The \textit{unitary part} of a generic representation $\rho$ is a representation $\rho_u:\pi_1(S_{1,1})\longrightarrow U(1)$ defined as $\rho_u(\gamma)=\exp\big(i\arg(\text{Li}\circ\rho(\gamma))\big)$.
\end{defn}


\begin{defn}\label{irrhand} Given a co-axial representation $\rho$, a handle on $S_{g,1}$ generated by a pair $\{\alpha, \beta\}$ of simple closed curves intersecting only once will be called \textit{rational} if $\rho_u(\alpha)$ and $\rho_u(\beta)$ generate a discrete subgroup of $U(1)$. Alternatively, if the dilation factors of $\rho(\alpha)$ and $\rho(\beta)$ are $a$ and $b$ respectively, then the handle is rational if their arguments $\arg{a}, \arg{b}\in 2\pi\mathbb{Q}$ (the \textit{dilation factor} of an affine map $A(z) = a z$ is $a \in \C^\ast$). We will say the handle is \textit{irrational} if it is not rational.
\end{defn}

\noindent The following lemma concerns affine representation with dense unitary part.

\begin{lem}\label{coax-lem3} 
Let $\rho:\Pi \longrightarrow \text{Aff}(\C)$ be an affine representation such that $\textnormal{Li}\circ\rho$ is not unitary. Then there exists handle-generators $\{\alpha_j,\beta_j\}_{1\leq j\leq g}$ on $S_{g,1}$ such that  $\lvert a_j\rvert, \lvert b_j \rvert > 1$, where $a_j,b_j$ are the dilation factors of $\rho(\alpha_j), \rho(\beta_j)$ respectively. Moreover, in the case the unitary part $\rho_u$ has a dense image in $U(1)$, we can also ensure that $\arg{a_j}, \arg{b_j} \notin 2\pi \mathbb{Q}$ for each $1\leq j\leq g$ and, for any $\epsilon>0$, we can choose a set of handle-generators that satisfy, in addition to the above properties, $\lvert \arg{a_j} \rvert,  \lvert \arg{b_j} \rvert <\epsilon$  for each $j$. 
\end{lem}

\begin{proof}
First of all we notice that it is sufficient to prove the lemma for co-axial representations. In fact, the general case follows by replacing $\rho$ with $\text{Li}\circ\rho$. Choose an initial set of handle-generators $\{\alpha_j, \beta_j\}_{1\leq j\leq g}$; in the following argument whenever we modify this set of generators via a mapping class, we shall rename and continue denoting the resulting set by the same notation (\textit{viz.} $\alpha_j, \beta_j$). We shall also denote by $a_j$ and $b_j$ the dilation factors of $\rho(\alpha_j)$ and $\rho(\beta_j)$ respectively. We shall first show that we can choose handle-generators such that $\lvert a_j \rvert \neq 1$ and $\lvert b_j \rvert \neq 1$ for each $j$. 

\smallskip

\noindent Since $\rho$ is not unitary, it follows that $\lvert a_r \rvert \neq 1 $ for some $r$. Note that if $\lvert b_r \rvert \neq 1$ instead, we can interchange the handle-generators via the mapping class that takes the pair $\{\alpha_r, \beta_r\} \mapsto \{\beta_r, \alpha_r^{-1}\}$. Now if $\lvert b_r \rvert =1 $ we can change this pair via the mapping class that takes $\{\alpha_{r}, \beta_{r}\} \mapsto \{\alpha_{r}, \alpha_{r}\, \beta_{r}\}$; such a mapping class is supported on that handle, and Dehn-twists around $\alpha_{r}$. This makes the modulus of the dilation factor of the $\rho$-image of the second generator also different from $1$. If $\lvert a_s \rvert =1$ for some other index $s$, then we change the two handles (the $r$-th and $s$-th) via the mapping class  $\phi$  that takes  $\{\alpha_r, \beta_r\} \mapsto \{\alpha_r, \beta_r\,\beta_s\}$ and $\{\alpha_{s}, \beta_{s}\} \mapsto \{\alpha_r^{-1} \alpha_s,  \beta_{s}\}$, see Figure \eqref{changeofbasis}. In this way we can make sure that the modulus of the dilation factor of the first generator of the second handle is $\lvert a_r^{-1} a_s \rvert \neq 1$,  and so we can continue the process as above, until all handle-generators have their corresponding dilation factors not equal to $1$.

\begin{figure}
  \centering
  \includegraphics[scale=0.35]{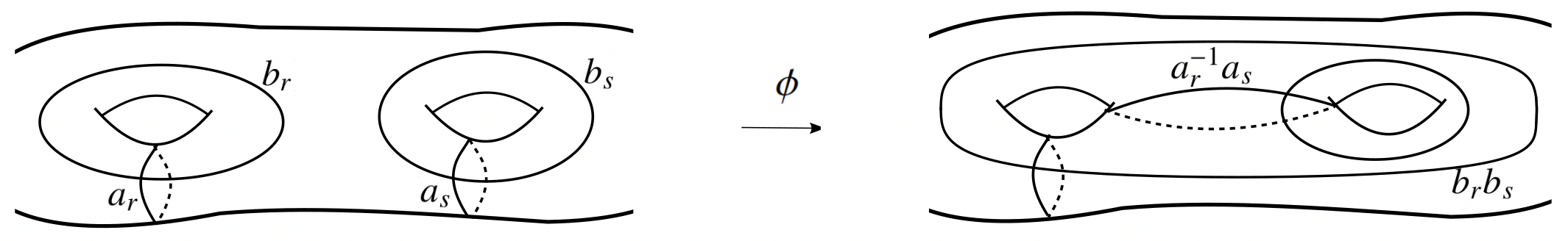}
  \caption{The mapping class $\phi$ changes the pairs of handle-generators  $\{\alpha_r, \beta_r\}$ and $\{\alpha_{s}, \beta_{s}\}$ to $ \{\alpha_r, \beta_r\,\beta_s\}$ and $ \{\alpha_r^{-1} \alpha_s,  \beta_{s}\}$ respectively.}
  \label{changeofbasis}
\end{figure}

\smallskip

\noindent To ensure the dilation factors are strictly greater than $1$ in modulus, we perform the following modifications:

\smallskip

\begin{claim}\label{claim1}
\textit{For any handle, there is a change of generators by a mapping class such that their dilation factors satisfy $\lvert a \rvert >1$ and $\lvert b \rvert >1$.}
\end{claim}

\begin{proof}[Proof of Claim \ref{claim1}]
Assume that $\lvert b \rvert <1$. Recall that we have already ensured above that $\lvert a\rvert \neq 1$. We can use the change of generators $(A,B) \mapsto (A, A^nB)$ for $n\in \mathbb{Z}$. that is effected by a (power of a) Dehn-twist around the handle-generator corresponding to $A$. Note that the dilation factor of $A^nB$ is $a^n b$. Thus for a suitable choice of sign of $n$, and for $\lvert n \rvert$ large enough, the dilation factor of $A^nB$ is strictly greater than $1$ in modulus.  Now the second generator $B^\prime = A^nB$ has the desired property. If the first generator (which remains unchanged) still has $\lvert a \rvert <1$, then we perform the change of generators $(A,B^\prime ) \mapsto (B^\prime, A^{-1})$, which is again effected by an element of $\text{SL}(2,\mathbb{Z})$ (and hence a mapping class).
\end{proof}

\smallskip

\noindent Now assume that the image of the unitary part $\rho_u$ is dense in  $U(1)$. We shall perform a change of generators exactly as in the first part of the proof, so that the arguments of the dilation factors are all irrational; we shall only observe that these modifications do not change the property that the dilation factors are greater than $1$ in modulus.

\noindent Assume there is some $r\in \{1,2,\ldots, g\}$ such that $\arg{a_r} \notin 2\pi \mathbb{Q}$; if instead there is some $r$ such that $\arg{b_r} \notin 2\pi \mathbb{Q}$, then we can switch the roles of $\alpha_r$ and $\beta_r$ in what follows (\textit{e.g.} instead of Dehn-twists around $\alpha_r$ we perform Dehn-twists around $\beta_r$). If $\arg{b_r} \in 2\pi \mathbb{Q}$ we can change this pair of handle-generators to $\{\alpha_{r}, \alpha_{r}^n\,\beta_{r}\}$ for any $n\in \mathbb{Z}$. It is easy to see that for any $n> 0$, the resulting new handle-generator will satisfy $\arg{b_r} \notin 2\pi \mathbb{Q}$, and since $\lvert a_r\rvert > 1$, the dilation factor  $\lvert b_r \rvert >1$ also. 
Now let $\arg{a_s} \in 2\pi \mathbb{Q}$ for some $s$; as before, we use the mapping class $\phi$ to change the two pairs of generators $\{\alpha_r, \beta_r\} \mapsto \{\alpha_r, \beta_r\,\beta_s\}$ and $\{\alpha_{s}, \beta_{s}\} \mapsto \{\alpha_r^{-1}\,\alpha_s, \beta_{s}\}$. Again, we rename these new pairs as $\{\alpha_r,\beta_r\}$ and $\{\alpha_s,\beta_s\}$ respectively. The new generator of the latter handle now has $\arg{a_s} \notin 2\pi \mathbb{Q}$.
Note that by the change $ \{\alpha_s, \beta_s\} \mapsto \{\alpha_s\,\beta_s^n, \beta_s\}$ for $n\gg 0$ (achieved by Dehn-twists along $\beta_s$ on the $s$-th handle), we could have arranged that prior to acting by $\phi$, the modulus of the dilation factor $\lvert a_s \rvert \gg \lvert a_r\rvert $, so that after acting by $\phi$, the dilation factor still satisfies $\lvert a_s \rvert >1$. 

\smallskip

\noindent Finally, fix $\epsilon>0$. We shall show that for $j$-th handle for any $1\leq j\leq g$, there is a change of generators by a mapping class supported on the handle, such for the resulting pair of generators we have $\lvert \arg{a_j} \rvert, \lvert \arg{b_j} \rvert < \epsilon$. Indeed, we can perform Dehn-twists as usual to change the handle-generators to $ \{\alpha_{j}, \alpha_{j}^n\,\beta_{j}\}$, for any $n\in \mathbb{Z}$. As before, for any $n>0$ the dilation factors remain greater than $1$ in modulus. Although the new argument could lie in $2\pi \mathbb{Q}$ for some integer, say $N$, it cannot be in $2\pi \mathbb{Q}$ for any $n\neq N$. (If $N\arg{a_j} + \arg{b_j} \in 2\pi \mathbb{Q}$ and $M \arg{a_j} + \arg{b_j} \in 2\pi \mathbb{Q}$ for $N\neq M$ then $(N-M) a_j \in 2\pi \mathbb{Q}$ which is a contradiction.) 
Since $\rho(\alpha_r)$ is an irrational rotation of the circle, we can choose $n>N$ such that $\lvert \arg{b_j}  \rvert < \epsilon$.  
Similarly, we perform a power of a Dehn-twist around $b_j$, to ensure that $\lvert \arg{a_j}  \rvert < \epsilon$.
\end{proof}

\textit{Remark.} Let $\rho$ be an \textit{Euclidean} representation, namely an affine representation with unitary linear part, that is $\rho_u=\text{Li}\circ\rho$. Assume the image of $\rho_u$ to be dense in $U(1)$. It worth noticing that, although the first claim of Lemma \ref{coax-lem3} never holds for Euclidean representations, it is still possible to find a basis of handle generators such that the linear parts of the $\rho$-images have arbitrarily small argument.

\smallskip

\begin{cor}\label{approxlem}
Let $\rho:\Pi \to \text{Aff}(\C)$ be an affine representation. If the unitary part $\rho_u$ has a dense image in $\text{U}(1)$, then for any $\epsilon>0$ there exists handle-generators $\{\alpha_j,\beta_j\}_{1\leq j\leq g}$ on $S_{g,1}$ such that the inequalities $-\epsilon<\arg b_j<0$ and $0\le \arg a_j +\arg b_j < \epsilon$ hold for each $j$.
\end{cor}

\begin{proof}
The first thing we notice is that the second part of the proof above works \textit{mutatis mutandis} even when $\text{Li}\circ\rho$ is unitary with the only exception being that the dilatation factors are always equal to one. What follows is nothing but a refinement of Lemma \ref{coax-lem3}. Again, we shall assume for simplicity $\rho$ co-axial and the general case comes by replacing $\rho$ with $\text{Li}\circ\rho$. In fact, in the same notation as above, $\rho(\alpha_r)$ is an irrational rotation we can choose $n>N$ such that $\lvert \arg b_j\vert<\epsilon$ and $-\epsilon<\arg b_j< 0$. We now perform Dehn-twist around $b_j$ to ensure $0< -\arg b_j <\arg a_j<\epsilon$. The result follows.
\end{proof}

\noindent We now consider affine representations whose unitary part is discrete in $U(1)$.

\begin{lem}\label{coax-lem1} 
Let $\rho:\Pi \to \text{Aff}(\C)$ be an affine representation such that its unitary part $\rho_u$ has a discrete image in $U(1)$. Then
there exists handle-generators $\{\alpha_j,\beta_j\}_{1\leq j\leq g}$ on $S_{g,1}$ such that 
\begin{equation}
    \rho_u(\alpha_j) = \exp(2\pi\,i/m)\,\, \text{ and } \,\,\rho_u(\beta_j) = 1
\end{equation}
 for each $j$, where $\rho_u(\Pi) \cong \mathbb{Z}_m$. 
 In fact, for any surjective homomorphism $h:\Pi \to \mathbb{Z}_m$ we can find handle-generators such that $\rho_u(\alpha_j) = h(\alpha_j)$ and $\rho_u(\beta_j) = h(\beta_j)$ for each $j$. 
\end{lem}

\begin{proof}
Since $U(1)$ is abelian, $\rho_u$ factors through the homology group $\Gamma = H_1(S_{g,1}, \mathbb{Z})$. Fix a set of handle-generators $\{\alpha^\prime_j,\beta^\prime_j\}_{1\leq j\leq g}$; this is also a set of generators of $\Gamma$. Since the image in $U(1)$ is a discrete group, it must be cyclic, say of order $m\geq 3$. Here, the case of $m=2$ is ruled out by our assumption that the image is not of order two (\textit{c.f.} Lemma \ref{excase}). Thus, we can think of the unitary part as a surjective homomorphism $\rho_u:\Gamma \to \mathbb{Z}_m$, where $\mathbb{Z}_m$ is the cyclic subgroup of $U(1)$ generated by $\exp{(2\pi\, i/m)}$.  Then, by \cite[Proposition 3.2]{Fils2}, there is $A \in \text{Sp}(2g,\mathbb{Z})\cong\text{Aut}^+(\Gamma)$ such that $\rho_u \circ A:\Gamma \to \mathbb{Z}_m$  satisfies $\rho_u(\alpha^\prime_j)= 1$ and $\rho_u(\beta^\prime_j)= 0$ for each $1\leq j\leq g$. Indeed, by the same Proposition, for any surjective homomorphism $h_u:\Gamma \to \mathbb{Z}_m$ there exists an automorphism $A$ of $\Gamma$ such that $\rho_u \circ A = h_u$. The automorphism $A$ is induced by a mapping class $\phi:S_{g,1} \to S_{g,1}$, and defining $\alpha_j := \phi(\alpha^\prime_j)$ and $\beta_j: =\phi(\beta^\prime_j)$ then defines our desired set of handle-generators. 
\end{proof}
 
\smallskip

\noindent We finally conclude with a lemma specific to non co-axial representations. We shall make use of the following result in subsection \ref{affk1notcoa}.

\begin{lem}\label{non-coax} Let $\rho: \pi_1(S_{g,1}) \to \affc$ be a non-trivial representation as in the statement of Proposition \ref{affk1}. If $\rho$ is not co-axial, then we can choose pairs of handle-generators $\{\alpha_i, \beta_i\}_{1\leq i\leq g}$ such that their commutators are all non-trivial, i.e.\ $\rho([\alpha_i,\beta_i]) \neq I$ for each $1\leq i\leq g$. 
\end{lem} 

\begin{proof}
An affine map $A(z) = a z + b$ has a fixed point $\frac{b}{1-a} \in \C$, unless $a = 1$, \textit{i.e.} $A$ is a translation, in which case the fixed-point set $\text{Fix}(A) = \emptyset$.  Two affine maps $A$ and $B$ commute if and only if their fixed-point sets are identical. We shall also use the following elementary fact: if $A,B$ are affine maps with different fixed-point sets, then the fixed-point set of $B\circ A$ is different from that of $A$. Therefore, it suffices to show that one can choose each pair $\{\alpha_i,\beta_i\}$ of handle-generators such that their fixed-point sets are not identical; for the purposes of this proof we shall call a pair having this property (and the corresponding handle) "good". 

\smallskip

\noindent We start with some set of handle-generators $\{\alpha_i,\beta_i\}_{1\leq i\leq g}$; note that these $2g$ elements generate the fundamental group $\pi_1(S_{g,1})$. By Lemma \ref{cbas2} we can also assume that $\rho(\alpha_i)$ and $\rho(\beta_i)$ are non-trivial affine maps for each $1\leq i\leq g$. In what follows we shall modify this initial choice of generators by mapping class group elements till the pair of generators of each handle is good. The basic idea of this modification is the following: suppose $\{\alpha_i, \beta_i\}$ is a good pair, and $\{\alpha_j, \beta_j\}$ is not, then we replace these two  pairs of handle-generators by the pairs $\{\alpha_i, \beta_i\,\beta_j\}$ and $\{\alpha_i^{-1}\,\alpha_j, \beta_j\}$. Note that this change of handle-generators is effected by a mapping class $\phi:S_{g,1} \to S_{g,1}$, see Figure \eqref{changeofbasis}. Let $F$ be the common fixed-point set of $\rho(\alpha_j)$ and $\rho(\beta_j)$. We divide into two cases: 

\smallskip

\noindent \textit{Case A: The fixed-point set $F\neq \phi$.}  In this case $F$ is a single point; in what follows we assume that $F=0\in \C$ to simplify our computations because the general case can be reduced to this via a conjugation. Let $z\mapsto a_j z$ and $z\mapsto b_j z$ be the $\rho$-images of the generators $\alpha_j$ and $\beta_j$ respectively, where $a_j, b_j \in \C \setminus \{0,1\}$. We can assume without loss of generality that the fixed-point sets of $\rho(\alpha_i)$ and $\rho(\beta_i)$ are both distinct from $F$, since otherwise, if say $\text{Fix}(\rho(\alpha_i)) =F$ then we can perform a Dehn-twist in that handle around $\beta_i$ to change its generators $\{\alpha_i,\,\beta_i\} \mapsto \{\alpha_i\,\beta_i, \beta_i\}$. By the elementary fact noted above, $\text{Fix}(\rho(\alpha_i\,\beta_i)) \neq F$, so this new pair of generators has the required property. The same fact implies that the new second handle obtained after acting by the mapping class $\phi$, \textit{i.e.}  generated by  $\{\alpha_i^{-1}\,\alpha_j, \beta_j\}$, is good. However, it could still happen that for the new first handle, the $\rho$-images of the generators, namely $\rho(\alpha_i)$ and $\rho(\beta_i\,\beta_j) $, have the same fixed-point set. If $\rho(\alpha_i)$ is the affine map $z\mapsto a_i z + c_i$  and $\rho(\beta_i)$ is the affine map $z\mapsto b_i z + d_i$ then this happens when 


\begin{equation}\label{510-1} 
        \text{Fix}\big(\rho(\alpha_i)\big) = \frac{c_i}{1-a_i} = \frac{d_i}{1 - b_i\,b_j} = \text{Fix}\big(\rho(\beta_i\,\beta_j)\big).
\end{equation}

\noindent In this case, we first change the generators of the second  handle at the very beginning of the construction, by a Dehn twist around $\alpha_j$, namely $\{\alpha_j,\,\beta_j\} \mapsto \{\alpha_j, \alpha_j\,\beta_j\}$. This does not change the property that the fixed-point set of both generators is $F$; however after acting by the mapping class $\phi$, the new two pairs of handle-generators are now $\{\alpha_i,\, \beta_i\,\alpha_j\,\beta_j\}$ and $\{\alpha_i^{-1}\,\alpha_j,\, \alpha_j\,\beta_j\}$. The latter is a good pair for the same reason as before; the former is not a good pair only if
\begin{equation}\label{510-2}
\frac{c_i}{1-a_i} = \frac{d_i}{1-b_i\,a_j\,b_j}
\end{equation}
It is easy to check that \eqref{510-1} and \eqref{510-2} cannot simultaneously hold, since by our assumption $a_j\neq 1$.

\medskip

\noindent \textit{Case B: The fixed-point set $F= \phi$.} In this case both $\rho(\alpha_j)$ and $\rho(\beta_j)$ are translations, say $z\mapsto z + v$ and $z\mapsto z + w$ respectively. As above, let $z\mapsto a_i z + c_i$  and $z\mapsto b_i z + d_i$ be $\rho(\alpha_i)$ and $\rho(\beta_i)$ respectively. Then the pair of generators $\{\alpha_i^{-1}\,\alpha_j, \beta_j\}$ of the second handle is good since one generator maps to a translation, while the other does not. The new pair of generators $\{\alpha_i, \beta_i\,\beta_j\}$ of the first handle is either also good, in which case we are done, or else

\begin{equation}\label{510-3} 
    \text{Fix}(\rho(\alpha_i)) = \frac{c_i}{1-a_i} = \frac{b_i\,w+d_i}{1-b_i} = \text{Fix}(\rho(\beta_i\,\beta_j)).
\end{equation}


\noindent In the latter case, we proceed as in \textit{Case A}, namely, we first replace $\{\alpha_j,\beta_j\}$ with $\{\alpha_j, \alpha_j\,\beta_j\}$ to the second handle, at the beginning of the construction. The new second generator of the second handle is now the translation $z\mapsto z + v+w$, and after acting by the mapping class $\phi$, the new two pairs of handle-generators are $\{\alpha_i,\, \alpha\,\beta_i\,\alpha_j\,\beta_j\}$ and $\{\alpha_i^{-1}\,\alpha_j,\, \alpha_j\,\beta_j\}$. The latter is a good pair for the same reason as before, namely because one generator maps to a translation while the other does not. The first  pair must also be good, because otherwise 
\begin{equation}
    \frac{c_i}{1-a_i} = \frac{b_i\,v+b_i\,w+d_i }{1-b_i}  
\end{equation}
\noindent which contradicts \eqref{510-3} since we know $v \neq 0$ as none of the handle-generators map to the identity element. 

\smallskip


\noindent Thus, if there is one good pair of handle-generators, then we can use the above modification repeatedly to make each handle good. To complete the argument, we need to show that there exists a good handle: For this, note that since $\rho$ is not co-axial, there exists two elements from the initial set of generators that do not have the same fixed-point set. If they are generators for the same handle, then we already have one good pair. If not, suppose they belong to two handles neither of which is good; namely, suppose there are two pairs of handle-generators  $\{\alpha_i,\,\beta_i\}$ and $\{\alpha_j,\, \beta_j\}$ such that the elements in each pair have the same fixed point set, but the fixed-point sets for the pairs are not identical. Then we change the pair of handles by the mapping class $\phi$ exactly as above, namely where the two new pairs of handle-generators are $\{\alpha_i,\, \beta_i\,\beta_j\}$ and $\{\alpha_i^{-1}\,\alpha_j,\, \beta_j\}$. It follows from the elementary fact observed at the beginning of the proof that both of these are now good pairs.  
\end{proof} 

\subsubsection{Proof of Proposition \ref{affk1}: Co-axial representations}\label{affk1coa}

We shall divide this proof into three cases: 
\begin{itemize}
\item[(i)] $\rho$ is not unitary, i.e.\ its image in $\pslc$ does not lie in the circle subgroup \[U(1)= \{\text{diag}(e^{i\theta}, e^{-i\theta})\ \vert\ \theta \in \R\}/\{\pm I\},\]
    \item[(ii)] $\rho$ is unitary, but the image of $\rho$ is dense in $U(1)$, 
    \item[(iii)] $\rho$ is unitary, and the image of $\rho$ is finite, but not of order two.
\end{itemize}
Recall that as in the previous section, we are considering representations from $\Pi = \pi_1(S_{g,1})$.

\medskip 

\noindent\textit{Case (i): Non-unitary case.} First consider the quadrilaterals $\{Q_1, Q_2,\ldots, Q_g\}$ where for each $1\leq i\leq g$, the quadrilateral $Q_i$ is constructed exactly as in Proposition \ref{affg1k1} by taking the $\rho$-images $A_i$ and $B_i$ of the $i$-th handle, such that the base-point $p_i$ of $Q_i$ is also a vertex of $Q_{i-1}$ for each $i\geq 2$. 

\smallskip

\noindent The key idea is that we can do this so that each these quadrilaterals are pairwise disjoint, except for adjacent quadrilaterals which intersect only at a single vertex, \textit{e.g.}\ $Q_{i-1} \cap Q_i = \{p_i\}$. In order to show that we can do this, we consider two sub-cases, involving the unitary part $\rho_u$ of $\rho$, see Definition \ref{up}.

\medskip

\noindent \textit{Sub-case 1: $\rho_u$ has dense image.}  Recall from Proposition \ref{affg1k1} that the quadrilateral $Q$ corresponding to a handle (with generators mapping to $A$ and $B$) and base-point $p$ is defined by the oriented polygon $p \mapsto A(p) \mapsto AB(p)=BA(p) \mapsto B(p) \mapsto p$, see equation \eqref{pol}. It follows that if $A$ and $B$ have dilation factors each of modulus strictly greater than $1$, and with argument sufficiently small, then
\begin{itemize}
    \item[(a)] $Q$ lies entirely to the right of the vertical line passing through $p$, and
    \item[(b)] the "rightmost" point of $Q$ is $AB(p)=BA(p)$. 
\end{itemize}

\noindent Thus, to construct the desired non-overlapping "chain" of quadrilaterals $\{Q_1,Q_2,\ldots, Q_g\}$, we choose handle-generators as in Lemma \ref{coax-lem3}. We choose a base-point for $Q_1$, corresponding to the first handle, to be an point $p\in \mathbb{R}^+\subset \C$, and then define the base-point for each successive quadrilateral, corresponding to the next handle, to be the right-most point of the preceding quadrilateral, as in (b) above. 

\begin{figure}[h]
  \centering
  \includegraphics[scale=0.4]{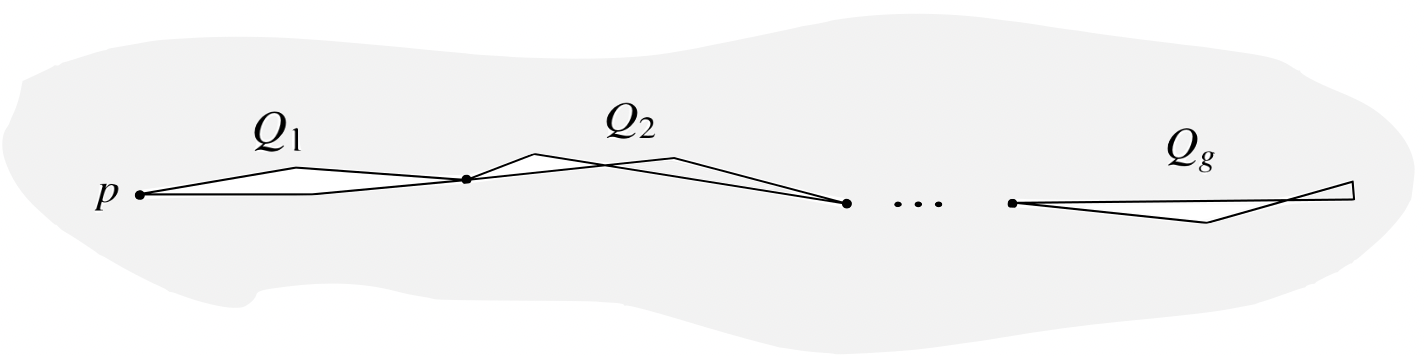}
  \caption{In sub-case 1, the chain of quadrilaterals proceeds towards the right and remains close to the positive real axis. }
  \label{chainofquads}
\end{figure}

\noindent This oriented chain of quadrilaterals bounds an immersed punctured disk on its right, where the puncture is at the point at infinity. In other words, $Q_1 \cup Q_2\cup \cdots \cup Q_g$ bounds an immersed disk in $\cp$ that contains the point $\infty$. For each quadrilateral, we can identify pairs of sides using the affine maps corresponding to the generators of that handle; this results in a genus $g$ surface equipped with a branched projective structure. Moreover, there is a unique branch-point, namely the point where all the vertices of the quadrilaterals get identified to; removing the branch-point we obtain the desired surface homeomorphic to $S_{g,1}$ with a projective structure having monodromy $\rho$. 

\medskip 

\noindent \textit{Sub-case 2: $\rho_u$ has discrete image.} Assume that the image is a cyclic group of order $m\geq 1$. Recall that we are in the case when $\rho$ is not unitary. We now observe that on each handle, we can perform the change of the pair of generators such that our Claim \ref{claim1} holds, and the handle-generators still satisfy the conclusion of Lemma \ref{coax-lem1}. To see this is true, observe that for the change of handle-generators  $\{\alpha,\beta\} \mapsto \{\alpha,\, \alpha^n \beta\}$ for $n\in \mathbb{Z}$, if $\rho_u(\alpha) =1$ and $\rho_u(\beta)=0$, then $\rho_u(\alpha^n\beta) = 1$ whenever $n \equiv m\, (\text{mod } 1)$. Recall from the proof of Claim \ref{claim1} that such Dehn-twists around $\alpha$ can ensure that the resulting new generator satisfies $\lvert b\rvert >1$. For the other generator, switch the roles of $\alpha$ and $\beta$, namely consider $\{\alpha,\beta\} \mapsto \{\alpha\,\beta^n,\, \beta\}$ for some $n\in \mathbb{Z}$. Note that in this case, the $\rho_u$-images of the new generators remain unchanged. Thus, we can assume that the handle-generators satisfy
\begin{itemize}
    \item[1.] $\lvert a_j \rvert, \lvert b_j\rvert >1$, and
    \item[2.] $\rho_u(\alpha_j) = \exp{\Big(\frac{2\pi\,i}{m}\Big)}$ and $\rho_u(\beta_j)=1$,
\end{itemize}
for each $1\leq j\leq g$.

\smallskip 

\noindent Let $Q_j$ be the quadrilateral corresponding to the $j$-th handle, with base-point $p_j\in \C$. Note that the edges of $Q_j$ are 
\begin{equation*}
p_j \longmapsto \lvert a_j \rvert \exp{\Bigg(\frac{2\pi\,i}{m}\Bigg)}\,p_j \longmapsto \lvert a_j \rvert \lvert b_j \rvert \exp{\Bigg(\frac{2\pi\,i}{m}\Bigg)}\, p_j \longmapsto \lvert b_j \rvert p_j \longmapsto p_j
\end{equation*}

\noindent In other words, in polar coordinates on $\C$, if $\lvert p_j \rvert = R$ and $\arg{p_j} = \theta_0$, the quadrilateral $Q_j$ bounds the rectangular region 
\begin{equation*}
    \Bigg\{(r,\theta)\ \vert\  R  \leq r  \leq \lvert a_j \rvert \lvert b_j \rvert \,R \text{ and }\ \theta_0 \leq \theta \leq \theta_0 + \frac{2\pi}{m} \Bigg\}.
\end{equation*} Note that the third vertex of $Q_j$ is an extreme point of the region, furthest from the origin.

\smallskip 

\noindent We choose the base-point of $Q_1$ to be $p_1=1$, and for each successive quadrilateral $Q_j$, define the base-point $p_j$ to be the third (i.e.\ extreme) vertex of $Q_{j-1}$. Note that $\lvert p_j\rvert > \lvert p_{j-1}\rvert$ for each $2\leq j\leq g$. The quadrilateral $Q_j$ can only intersect $Q_{j-1}$ at that common vertex, since all the remaining vertices of the preceding quadrilaterals $Q_1,Q_2,\ldots ,Q_{j-1}$ lie in the interior of the disk of radius $\lvert p_j \rvert$ around the origin.

\smallskip 

\begin{figure}
  \centering
  \includegraphics[scale=0.4]{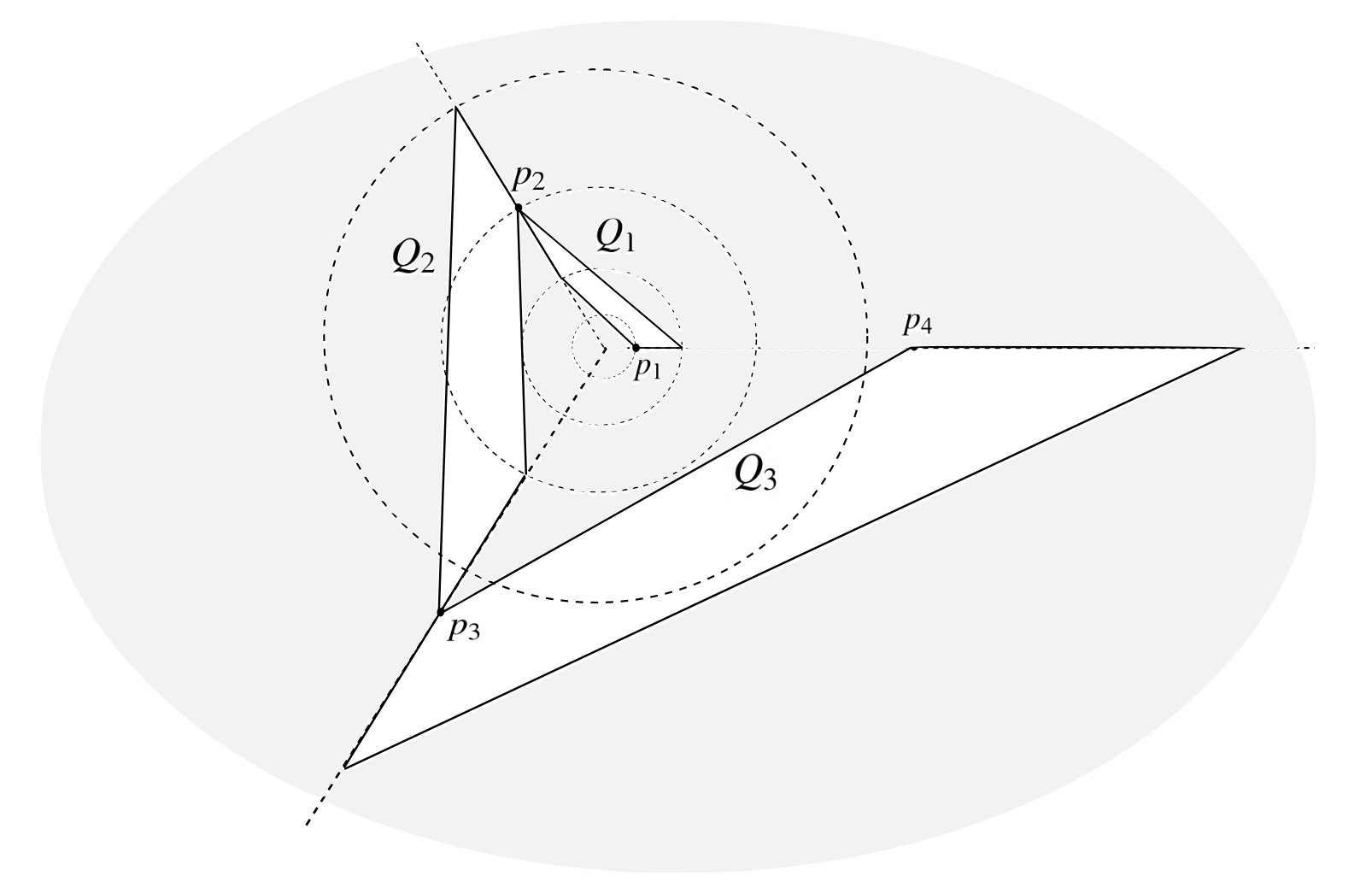}
  \caption{In sub-case 2, the chain of quadrilaterals is spiralling if the order of the discrete image is $m>1$. (Shown above for $m=3$.) }
  \label{spiralquad}
\end{figure}

\noindent The resulting sequence of quadrilaterals $Q_1,Q_2,\ldots, Q_g$ thus forms a non-overlapping chain, as we desired. Note that in the special case that $m=1$ (\textit{i.e.} $\rho_u$ is the trivial representation), then each quadrilateral is degenerate, as its sides lie along the real line, and they form a chain along the positive real axis; otherwise, if $m>1$, this chain is ``spiralling" as shown in Figure \eqref{spiralquad}.  As in sub-case 1, this (oriented) chain bounds an immersed (in fact an embedded) disk in $\cp$ on its right containing the point $\infty$. Identifying pairs of sides of each quadrilateral $Q_j$ using the affine maps $\rho(\alpha_j)$ and $\rho(\beta_j)$, we obtain a surface of genus $g$ equipped with a branched projective structure and monodromy $\rho$, with a unique branch-point where all the vertices of the chain get identified. Removing this branch-point, we obtain the desired projective structure on $S_{g,1}$.

\medskip 

\noindent\textit{Case (ii): Unitary with dense image.} Let us assume now that $\rho$ is unitary (i.e.\ $\rho=\rho_u$), and with an image that is dense in $\text{U}(1)$. Fix an $\epsilon>0$, the choice of which shall be made clearer later.  We can apply the proof of the second part of Lemma \ref{coax-lem3} to obtain a change of generators (effected by some mapping class) such that the resulting handle-generators $\{\alpha_j,\,\beta_j\}_{1\leq j\leq g}$ satisfy:
\begin{itemize}
    \item[1.] $\arg{a_j}, \arg{b_j} \notin 2\pi \mathbb{Q}$, and 
    \item[2.] $\lvert \arg{a_j} \rvert,  \lvert \arg{b_j} \rvert <\epsilon$
\end{itemize}
for each $1\leq j\leq g$. This change of generators is in fact exactly as in \cite[Lemmata 11.4 and 11.5]{CFG}.
\smallskip

\noindent Thus, we can assume that for the $j$-th handle, the pair of generators map to the elements of the form
\begin{equation*}
    z\mapsto \exp\Big({i\,\theta_j^1}\Big)\,z\,\,\text{ and }\,\, z\mapsto \exp\Big({i\,\theta_j^2}\Big)\,z
\end{equation*}

\noindent of $\text{U}(1)$ respectively, where $0<\theta_j^1,\theta_j^2 <\epsilon$, for each $1\leq j\leq g$. Choose the base-point for the first handle to be on the positive real axis, say $p_1 = R \in \mathbb{R}^+$. Then the quadrilateral $Q_1$ corresponding to the first handle, given by $p_1\mapsto \text{exp}({i\,\theta_j^1})p_1 \mapsto \text{exp}({i\,(\theta_j^1 + \theta_j^2)})p_1 \mapsto  \text{exp}({i\,\theta_j^2})p_1 \mapsto p_1$, has vertices on the circle of radius $R$, and lies in the sector bounded by rays at angles $0$ and $\theta_1^1 + \theta_1^2$. Choose the base-point of the next handle to be $p_2 = \text{exp}\Big({i\,(\theta_j^1 + \theta_j^2)}\Big)\,p_1$; the quadrilateral  $Q_2$ then lies in an adjacent  sector of angular width $\theta_2^1 + \theta_2^2 < 2\epsilon$.  We can continue placing quadrilaterals for successive handles, choosing the base-point of each to be the extreme point for the previous quadrilateral; each is contained in a sector of angular width less than $2\epsilon$.   Our initial choice of  $\epsilon>0$  can be made such that $g$ such sectors fit without overlapping, i.e.  $2g\epsilon < 2\pi$.  We thus obtain an oriented chain of quadrilaterals $Q_1,Q_2,\ldots, Q_g$, as in case (i), where successive handles intersect at a common vertex, and every other pair is disjoint. Their union then bounds an immersed disk  in $\cp$ on its right, containing the point $\infty$.  As in case (i), we then identify pairs of edges of each quadrilateral using the maps corresponding to the generators, to obtain a surface homeomorphic to $S_g$, equipped with a branched projective structure with a unique branch-point. Deleting the branch-point, we obtain the desired projective structure on $S_{g,1}$ with monodromy $\rho$. 

\medskip

\noindent\textit{Case (iii):} We now consider the remaining case when $\rho$ is co-axial, but the image of $\rho$ is a finite group in $\text{U}(1)$. Let the order of this finite group be $m\geq 3$; here, recall that  $m\neq 2$ by Lemma \ref{excase}.

\noindent We first apply Lemma \ref{coax-lem1} to  obtain handle-generators $\{\alpha_j,\,\beta_j\}_{1\leq j \leq g}$, such that
\begin{equation*}
    \rho(\alpha_j)=\rho(\beta_j) = \exp{\Bigg(\frac{2\pi\,i}{m}\Bigg)}
\end{equation*}
for each $j$. For any handle generated by $\{\alpha_j,\,\beta_j\}$, and a choice of a base-point $p\in \C$, the quadrilateral $Q$ with edges 
\begin{equation*}
    p\mapsto \rho(\alpha_j)p \mapsto \rho(\alpha_j\,\beta_j)p\mapsto \rho(\beta_j)p\mapsto p
\end{equation*}
is a degenerate ``V"-shaped quadrilateral, since the second and fourth vertices coincide. Such a quadrilateral bounds an immersed (in fact embedded) disk in $\cp$ in its exterior. However, all vertices lie on the circle of radius $\lvert p\rvert$ centered at $0$, and they span an angle $4\pi/m$ at the origin.  This makes it difficult to form a non-overlapping chain of quadrilaterals, as we were able to do in Cases (i) and (ii). We resolve this difficulty by using a ``grafting" construction that we shall describe next,  the idea of which is similar to Definition \ref{glue-new}. 

\smallskip 

\noindent First, we need to introduce the following 

\begin{defn}[Projective handle]\label{projhandle}
A \textit{projective handle} will refer to a branched projective structure on a torus with a single branch-point, obtained by identifying pairs of sides of a quadrilateral $Q$ in $\cp$ that bounds an immersed disk (recall that the preceding cases have involved constructing such projective handles).
\end{defn}

\begin{figure}[h] 
  \centering
  \includegraphics[scale=0.4]{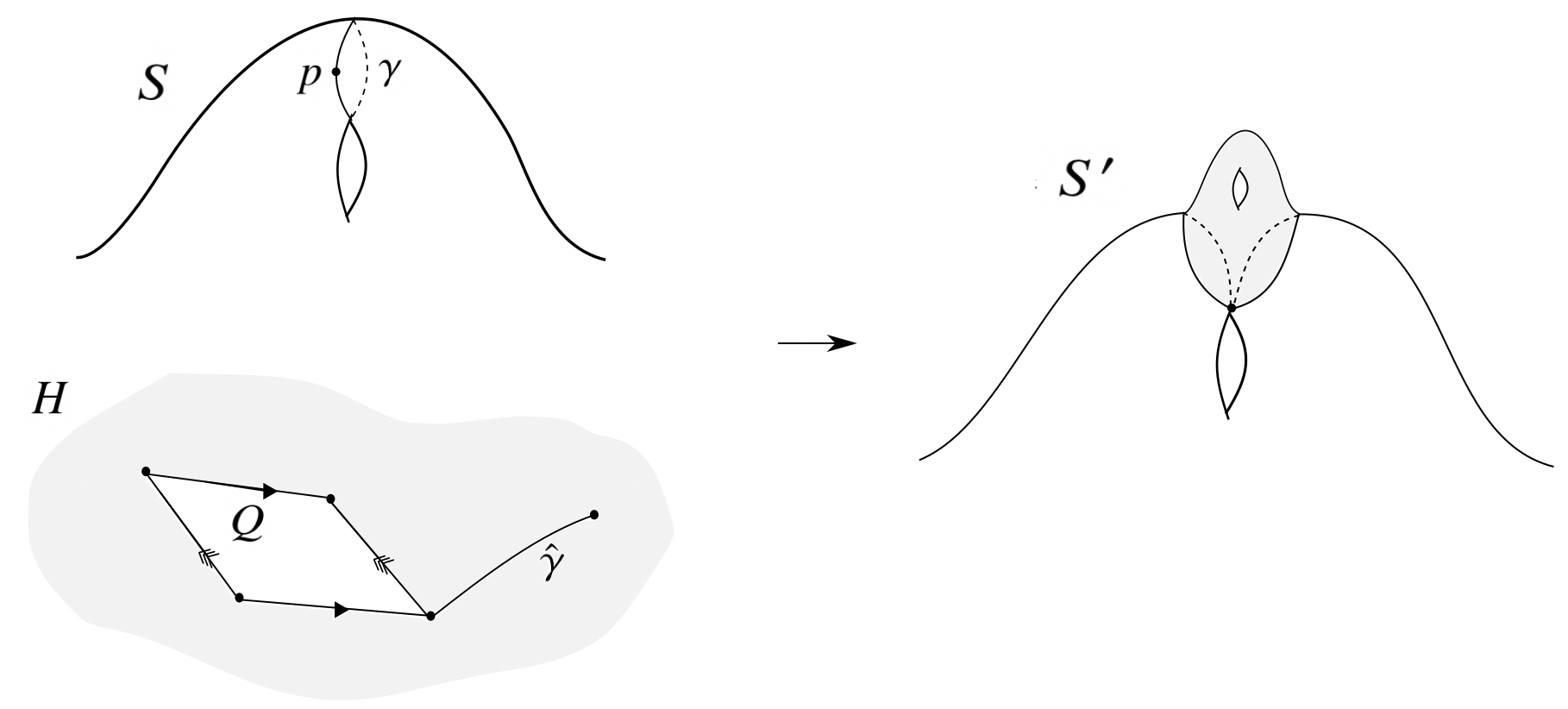}
  \caption{In Definition \ref{bubbhand} the handle $H$ is slit along $\widehat{\gamma}$ and grafted in on $S$ along $\gamma$. }
  \label{grafthandle}
\end{figure}

\begin{defn}[Grafting in a handle]\label{bubbhand} Let $S$ be a surface equipped with a branched projective structure, and let $\gamma$ be an embedded arc on $S$ from a branch-point $p$ to itself that develops onto an embedded arc $\widehat{\gamma}$ on $\cp$. Suppose $H$ is a projective handle that corresponds to a quadrilateral $Q$ on $\cp$ such that $\widehat{\gamma}$ lies in the disk in $\cp$ bounded by $Q$, and an endpoint of $\widehat{\gamma}$ is a vertex of $Q$.
Consider the one-holed torus $T$ obtained by introducing a slit in $H$ along the arc that develops onto $\widehat{\gamma}$; let the resulting two sides of the slit be $\sigma^+$ and $\sigma^-$. Then cut along the arc $\gamma$ on $S$ and identifying the resulting sides with the boundary arc $\sigma^+$ and $\sigma^-$ on $T$ respectively, so that the genus of the resulting surface $S^\prime$ is one more than that of $S$. Here, the identification is such that the developing maps to $\cp$ are precisely the same; the surface $S^\prime$ thus acquires a (branched) projective structure. See Figure \eqref{grafthandle}. 
\end{defn}

\noindent Note that in the construction above, 
\begin{itemize}
    \item there are no new branch-points that are introduced, but the order of the branch-point $p$ increases by two, and 
    \item the monodromy representations when restricted to the sub-surfaces $S$ and $H$ remain unchanged. 
\end{itemize}

\medskip 

\noindent We can now construct the projective structure on $S_{g,1}$ with monodromy $\rho$ (which is co-axial, unitary and discrete) by successively grafting in handles, as we now describe.

\smallskip

\noindent Start with the base-point $p_1\in \mathbb{R}^+\subset \C$ and the quadrilateral $Q_1$ corresponding to the first handle, generated by $\{\alpha_1,\,\beta_1\}$. The projective handle corresponding to $Q_1$ is our initial surface $S_1$. We shall successively graft in $g$ handles as in Definition \ref{bubbhand}; in what follows we describe the $j$-th step, where we assume we have a surface $S_{j}$ of genus $1\leq j <g$ equipped with a branched projective structure with a unique branch-point.

\begin{figure}[h!]
  \centering
  \includegraphics[scale=0.28]{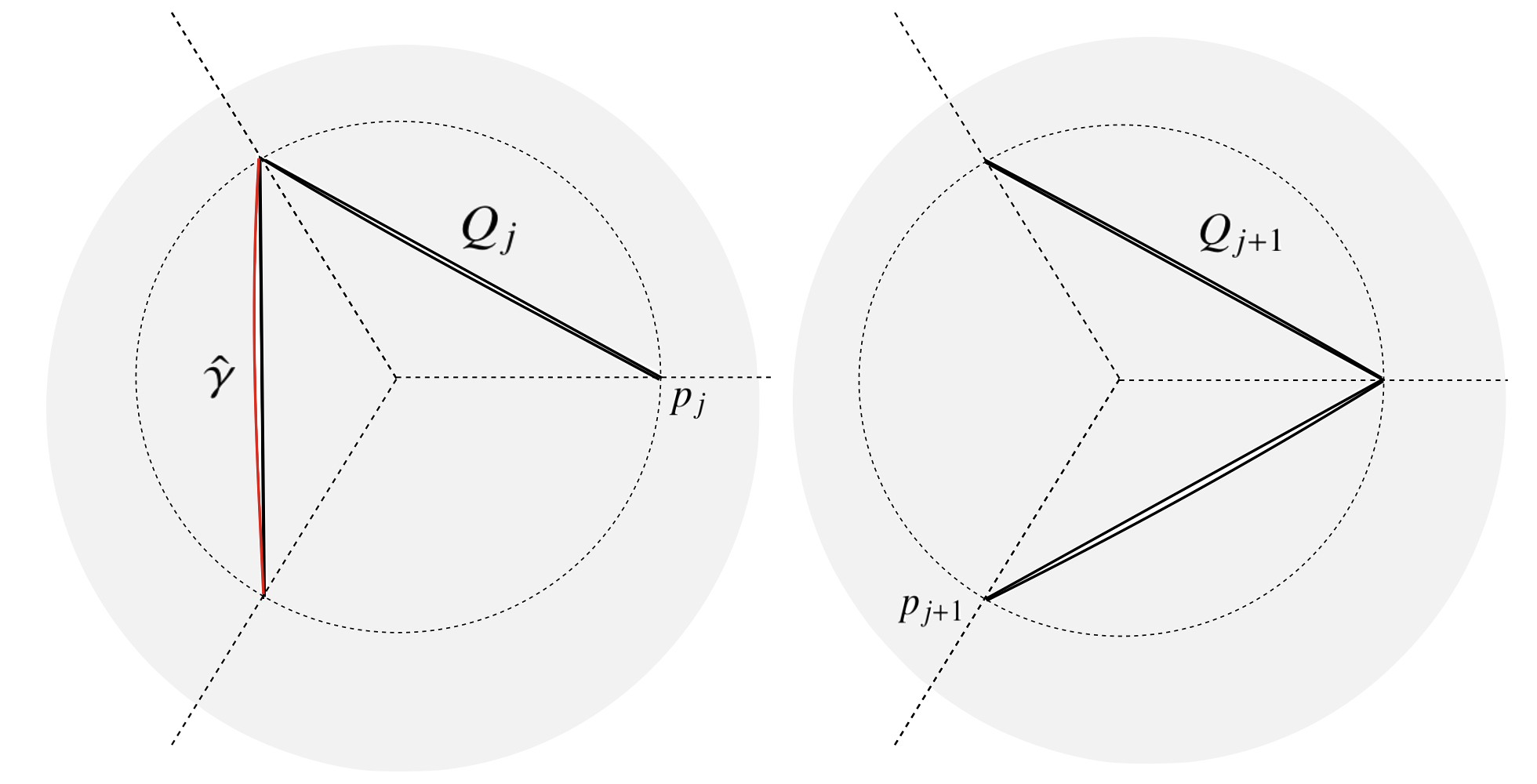}
  \caption{Two successive quadrilaterals in the case $m=3$, shown here on two different copies of $\cp$; in the inductive step the projective handle corresponding to the latter is grafted in by Definition \ref{bubbhand}. }
  \label{successivequads}
\end{figure}

\noindent Let $Q_j$ be the quadrilateral for the $j$-th handle, which by our construction will be a sub-surface of $S_j$. Let $H$ be the projective handle corresponding to the quadrilateral $Q_{j+1}$ of the next handle, when the base-point $p_{j+1}$ for that is taken to be the third (\textit{i.e.} extreme) point of $Q_j$. Note that $\gamma = \alpha_j\text{ or } \beta_j$ is an embedded arc from the branch-point on $S_j$ to itself, and we can choose a developing image that is an embedded arc $\widehat{\gamma}$ in $\cp$, namely one of the sides of $Q_j$ that is incident to $p_{j+1}$.
Moreover, $\widehat{\gamma}$ lies in the exterior of the quadrilateral $Q_{j+1}$ for $H$, see Figure \ref{successivequads}. 

\noindent Hence the construction in Definition \ref{bubbhand} can be applied, that is, we can cut along $\gamma$ on $S_j$ and graft in the handle $H$; the resulting surface is $S_{j+1}$.

\noindent At the end of $g$ such steps, we obtain a genus-$g$ surface $S_g$ with a branched projective structure with a unique branch-point. Removing the branch-point, we obtain a projective surface homeomorphic to $S_{g,1}$ with monodromy $\rho$, i.e.\ with $\rho(\eta) = \exp(2\pi\,i/m)$ for any handle-generator $\eta$. 

\medskip 

\noindent We have dealt with cases (i) - (iii), and this completes the first part of the proof of Proposition \ref{affk1}.

\subsubsection{Proof of Proposition \ref{affk1}: Non co-axial representations}\label{affk1notcoa}
Let $\rho: \pi_1(S_{g,1}) \to \affc$  be a non-trivial representation that is not co-axial. Since the puncture has trivial monodromy we regard $\rho$ as a representation $\overline{\rho}:\pi_1(S_g)\longrightarrow\affc$ and we shall realise it as the monodromy of a branch projective structure with a single branch point. We shall eventually delete the branch point to get a complex projective structure on $S_{g,1}$ with the desired monodromy.
\medskip 

\noindent \textit{Step 1: Commutators determine a convex polygon $\mathcal{C}$.} Given a non-coaxial representation $\rho$, Lemma \ref{non-coax} applies and hence we can assume the existence of a basis of handle-generators $\{\alpha_i, \beta_i\}_{1\leq i\leq g}$ such that $\rho([\alpha_i,\beta_i])\neq \text{I}$ for each $i$. It is easy to see that these commutators are all translations, which we denote by $t_1,t_2,\ldots, t_n$. It is also not hard to show that there exists a permutation of the handles (realized by a mapping class) such that there is an (possibly degenerate) oriented convex polygon $\mathcal{C}\subset\C$ with $g$ sides, such that the endpoints of the $i$-th side differ by the translation $t_i$, for each $1\leq i\leq g$. Indeed, one permutation that works is the one that puts the arguments of the translations in increasing order, \textit{i.e.} if $t_i\lvert t_i \rvert^{-1}$ are in counter-clockwise order on the unit circle (\textit{c.f.} \cite[Proof of Proposition 6.1]{CFG}). With respect to this choice, the piece-wise linear curve $\partial\mathcal{C}$ bounds the polygon $\mathcal{C}$ on its \textit{left}. Furthermore, the convex polygon $\mathcal{C}$ can be placed everywhere in $\C$. In fact, given any starting point $p_1\in\C$, the $i$-th side is from $p_i$ to $p_{i+1}=p_i+t_i$, where $i\in \{1,2,\ldots,g\}$ in the re-ordered set of handles and $p_{g+1}=p_1$ because $t_1+\cdots+t_g=0$.

\medskip

\noindent \textit{Step 2: Realizing a one-holed torus.} Let $S_{1,1}\cong H\subset S_g$ be any handle, see Definition \ref{handle}, such that the representation $\rho_{|H}$ induced by the inclusion $H\hookrightarrow S_g$ is not abelian. In this step we show how to realize
$\rho_{|H}:\pi_1(S_{1,1}) \to \text{Aff}(\C)$, as the holonomy of a branched projective structure on a one-holed torus $T$ with linear boundary, except for at most one corner point, and no interior branch point. 

\smallskip

\noindent For this, we shall need an immersed disk in $\cp$ containing $\infty$ bounded by an Euclidean pentagon $\mathcal{P}$. Such a pentagon is determined by a choice of a base-point $p$ and the $\rho_{|H}$-images of the handle-generators, that we denote by $A$ and $B$. Recall that, here, we shall assume that $A$ and $B$ do not commute. Thus, given a point $p\in\C$, the pentagon $\mathcal{P}\subset \cp$ is defined as the region containing the infinity $\infty$ and bounded  by the chain
\begin{equation}\label{pentbaseoncomm}
    p \longmapsto [A,B](p)\longmapsto B^{-1}(p)\longmapsto A^{-1}B^{-1}(p) \longmapsto BA^{-1}B^{-1}(p)\longmapsto p
\end{equation} 

\noindent on the \textit{right}, so that the base-point $p$ is an extremal point of the segment $\sigma$ corresponding to the commutator.

\smallskip

\noindent We denote the oriented sides of $\mathcal{P}$ as follows: $e_1 = \overline{BA^{-1}B^{-1}(p)\, p}$, $e_2 = \overline{A^{-1}B^{-1}(p)\,BA^{-1}B^{-1}(p) }$, $e_3 = \overline{A^{-1}B^{-1}(p)\, B^{-1}(p)}$, $e_4 = \overline{B^{-1}(p)\,[A,B](p)}$ and, finally, $\sigma = \overline{p\, [A,B](p)}$. Notice that a similar construction already appeared in Proposition \ref{affk2}, see figure \ref{leftrightchoice}, where the base-point has been chosen to be a different vertex. It is not clear a priori that this pentagon bounds an immersed disk for a suitable choice of $p$, therefore the key assertion here is the following:

\begin{lem}\label{pent_lemma}
Let $p_0\in\C$ be any point with positive imaginary part sufficiently large. Define
\begin{equation}\label{parbaspoint}
    p_t=
    \begin{cases}
    (1+t)\Re{(p_0)}+i\,\Im{(p_0)} \quad \text{if}\quad \Re{(p_0)}>0\\
    (1-t)\Re{(p_0)}+i\,\Im{(p_0)} \quad \text{if}\quad \Re{(p_0)}<0\\
    \end{cases}
\end{equation} for $t>0$. Then there is a basis $\{\alpha,\beta\}$ of $\pi_1(S_{1,1})$ and $t_0>0$ such that the pentagon \eqref{pentbaseoncomm}
 \[  p_t \longmapsto [A,B](p_t)\longmapsto B^{-1}(p_t)\longmapsto A^{-1}B^{-1}(p_t) \longmapsto BA^{-1}B^{-1}(p_t)\longmapsto p_t
 \] based at $p_t$ bounds an immersed disk in $\cp$ containing the infinity on its right for any $t\ge t_0$; where $A=\rho(\alpha)$ and $B=\rho(\beta)$.
\end{lem}

\textit{(Here $\Re(z)$ and $\Im(z)$ are the real and imaginary parts respectively, of the complex number $z$.)}

\smallskip

\noindent The polygon \eqref{pentbaseoncomm} and its shape highly depend on the choice of a base-point $p_0$. A priori, there are \textit{no} restrictions on such a choice. The lemma above says that if the base-point $p_0$ is taken with positive imaginary part sufficiently large then, by moving $p_0$ horizontally, we eventually find a time $t_0$ such that the polygon \eqref{pentbaseoncomm} bounds on its right an immersed pentagon in $\cp$ containing the infinity. How large the imaginary part of $p_0$ must be will be clear in context, case by case. Due to the technicality of the Lemma \ref{pent_lemma}, we postpone its proof to the end of the current subsection \ref{hgask1}.

\medskip

\noindent Suppose Lemma \ref{pent_lemma} holds, \textit{i.e.} there is such an immersed disk. The desired branched projective structure is then obtained by identifying the oriented sides $e_1$ and $e_3$ via the affine map $A$ and the sides $e_2$ and $e_4$ via the affine map $B$. The resulting surface is a one-holed torus $T$, where the side $\sigma$ after the identifications forms the boundary  $\partial T$; the vertices of $\mathcal{P}$ get identified to the unique corner-point that lies on that boundary. We shall consider handles thus constructed in the next step below. 


 
\medskip 

\noindent \textit{Step 3: Gluing handles to the polygon $\mathcal{C}$.} Let $\{\alpha_i,\beta_i\}_{1\le i\le g}$ be a set of handle-generators as given in the Step 1, namely such that $t_i=\rho\big([\alpha_i,\beta_i]\big)\neq\text{I}$ for each $i$. We can order the handles cyclically so that the translations $t_i$ form a convex polygon $\mathcal{C}\subset \C$ and $\partial\mathcal{C}$ bounds the polygon on its left.

\smallskip

\noindent Let $H_i$ be the handle generated by $\{\alpha_i, \beta_i\}$, where $i \in \{1,2,\ldots, g\}$. For any $i$, we want to apply the second step above and then obtain a one-holed torus $T_i$ which carries a branched projective structure having holonomy $\rho_i = \rho\vert_{H_i}$. In fact, given a suitable starting point $p_i$, Lemma \ref{pent_lemma} above states that by perturbing $p_i$ horizontally, \text{i.e.} by preserving the imaginary coordinate, the chain \eqref{pentbaseoncomm} eventually bounds an immersed pentagon $\mathcal{P}_i\subset\cp$ on its right containing the point at infinity. 

\smallskip

\noindent Recall there are no restriction on where to place $\mathcal{C}$, that is its shape does not depend on the base-point; indeed, changing the base-point changes $\mathcal{C}$ by a translation. Therefore we place it sufficiently far from the origin so that the real and imaginary parts of each vertex are sufficiently large and hence Lemma \ref{pent_lemma} applies for each handle. More precisely, the $i$-th vertex of $\mathcal{C}$ will serve as the base-point for the $i$-th handle. A fundamental membrane for the developing image of $T_i$ is $R_i$, the immersed pentagonal region in $\cp$ on the right of the chain \eqref{pentbaseoncomm} based at the $i$-th vertex of $\mathcal{C}$. The image of the boundary $\partial T_i$ is exactly the $i$-th side of the convex polygon $\mathcal{C}$ constructed in Step 1.

\smallskip

\noindent Consider the region $R_0 \subset \C$ bounded by the polygon $\mathcal{C}$; note that $R$ is empty if $\mathcal{C}$ is degenerate (i.e.\ all sides are collinear).  Define the space $\overline{R}_0 = R_0/\sim$ where $\sim$ identifies all the vertices of $\mathcal{C}$ to a point; topologically, $\overline{R}_0$ is homotopy-equivalent to a $g$-holed sphere, and the $i$-th side of $\mathcal{C}$ defines an $i$-th "boundary circle" $c_i$ on $\overline{R}_0$. We now glue $\bar{R}_0$ with the one-holed tori obtained above by identifying the boundary of $T_i$ with $c_i$ for each $1\leq i\leq g$, such that the resulting surface is homeomorphic to $S_g$. This surface acquires a branched projective structure with a unique branch point, and a fundamental membrane in the image of its developing map is $R_0 \cup \bigcup_{i=1}^g R_i$. Removing the branch-point, we obtain our desired projective structure on $S_{g,1}$ with holonomy $\rho$. 

\smallskip

\noindent It only remains to prove Lemma \ref{pent_lemma}.

\begin{figure}
  \centering
  \includegraphics[scale=0.3]{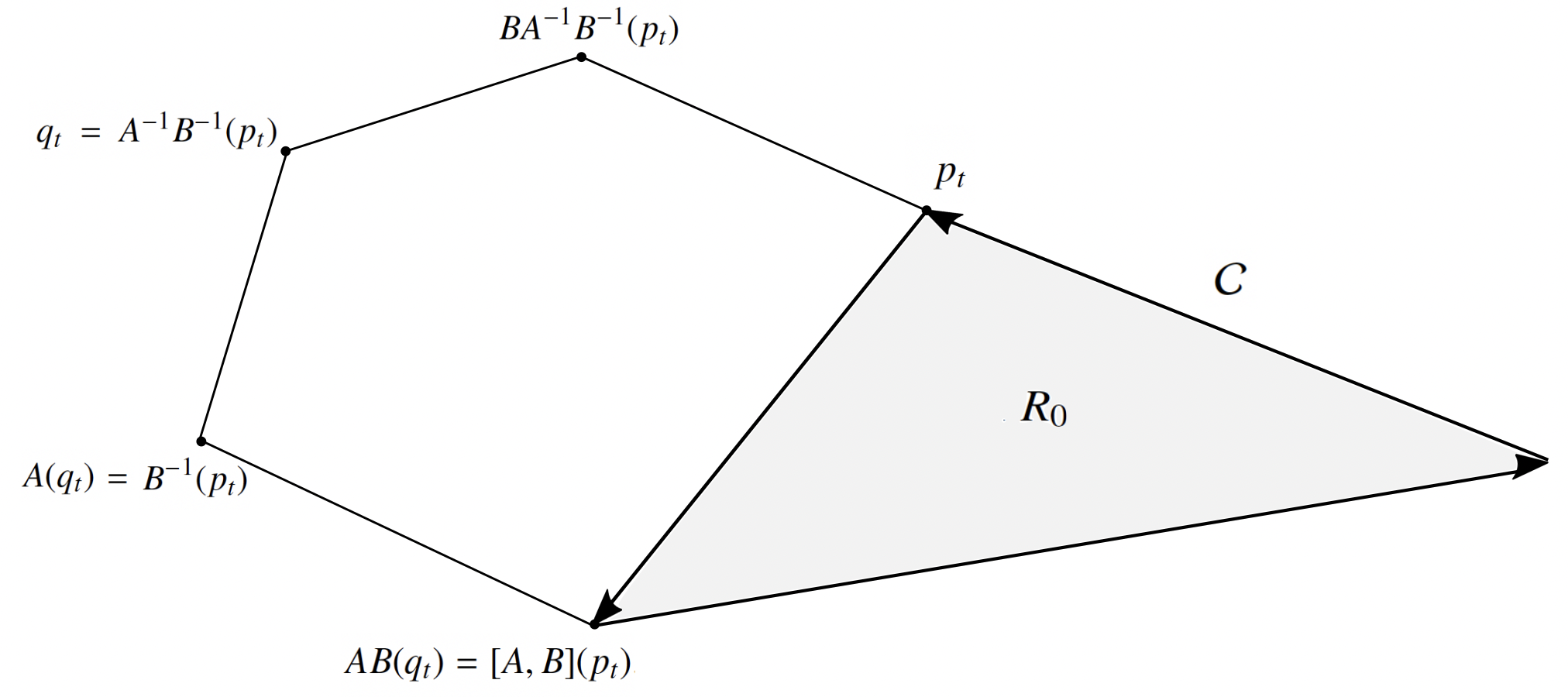}
  \caption{The vertices of the pentagon in the proof of Lemma \ref{pent_lemma} are shown labelled. The commutator edge $\overline{p_t\, [A,B](p_t)}$ coincides with one of the edges of the convex polygon $\mathcal{C}$, and the region $R_0$  bounded by $\mathcal{C}$ lies on its left.}
\end{figure}

\begin{proof}[Proof of Lemma \ref{pent_lemma}]
Given any set of handle generators, say $\{\alpha,\beta\}$, for $\pi_1(S_{1,1})$ we shall denote the $\rho$-images of $\alpha,\beta$ by $A(z)=az+c$ and $B(z)=bz+d$ respectively. We observe that the proof follows as soon as we show that three consecutive edges of the pentagon
are embedded in $\C$. This is because it is easy to verify that a closed oriented curve in $\C\cup \{\infty\} $ that is obtained by concatenating two embedded arcs always bounds an immersed disk on either of its sides (\textit{c.f.} Figure \ref{leftrightchoice}). Let us start with some generalities.

\smallskip

\noindent Let $p_0\in \C$ be any point on the upper half-plane such that $[A,B](p_0)$ is contained in the same half-plane. Notice that this can be made sure by choosing $p_0$ with positive imaginary part and greater than $2\,\lvert p_0 -  [A,B](p_0)\rvert $. Let $p_t$ be defined as in the equation \eqref{parbaspoint}. As a consequence of our definition, for $t$ running to the infinity, the imaginary part of $p_t$ remains constant and hence the points $p_t$ and $[A,B](p_t)$ both lie on the upper half-plane for any time $t\ge0$.

\smallskip

\noindent We now define $q_t=A^{-1}B^{-1}(p_t)$ for any $t\ge0$. We observe that, for $t$ tending to infinity, the following limits hold
\begin{equation}\label{limits}
    \lvert q_t\rvert \longrightarrow \infty \,\,\text{ and }\,\, \arg q_t \longrightarrow \arg \overline{ab}=\delta.
\end{equation}
The second limit can be easily explained as follows. We first notice that $q_t$ can be written as follows:
\begin{equation}
    q_t=\frac{\big(p_t-(bc+d)\big)\,\overline{ab}}{\lvert ab\rvert^2}.
\end{equation} By setting $w = bc+d$, then
\begin{align}
    \label{realqt} \Re{(q_t)}&=\frac{1}{\lvert ab \rvert^2}\Bigg(\Big(\Re{(p_t)}-\Re{(w})\Big)\Re{(\overline{ab})}-\Big(\Im{(p_t)}-\Im{(w)}\Big)\Im(\overline{ab})\Bigg)\\
    \label{imqt} \Im{(q_t)}&=\frac{1}{\lvert ab \rvert^2}\Bigg(\Big(\Re{(p_t)}-\Re{(w})\Big)\Im{(\overline{ab})}+\Big(\Im{(p_t)}-\Im{(w)}\Big)\Re(\overline{ab})\Bigg).
\end{align} Now it is a routine exercise to check that $\arg q_t\longrightarrow\delta$ for $t$ running to the infinity (recall that the imaginary part of $p_t$ is constant as a function of $t$ and equal to $\Im{(p_0)}$). 

\smallskip 

\noindent We shall now distinguish two cases according to the image of the unitary part of $\rho$.

\smallskip

\noindent \textit{Case 1. The unitary part $\rho_u$ has dense image in $U(1)$.} In this case our Corollary \ref{approxlem} applies and hence, for any arbitrarily small $\epsilon>0$, there is a set of handle generators $\{\alpha,\beta\}$ such that 
\[ 0<\arg a < \epsilon,\ 0< -\arg b<\epsilon\,\,\text{ and }\,\, 0\le \arg a+ \arg b<\epsilon.
\]

\smallskip

\noindent The inequalities $0\le\arg a+ \arg b<\epsilon$ readily implies $-\epsilon<\delta\le0$ since $\delta = -\arg a - \arg b$, and note that $\delta =0$ if and only if $ab\in\R$. Let us start by assuming $\delta<0$. Then, for any $t$ large enough, the point $q_t$ always lies on the lower half-plane and its norm can be taken to be  arbitrarily large. 
\smallskip 

\noindent We now consider the point $A(q_t)=aq_t+c$. For any $t$ large enough, $\arg A(q_t)$ is barely affected by the translational part of $A$, in other words $\arg A(q_t)\approx \arg aq_t$. More precisely, since $\arg q_t$ tends to $\delta$ (see formula \eqref{limits}) then $\arg aq_t\longrightarrow - \arg b>0$ and the open ball $B\big(aq_t,\,2|c|\big)$ is entirely contained in the upper half-plane for any $t$ sufficiently big. Clearly, $A(q_t)\in B\big(aq_t,\,2|c|\big)$.\\
In the same fashion we can observe that $\arg bq_t\longrightarrow -\arg a<0$ and the open ball $B\big(bq_t,\,2|d|\big)$ is entirely contained in the lower half-plane for any $t$ sufficiently big. Similarly to the above, it is clear that $B(q_t)\in B\big(bq_t,\,2|d|\big)$.

\smallskip

\noindent Recall that there exists a $t_1>0$ such that $q_t$ always lies in the lower half-plane for any time $t>t_1$. This necessarily forces  the segment $\overline{q_t\, B(q_t)}$ to be contained in the lower half-plane. On the other hand, there exists a  $t_2>0$ such that $A(q_t)$ always lies in the upper half-plane for any $t>t_2$ as already observed and this forces the edge $\overline{A(q_t),\,AB(q_t)}$ to be entirely contained in the upper half-plane, where $AB(q_t)=[A,B](p_t)$. As a consequence, the chain of segments 
\begin{equation}
    AB(q_t)\longrightarrow A(q_t) \longrightarrow q_t \longrightarrow B(q_t)
\end{equation} is embedded in $\C$. In fact, the edges $\overline{q_t\,B(q_t)}$ and $\overline{AB(q_t)\,A(q_t)}$ cannot intersect for any time $t>\max\{t_1,t_2\}$ because they lie on different half-planes. Therefore the polygon \eqref{pentbaseoncomm} bounds an immersed disk on its right containing the infinity on the Riemann sphere.

\smallskip

\noindent Let us now assume $\delta=0$. This case occurs if and only if $ab\in \mathbb{R}$ and this implies $\Im{(\overline{ab})}=0$. The imaginary part $\Im{(q_t)}$ of $q_t$ seen as a function of $t$ is constant, see the formula \eqref{imqt}, and it may be positive. Therefore the argument above might simply not apply in this case. We bypass this issue as follows. Recall that $ab\in\mathbb{R}$ if and only if $\arg a=-\arg b$. Then it is sufficient to replace the given pair of handle-generators  $\{\alpha,\beta\}$ with $\{\alpha',\beta'\}=\{\alpha\beta^{-1},\beta\}$ in order to fall in the case of $\delta<0$. In fact, let $a'$ and $b'$ be the linear parts of $\rho(\alpha')$ and $\rho(\beta')$ respectively, then
\begin{equation}
    \arg a'=\arg a-\arg b=2\arg a \,\, \text{ and } \arg b'=\arg b.
\end{equation} Now it is an easy matter to check that both $0<\arg a', -\arg b'<2\epsilon$ and $0<\arg a'+\arg b'<2\epsilon$ hold which imply $\arg \overline{a'b'}<0$ as desired.

\medskip

\noindent \textit{Case 2. The unitary part $\rho_u$ has discrete image in $U(1)$.} We now suppose the unitary part $\rho_u$ of $\rho$ is discrete, that means $\text{Im}(\rho_u)\cong\Z_m$ for some $m\ge2$ or trivial. We shall consider these cases separately.

\smallskip

\noindent \textit{Sub-case (i): $\rho_u$ non-trivial.} Suppose the unitary part $\rho_u$ of $\rho$ is a non-trivial representation with discrete image isomorphic to $\mathbb{Z}_m$ for some $m\ge2$. Our Lemma \ref{coax-lem1} applies and hence we can find a set of handle generators $\{\alpha,\beta\}$ such that
\[ \rho_u(\alpha)=\exp\Bigg(\,\frac{2\pi\,i}{m}\,\Bigg) \,\,\text{ and }\,\, \rho_u(\beta)=1.
\] In this case the proof does not differ much from the previous one but it can be simplified a little. Since $\text{Li}\circ\rho(\beta)=b\in\mathbb{R}^+$, we can immediately notice that the imaginary part $\Im{(A(q_t))}$ is constant as a function of $t$. In fact, by setting $q_t=A^{-1}B^{-1}(p_t)$ as above, it is sufficient to observe that
\[ A(q_t)=B^{-1}(p_t)=\frac{p_t-d}{b}.
\] In particular, since $p_0$ can be chosen arbitrarily, we may suppose $\Im{(p_0)}$ big enough to make $\Im{(A(q_t))}>0$. As a direct consequence, we can deduce that the segment $\overline{A(q_t), AB(q_t)}$ is entirely contained in the upper half-plane. 

\smallskip

\noindent The rest of the proof now proceeds as in the case $1$ above. For $t\longrightarrow +\infty$ the point $q_t$ tends to the infinity and $\arg q_t\longrightarrow -\arg a=-\frac{2\pi\,i}{m}<0$. For $t$ big enough, the open ball $B\big(bq_t,\,2|d|\big)$, which contains the segment  $\overline{q_t, B(q_t)}$ is entirely contained in the lower half-plane. Therefore, the chain of segments 
\begin{equation}
    AB(q_t)\longrightarrow A(q_t) \longrightarrow q_t \longrightarrow B(q_t)
\end{equation} is embedded in $\C$ because the edges $\overline{q_t\,B(q_t)}$ and $\overline{AB(q_t)\,A(q_t)}$ cannot intersect. Therefore the polygon \eqref{pentbaseoncomm} bounds an immersed disk on its right containing the infinity on the Riemann sphere.

\medskip

\noindent \textit{Sub-case (ii): $\rho_u$ trivial.} We begin by applying some reduction in order to put $a,b$ in a more convenient form. A first important fact to note is that, whenever $\rho_u$ is trivial, then $\text{Li}\circ\rho(\alpha)$ and $\text{Li}\circ\rho(\beta)$ are both real and different from zero.
Moreover, $a,b$ cannot be both equal to $1$. In fact, if this was the case, then $[A,B]=\text{I}$. In the case one between $a$ or $b$ is equal to $1$ we may apply a suitable Dehn-twist to make both different from $1$. Therefore, we can suppose $a,b\neq1$. We then apply our Claim \ref{claim1}, if necessary, to make them both greater that one in modulus; thus we may assume $|a|>|b|>1$. Finally, there is no loss of generality in assuming $b$ positive, whereas $a$ could be positive or negative. 
 
\smallskip
 
\noindent Let $q_t=A^{-1}B^{-1}(p_t)$ as above. Notice that $\Im{(\overline{ab})}=0$ because both $a,b\in\mathbb{R}^*$. In this special case formul\ae \,\eqref{realqt} and \eqref{imqt} simplify as follow:
\begin{equation}\label{realimqtred}
    \Re{(q_t)}=\frac{1}{ab}\Big(\Re{(p_t)}-\Re{(w})\Big)\,\text{ and }\,\Im{(q_t)}=\frac{1}{ab}\Big(\Im{(p_t)}-\Im{(w)}\Big),
\end{equation} where $w$ is defined as above. Thus the imaginary part of $q_t$, seen as a function of $t$, remains constant because it does not longer depend on $\Re{(p_t)}$. 

\smallskip

\noindent Let us consider the point $A(q_t)$. It is an easy matter to check that, for $t$ running to the infinity, the value $\Im{\big(A(q_t)\big)}$ seen as a function of $t$ remains constant. In fact, by recalling that $\Im{(p_t)}$ is assumed to be constant in $t$, it is sufficient to notice that
\[ A(q_t)=aq_t+c=\Bigg(\frac{1}{b}\,\Big(\Re{(p_t)}-\Re{(w)}\Big)+\Re{(c)}\Bigg)+i\,\Bigg(\frac{1}{b}\,\Big(\Im{(p_t)}-\Im{(w)}\Big)+\Im{(c)}\Bigg).
\]
\noindent In the same fashion, it is possible to check that 
\[ B(q_t)=bq_t+d=\Bigg(\frac{1}{a}\,\Big(\Re{(p_t)}-\Re{(w)}\Big)+\Re{(d)}\Bigg)+i\,\Bigg(\frac{1}{a}\,\Big(\Im{(p_t)}-\Im{(w)}\Big)+\Im{(d)}\Bigg),
\] and so even the imaginary part of $B(q_t)$ is constant as a function of $t$. Let us finally consider the points $BA(q_t)=p_t$ and $AB(q_t)=[A,B](p_t)$. As a consequence of our definitions even their imaginary parts are constants as functions of $t$. More precisely, we have already seen in \eqref{realimqtred} that $\Im{(q_t)}=\frac{1}{ab}\Big(\Im{(p_t)}-\Im{(w)}\Big)$ and an easy computation shows that \[ \Im{\big([A,B](p_t)\big)}=\Im{(AB(q_t))}=\Big(\Im{(p_t)}-\Im{(w)}\Big)+a\Im{(d)}+\Im{(c)}.
\]
\noindent Let us now recall that the starting point $p_0$ can be chose arbitrarily on the upper half-plane. This allows us to choose the initial point $p_0$ such that its imaginary part $\Im{(p_0)}$ is positive and sufficiently large to guarantee that the one of the following chain of inequalities holds for any $t\ge0$
\begin{align*}
    \frac{1}{ab}\Big(\Im{(p_t)}-\Im{(w)}\Big) & <\frac{1}{a}\,\Big(\Im{(p_t)}-\Im{(w)}\Big)+\Im{(d)}\\
    & <\frac{1}{b}\,\Big(\Im{(p_t)}-\Im{(w)}\Big)+\Im{(c)}\\
    & <\Big(\Im{(p_t)}-\Im{(w)}\Big)+a\Im{(d)}+\Im{(c)}
\end{align*} if $a>b>1$, or 
\begin{align*}
    \frac{1}{a}\Big(\Im{(p_t)}-\Im{(w)}\Big)+\Im{(d)} & <\frac{1}{ab}\,\Big(\Im{(p_t)}-\Im{(w)}\Big)\\
    & <\frac{1}{b}\,\Big(\Im{(p_t)}-\Im{(w)}\Big)+\Im{(c)}\\
    & <\Big(\Im{(p_t)}-\Im{(w)}\Big)+a\Im{(d)}+\Im{(c)}
\end{align*} if $a<0<b$ (recall that $1<|b|<|a|)$. Since $\overline{A(q_t)\,q_t}$ and $\overline{q_t\,B(q_t)}$ do not overlap and intersect only at $q_t$ for any $t$ large enough, we have that the chain of segments 
\begin{equation}
   AB(q_t)\longrightarrow A(q_t) \longrightarrow q_t \longrightarrow B(q_t)
\end{equation} is embedded in $\C$. Therefore, even in this case, the polygon \eqref{pentbaseoncomm} bounds an immersed disk on its right containing the infinity on the Riemann sphere.
\end{proof}

\noindent This completes the proof of the non-coaxial case and indeed the proof of Proposition \ref{affk1}. \qed

\smallskip

\subsubsection{An alternative proof}\label{altproofk1}  Proposition \ref{affk1} is also a consequence of the recent work of Le Fils in \cite{Fils2}, as we shall now describe. Indeed, that paper provides necessary and sufficient conditions for a representation $\rho:\pi_1(S_g)\longrightarrow \pslc$ to appear as the monodromy of some branched projective structure with prescribed singularities; here we are interested in the case of a single branch-point.   We begin by describing the obstructions for a representation to appear as the holonomy of some branched projective structure with a \emph{single} branch-point, following the discussion in \cite[\S 1]{Fils2}. 

\smallskip

\noindent \textit{The obstructions.} Here, we shall assume that $g\geq 2$. It is shown in \cite[Corollary 11.2.3]{GKM} that a representation $\rho:\pi_1(S_g)\longrightarrow \pslc$ that arises as the monodromy of a branched projective structure on $S_g$ with $n$ branch-points of orders $m_1,m_2,\dots,m_n$ lifts to a representation to $\slc$ if and only if $\sum m_i$ is even. Here we are interested in affine representations which are well-known to be liftable to a representation to $\slc$. Therefore, any representation $\rho:\pi_1(S_g)\longrightarrow \affc$ arises as the monodromy of some branched projective structure with one single branch-point of magnitude $2(m+1)\pi$ only if $m$ is even. This yields a first obstruction.

\smallskip

\noindent A second obstruction arises from the fact that whenever an affine representation $\rho$ arises as the monodromy of a branched projective structure with a single branch points of order $m$, then $m\ge2g-2$. In particular the equality holds if and only if the structure is a branched affine structure. This is stated in \cite[Proposition 6.18]{Fils2}.

\smallskip

\noindent When the image of $\rho$ is finite of order $N$, a third obstruction comes from the Riemann-Hurwitz formula. Let $\widehat{S}_g$ be the cover of $S_g$ associated to $\operatorname{ker}(\rho)$. The developing map yields a branched covering $\widehat{S}_g\longrightarrow \cp$ of degree $d$ and the Riemann-Hurwitz formula implies that
\begin{equation}
    N\chi(S_g)=\chi(\widehat{S}_g)=2d-Nm.
\end{equation} As $d$ cannot be smaller than $m+1$ we obtain
\begin{equation}\label{rhob}
    N\big(\,\chi(S_g)+m\big)\ge 2(m+1).
\end{equation} Compare with \cite[Section 6.2]{Fils2}.

\smallskip

\noindent An affine representation $\rho$ is said to be \emph{Euclidean} if $\text{Im}(\rho)$ is a subgroup of $\mathbb{S}^1\ltimes\C<\affc$. For Euclidean representations we may define the notion of \emph{volume} as a real number naturally attached to the representation. The volume appears as a further obstruction for realising a representation $\rho$ as the monodromy of a branched affine structure (and hence projective) on $S_g$ with a branch-point of order $2g-2$. However, this obstruction completely vanishes for realizing $\rho$ as the monodromy of some branched projective structure (no longer affine) on $S_g$ with a single branch-point of order $m\ge2g$. See \cite[Obstructions 4 and 5 in \S1]{Fils2} for further details. 

\noindent The sixth and last obstruction listed in \cite{Fils2}  concerns $g=2$ and dihedral (but not affine) representations, and rules out the possibility of a single branch-point of order two. We shall consider dihedral representations in the next section \S\ref{ssdih}.

\medskip

\textit{Remark.} The necessity of assuming that $\rho:\pi_1(S_{g,1})\longrightarrow \affc$ does not have a finite image of order two is a consequence of our Lemma \ref{excase}. Alternatively, such a necessity can be also deduced by the obstruction \eqref{rhob}. In fact, when $N=2$, the equation \eqref{rhob} is never satisfied for any $g\ge2$. 


\begin{proof}[Alternative proof of Proposition \ref{affk1}.] The proof is nothing but a direct consequence of the previous discussion. In \cite[Theorem 1.1]{Fils2} it is showed that an affine representation $\rho:\pi_1(S_g)\longrightarrow \affc$ arises as the holonomy of some branched projective structure, not necessarily affine, with one single branch-point if and only if it satisfies all the obstructions described above. By choosing $m$ even and bigger than $$\frac{2+k(2g-2)}{k-2}>2g-2,$$ we can see that the conditions of all the obstructions above are met. Hence, there exists a $\cp$-structure on $S_g$  with holonomy $\rho$ and  single branch-point of order $m$; deleting this branch-point we obtain our desired projective structure on $S_{g,1}$. \qedhere 
\end{proof}

\section{Dihedral representations}\label{ssdih}
\noindent A representation $\rho:\Pi\longrightarrow \pslc$ is called \emph{dihedral} if there exists a pair of points $F=\{p,q\}$ in $\cp$ globally preserved by the representation. Up to conjugation, we may assume $F=\{0,\infty\}$. Notice that a co-axial representation is, in particular, dihedral because the set $F$ is fixed point-wise by the representation. In this section we shall assume $\rho:\Pi \longrightarrow \pslc$ is a non-trivial, degenerate and dihedral (but not affine) representation. We recall for the reader's convenience that, according to our Definition \ref{degen}, a dihedral representation $\rho:\Pi\longrightarrow \pslc$ is degenerate if the monodromy around each puncture fixes the set $\{0,\infty\}$ pointwise and $\rho(\gamma)$ preserves $\{0,\infty\}$ for any $\gamma\in\Pi$.\\

\noindent The aim of this section is to prove the following

\begin{prop}\label{dihcase}
Let $\Pi$ be the fundamental group of a surface $S_{g,k}$ of negative Euler type and let $\rho:\pi_1(S_{g,\,k})\longrightarrow \pslc$ be a non-trivial, dihedral (but not affine) and degenerate representation such that at least one puncture has trivial monodromy. Then $\rho$ arises as the monodromy representation of a $\cp$-structure in $\mathcal{P}_g(k)$.
\end{prop}

\medskip 

\noindent The following observation is an immediate consequence of our definitions:

\begin{lem}\label{dihdegiscoa}
If a representation $\rho:\pi_1(S_{0,k})\longrightarrow \pslc$ is either 
\begin{itemize}
    \item[(a)] dihedral and degenerate, or
    \item[(b)] has an image which is a cyclic group of finite order,
\end{itemize}
then $\rho$ is co-axial and hence affine.
\end{lem}

\noindent In the light of the Lemma \ref{dihdegiscoa} (a) above we need to consider only surfaces of positive genus, since the genus $g=0$ case is covered by Proposition \ref{affg0}. We shall distinguish two cases according to the number of punctures. In both cases, we shall make use of the following technical result.

\begin{lem}\label{onedihhandle}
Let $\Pi$ be the fundamental group of a surface $S_{g,k}$ of genus $g$ and $k\ge0$ punctures. Let $\rho:\Pi\longrightarrow \pslc$ be a dihedral (not affine) and degenerate representation. Then there is a set of handle-generators $\{\alpha_i,\beta_i\}_{1\le i\le g}$ such that the restriction of $\rho$ to the handle generated by $\{\alpha_1,\,\beta_1\}$ is a dihedral representation and the restriction of $\rho$ to the complementary sub-surface is co-axial.
\end{lem}

\begin{proof}
The argument of this proof shares a few similarities with the proof of Proposition \ref{non-coax}. Let us begin with some observations. Being dihedral, the representation $\rho$ globally preserves a pair of points in $\cp$ which we may assume to be $\{0,\infty\}$. Since $\rho$ is not affine, there is a simple closed curve $\delta\in\Pi$ such that $\rho(\delta)(z)=\frac{1}{az}$ for some $a\in\C^*$. Being $\rho$ degenerate (Definition \ref{degen}), $\rho(\gamma)$ is a co-axial transformation of $\cp$ fixing $\{0,\infty\}$ for any simple closed curve $\gamma$ enclosing a puncture. Thus $\delta$ must be a handle generator.

\smallskip

\noindent We start with some set of handle-generators $\{\alpha_i,\beta_i\}_{1\le i\le g}$ and, in what follows, we modify this initial choice of generators by mapping class group elements until we get a new set of handle-generators with the desired property. Notice that, if $\rho(\eta)$ globally fixed $\{0,\infty\}$ for any $\eta\in\{\alpha_i,\beta_i\}$ then the representation $\rho$ would be co-axial, hence affine, and this leads to a contradiction. Therefore there exists $i\in\{1,\dots,g\}$ such that the restriction of $\rho$ to the handle $\{\alpha_i,\beta_i\}$ is a dihedral representation. We may even suppose that the points $\{0,\infty\}$ are point-wise fixed by $\rho(\alpha_i)$ and swap by $\rho(\beta_i)$; \textit{i.e.}
\begin{equation*}
\rho(\alpha_i)=
\begin{pmatrix}
a_i & 0\\
0 & 1
\end{pmatrix} \qquad \rho(\beta_i)=
\begin{pmatrix}
0 & 1\\
b_i & 0
\end{pmatrix}.
\end{equation*}
\noindent Let $\{\alpha_j,\beta_j\}$ be another pair and suppose that the restriction of $\rho$ to such handle is not co-axial. Up to replacing this pair with $\{\alpha_j\beta_j, \beta_j\}$ or $\{\alpha_j,\alpha_j\beta_j\}$ if needed, we may assume that both $\rho(\alpha_j)$ and $\rho(\beta_j)$ preserve the couple $\{0,\infty\}$ by swapping the points.

\smallskip

\noindent The basic modification here is as follows. We replace the handle-generators $\{\alpha_i,\,\beta_i\}$ with $\{\alpha_i,\, \beta_i\beta_j\}$ and we replace the handle-generators $\{\alpha_j,\beta_j\}$ with $\{\alpha_i^{-1}\alpha_j\beta_j,\,\beta_j\}$. This modification is effected by a mapping class element and the handle generated by $\{\alpha_i,\, \beta_i\beta_j\}$ remains disjoint from the handle generated by $\{\alpha_i^{-1}\alpha_j\beta_j,\,\beta_j\}$. It remains to show that the restriction of $\rho$ to one of these handles is dihedral and the restriction to the other handle is co-axial. By writing 
\begin{equation*}
\rho(\alpha_j)=
\begin{pmatrix}
0 & 1\\
a_j & 0
\end{pmatrix} \qquad\rho(\beta_j)=
\begin{pmatrix}
0 & 1\\
b_j & 0
\end{pmatrix},
\end{equation*}
it is an easy matter now to check that $\rho$ once restricted to the handle $\{\alpha_i,\, \beta_i\beta_j\}$ is co-axial because $\rho(\alpha_i)$ and $\rho(\beta_i)\rho(\beta_j)$ both fix the pair $\{0,\infty\}$ point-wise. In fact, $\rho(\beta_i)$ and $\rho(\beta_j)$ both swap $\{0,\infty\}$ and hence their product fix them point-wise. It is also easy to check that $\rho$ once restricted to the handle $\{\alpha_i^{-1}\alpha_j\beta_j,\,\beta_j\}$ is dihedral because $\rho(\beta_j)$ swaps the pair $\{0,\infty\}$. We can even notice that the transformation $\rho(\alpha_i)^{-1}\rho(\beta_i)\rho(\beta_j)$ keeps the pair $\{0,\infty\}$ point-wise fixed. An iterative argument will provide the a set of handle-generators with the desired property. Notice that this argument does not involved the number of punctures and hence the claim holds for any $k\ge0$.
\end{proof}


\subsection{Once-punctured surfaces} \label{dihk1} We begin by considering the case of once-punctured surfaces, \emph{i.e.} we assume $g>0$ and $k=1$.

\subsubsection{The once-punctured torus case} In this subsection we prove the following

\begin{lem}\label{dihg1k1}
Let $\rho:\pi_1(S_{1,1})\longrightarrow \pslc$ be a non-trivial, dihedral (but not affine) and degenerate representation such that the puncture has trivial monodromy. Then there is a projective structure on $S_{1,1}$ with monodromy $\rho$.
\end{lem}

\noindent Some generalities. Let $\rho:\pi_1(S_{1,1})\longrightarrow \pslc$ be a non-trivial dihedral (but not affine) and degenerate representation. We assume $\rho$ preserves $\{0,\infty\}$. Denote by $\alpha$ and $\beta$ denote two handle-generators and let $A=\rho(\alpha)$ and $B=\rho(\beta)$. We now distinguish two possible cases according on how many handle-generators act non-trivially on $\{0,\infty\}$. Without loss of generality, we may assume
\begin{equation}\label{standform2}
A=\begin{pmatrix}
a & 0\\
0 & 1
\end{pmatrix} \qquad
B=\begin{pmatrix}
0 & 1\\
b & 0
\end{pmatrix} ,
\end{equation} if one handle-generator fixes $\{0,\infty\}$ point-wise, or
\begin{equation}\label{standform3}
A=\begin{pmatrix}
0 & 1\\
a & 0
\end{pmatrix} \qquad
B=\begin{pmatrix}
0 & 1\\
b & 0
\end{pmatrix}
\end{equation} if both handle-generators act non-trivially on $\{0,\infty\}$. Notice that $a,b\in\C^*$. Let $\gamma=[\alpha,\beta]$ be a curve enclosing the puncture and let $\rho(\gamma)=[A,B]$. 

\smallskip 

\begin{proof}[Proof of Lemma \ref{dihg1k1}]Assuming the monodromy of the puncture to be trivial, \emph{i.e.} $\rho(\gamma)=\text{Id}$, it is possible to deduce some constraints about the possible values of $a,b$. A simple computation shows that $a=\pm1$ if $A,B$ are in the form \eqref{standform2} or $a=\pm b$ if $A,B$ are in the form \eqref{standform3}. Once again, given any base-point $p_0\in\C^\ast$, we can define in every case a polygon 
\begin{equation}\label{pol2}
 p_0\mapsto p_1\mapsto p_2\mapsto p_3\mapsto p_0
\end{equation} where the points $p_i$ are defined as: $p_1=A(p_0)$, $p_2=AB(p_0)=BA(p_0)$, and finally $p_3=B(p_0)$. The polygon 
bounds a possibly self-intersecting and possibly degenerate quadrilateral $\mathcal{Q}$ on the complex plane. As already done before, we shall denotes the directed edges as follows: $e_1=\overline{p_1\,p_2}$, $e_2=\overline{p_0\,p_1}$, $e_3=\overline{p_0\,p_3}$ and, finally, $e_4=\overline{p_3\,p_2}$. The edges of this polygon are related by the maps $A,B$ as follows: $A(e_3)=e_1$ and $B(e_2)=e_4$. Let us now discuss case by case.

\medskip 

\noindent \textit{Case 1: $A,B$ are in the form \eqref{standform2}.} We begin by observing that $a=1$ implies $A=I$ and therefore the image of $\rho$ is cyclic of order two. In particular, $\rho$ is co-axial (and hence affine) as observed in case (b) of the previous Lemma \ref{dihdegiscoa}. As $\rho$ is assumed to be dihedral but not affine, it follows that $a=-1$.  Given the matrices $A$ and $B$ as in the equation \eqref{standform2}, we notice that $p_0=-p_1$ and $p_2=-p_3$ because $A(z)=-z$. The polygon \eqref{pol2} is self-intersecting and bounds an immersed disk in $\cp$ containing the point $\infty$. Note that there always exists a good choice of $p_0$ such that the polygon is not degenerate (but still self-intersecting). This is because the polygon is degenerate whenever $p_0$, $-p_0=p_1$ and $p_3=B(p_0)$ are collinear and this happens if and only if there is a real $\lambda\neq0$ such that the equality $p_0^2\,(1-2\lambda)=b$ holds. It is easy to observe that then $p_2$ is necessarily collinear to the other three points. Then we proceed as in Proposition \ref{trans1} and Proposition \ref{affg1k1}.

\medskip 

\noindent \textit{Case 2: $A,B$ are in the form \eqref{standform3}.} In this second case we observe that $a=b$ implies $A=B$ and therefore the image of $\rho$ is cyclic of order two. In particular, $\rho$ is co-axial (and hence affine) as observed in Lemma \ref{dihdegiscoa}. As $\rho$ is supposed dihedral but not affine, it follows that $a=-b$. Given $A$ and $B$ as in the equation \eqref{standform3}, we note that $p_0=-p_2$ and $p_1=-p_3$. The base-point $p_0$ can be chosen in such a way the polygon \eqref{pol2} bounds an \textit{embedded} disk in $\cp$ containing the point $\infty$. More precisely, the polygon \eqref{pol2} is degenerate if and only if the three points $p_0$, $-p_0=p_2$ and $p_3$ are collinear and this happens whenever they satisfy the same relation as in Case 1 above. It is now an easy to see that there always $p_0\in\mathbb{C}$ such that the polygon \eqref{pol2} is non-degenerate. Even in this case we proceed as in Proposition \ref{trans1} and Proposition \ref{affg1k1}. \qedhere
\end{proof}

\subsubsection{Higher genus once-punctured surfaces} We then consider the case of once-punctured surfaces, \emph{i.e.} we assume $k=1$. Since we have already handled the case $g=1$ above, we can assume that $g\geq 2$. Here we shall prove the following

\begin{lem}\label{dihgk1}
Let $g\ge2$ and let $\rho:\pi_1(S_{g,1})\longrightarrow \pslc$ be a non-trivial, dihedral (but not affine) and degenerate representation such that the puncture has trivial monodromy. Then there is a projective structure on $S_{g,1}$ with monodromy $\rho$.
\end{lem}

\noindent Let $\rho:\pi_1(S_{g,1}) \to \pslc$ be a non-trivial, dihedral (but not affine) representation such that the puncture has trivial monodromy. We can regard $\rho$ as a representation $\overline{\rho}:\pi_1(S_g)\longrightarrow \pslc$ and therefore the basic idea, again, is to realize this latter as the monodromy of a branched projective structure with a single branch point. By deleting such a point, we will get the desired result.

\smallskip

\begin{proof}
By Lemma \ref{onedihhandle} we can assume that there is a handle, see Definition \ref{handle}, $H \subset S_g$ generated by the pair $\{\alpha_1,\beta_1\}$ such that $\rho\vert_H$ is dihedral (but not affine), and the restriction of $\rho$ to the complementary sub-surface is a co-axial representation. Recall that, up to conjugation, we may assume that $\rho$ globally preserves the pair of points $\{0,\infty\}\subset \cp$. Notice that, being $\rho\big([\alpha_i,\beta_i]\big)=\text{I}$ for each index $i=2,\dots,g$, it follows that $\rho\big([\alpha_1,\beta_1])=\text{I}$. Moreover, we may even assume that $\rho_{|H}(\alpha_1)$ and $\rho_{|H}(\beta_1)$ are in the form \eqref{standform3}. In fact, if they were in the form \eqref{standform2}, a suitable Dehn-twist make them in the desired form.

\smallskip

\noindent Let $\rho_0:\pi_1(S_{g-1,1})\longrightarrow \affc$ be the restriction of $\rho$ to the complement of $H\subset S_g$. It is co-axial and the puncture has trivial monodromy by construction. We can regard $\rho_0$ as a co-axial representation $\overline{\rho}_0:\pi_1(S_{g-1})\longrightarrow\affc$ and, by Proposition \ref{affk1} (in fact from the co-axial case of the proof in \S\ref{affk1coa}), $\overline{\rho}_0$ is realized as the holonomy of a branched projective structure on $S_{g-1}$ with a unique branch-point. Denote this projective surface by $S$. From the proof, see \S\ref{affk1coa}, this projective surface is in fact obtained by constructing a chain of quadrilaterals $\mathcal{C}$ in $\cp$ that bounds an immersed disk (see, for example, Figures \ref{chainofquads} or \ref{spiralquad}), and then identifying pairs of edges of these quadrilaterals. In particular, this construction defines a set of handle-generators that develop onto the edges of the quadrilaterals, which are embedded arcs in $\cp$. Let $\gamma$ be such a generator, which in our construction is a simple closed curve on the surface from the branch-point to itself, and let $\widehat{\gamma}\subset \cp $ be an embedded arc that it develops onto. 

\smallskip

\noindent Let us now consider the dihedral representation $\rho_{|H}:\pi_1(S_{1,1})\longrightarrow \pslc$. The puncture has trivial holonomy by construction, Lemma \ref{dihg1k1} applies and therefore $\rho_{|H}$ appears as the holonomy of a complex projective structure on a punctured torus. Let $\Sigma$ denote the projective handle, see Definition \ref{projhandle}, obtained by filling the puncture with a branched projective chart. From our construction, this projective handle is obtained by identifying sides of a quadrilateral $\mathcal{Q}$ on $\cp$ that bounds an embedded disk because the handle generators $\alpha_1,\beta_1$ are chosen in such a way $\rho_{|H}(\alpha_1)$ and $\rho_{|H}(\beta_1)$ are in the form \eqref{standform3}. In fact, from the proof of Proposition \ref{dihg1k1}, there was plenty of freedom in choosing the quadrilateral $\mathcal{Q}$, namely we could choose any base-point $p_0\in \C$ so that $\mathcal{Q}$, as defined by \eqref{pol2}, is non-degenerate and indeed embedded. In particular, we can choose such a base-point $p_0$ such that the following two properties are satisfied
\begin{itemize}
    \item $\widehat{\gamma}$ lies in the embedded disk bounded by $Q$, and 
    \item a vertex of $Q$ is an endpoint of $\widehat{\gamma}$.
\end{itemize} 
\noindent The desired structure with holonomy $\rho$ is then obtained by grafting the projective handle $\Sigma$ on $S$ along $\gamma$ as in Definition \ref{bubbhand}. The resulting surface is homeomorphic to $S_g$ and has a branched projective structure with a unique branch-point; recall that grafting in a handle does not change the monodromy of $H$ or its complement.  Deleting the branch-point we obtain our desired projective structure on $S_{g,1}$ with monodromy $\rho$. 
\end{proof} 
\medskip

\textit{Remark.} Even in this case there is an alternative proof can be derived from the results of  \cite{Fils2}. As in the preceding discussion, regard the representation  $\rho:\pi_1(S_{g,\,1})\longrightarrow \pslc$ with trivial monodromy around the puncture as a representation $\overline{\rho}:\pi_1(S_g)\longrightarrow\pslc$. As before, it suffice to realize this latter representation as the monodromy of a branched projective structure with one branch-point and thus obtain the desired projective structure on $S_{g,1}$ with monodromy $\rho$ by deleting the branch-point. However, according to the main Theorem in \cite{Fils2}, every dihedral representation $\pi_1(S_g)\longrightarrow \pslc$ can be realised as the monodromy of such a branched projective structure with a single branch-point of order at least three (\textit{c.f.} the last obstruction, as mentioned in \S\ref{altproofk1}). 

\subsection{Surfaces with at least two punctures.}\label{dihhgs} We finally consider punctured surfaces $S_{g,\,k}$ of genus at least one and with at least two punctures, \textit{i.e.} $g\ge1$ and $k\ge2$. This subsection is devoted to prove Proposition \ref{dihcase} for these remaining cases. We start by considering the case of surfaces with exactly two punctures and the general case with more than two punctures shall follow by extending our constructions. 

\begin{lem}\label{dihk2}
 Let $g\ge1$ and let $\rho:\pi_1(S_{g,2})\longrightarrow \pslc$ be a non-trivial dihedral (but not affine) and degenerate representation such that at least one puncture has trivial monodromy. Then there is a projective structure on $S_{g,2}$ with monodromy $\rho$.
\end{lem}

\begin{proof}
The case of genus $g=1$, namely of representations  $\rho:\pi_1(S_{1,2})\longrightarrow \pslc$ uses the same argument previously used for Proposition \ref{affk2} for the genus one case. The main difference here is that the sides are glued by elliptic transformations of order two preserving the pair $\{0,\infty\}\subset\cp$. Note that we can always find infinitely many rays joining the two punctures. The case of $g\geq 2$ is handled using the $g=1$ case, as we shall now describe. 

\smallskip

\noindent Let $\rho:\pi_1(S_{g,2})\longrightarrow \pslc$ be a dihedral and degenerate representation. By Lemma \ref{onedihhandle}, there is a set of handle-generators $\{\alpha_1,\beta_1,\dots,\alpha_g,\beta_g\}$ such that the restriction of $\rho$ to the handle generated by $\{\alpha_1,\,\beta_1\}$ is dihedral of the form \eqref{standform2} or form \eqref{standform3} and the restrictions of $\rho$ to each handle generated by $\{\alpha_i,\,\beta_i\}$, for $2\le i\le g$, is co-axial fixing $\{0,\infty\}$ point-wise. Assume $\rho(\gamma_1)=\text{Id}$, as a consequence we have that $\rho\big([\alpha_i,\beta_i]\big)=\text{Id}$ for all $i\ge2$ and $\rho\big([\alpha_1,\beta_1]\big)=\rho(\gamma_2)$ and is a dilation fixing $\{0,\infty\}$ point-wise.
\smallskip

\noindent Let $\rho_0\!\!:\!\pi_1(S_{g-1,2}) \to \pslc$ be the restriction of $\rho$ to the subsurface of $S_{g,2}$ homeomorphic to $S_{g-1,2}$ that contains all the handles with co-axial monodromy and one puncture with trivial monodromy. We notice that $\rho_0$ is a co-axial representation. Finally, let $\rho_1:\pi_1(S_{1,2}) \to \pslc$ be the restriction of $\rho$ to the complementary subsurface that contains, in particular, the remaining puncture. The representation $\rho_1$ is a dihedral (but non-affine) degenerate representation.

\smallskip
\noindent We may also assume the representation $\rho_0$ to be non-trivial. If $\rho_0$ was trivial then we can apply a proper change of basis $\{\alpha_2',\beta_2',\dots, \alpha_g',\beta_g'\}$ such that the restriction of $\rho_0$ to any handle $\langle\alpha_i',\beta_i'\rangle$ is not trivial and fix the set $\{0,\infty\}\subset \cp$ point-wise. In fact, since $\rho$ is dihedral but not affine, we may assume $\rho(\alpha_1)$ or $\rho(\beta_1)$ to be a dilation (that is of the form $z\mapsto c\,z$ for some $c\in\mathbb{C}^*$). Then we can apply the change of basis, handle by handle, as described in Lemma \ref{cbas} in order to get the desired basis.

\medskip

\noindent We start by considering the representation $\rho_0$. Our Proposition \ref{affk2} applies and therefore $\rho_0$ can be realized as the monodromy of some complex projective structure (in fact an affine structure) on $S_{g-1,2}$ with two punctures with trivial monodromy. We briefly recall the construction. Let $p\in\mathbb{C}^*$ be a point and suppose $p$ is not a fixed point of $A_i=\rho(\alpha_i)$ or $B_i=\rho(\beta_i)$ for any $i=1,\dots,g$. For any $i=2,\dots,g$, define $\mathcal{Q}_i$ to be the quadrilateral based at $p$ whose sides are defined by 
\begin{equation}\label{pol3}
 p\mapsto A_i(p)\mapsto A_iB_i(p)=B_iA_i(p)\mapsto B_i(p)\mapsto p
\end{equation} and define $\Sigma_i$ the $2-$punctured torus obtained by the cross-wise identification given by the mappings $A_i$ and $B_i$, where we subsequently delete the branch-point arising from the vertices of the polygon. 
Choose rays from $p$ to the puncture at infinity in $\mathbb{C}\setminus\mathcal{Q}_i$, for each $i$, and glue these surfaces along the rays (see Definition \ref{glue2}).
We thus obtain a surface $\Sigma$ homeomorphic to $S_{g-1,2}$ equipped with a complex projective structure (in fact, an affine structure) with monodromy $\rho_0$.

\smallskip

\noindent Let us now consider $\rho_1$. It is dihedral, degenerate and, by construction, at least one puncture has trivial monodromy. By the $g=1$ case handled at the beginning,  $\rho_1$ can be realized as the monodromy of some complex projective structure on $S_{1,2}$. In fact, given $p\in \mathbb{C}^*$ as above, we proceed as in the genus one case of Proposition \ref{affk2}: We can define an immersed polygonal curve based at $p$ (\textit{i.e.} the directed curve $L$ in Proposition \ref{affk2}) and then glue the sides of such a polygon by using the mappings $A_1$, $B_1$ and $[A_1,B_1]$. The resulting surface $\Sigma'$ is homeomorphic to $S_{1,2}$ and carries a complex projective structure with holonomy $\rho_1$. Note that there are arcs between the punctures that develop onto rays in $\mathbb{C}$ going towards the puncture at infinity.
\smallskip

\noindent We now glue together these two structures along the rays by using Definition \ref{glue-new}, as we now describe. Let $r$ any ray in $\Sigma$ joining the punctures and let $r'$ be any ray joining the two punctures of $\Sigma'$. By construction, the ray $r$ develops onto a ray $\overline{r}\subset \cp$ leaving from $p$. In the same fashion, the ray $r'$ develops onto a ray $\overline{r}'\subset\cp$ leaving from $p$. Note that these rays may or may not coincide, but they have the same starting point so Definition \ref{glue-new} applies. We glue the surfaces $\Sigma$ and $\Sigma'$ along the rays  $r$ and $r'$ as in that definition. The resulting surfaces is homeomorphic to $S_{g,2}$ and carries a complex projective structure with monodromy $\rho$ as desired.
\end{proof}

\begin{cor}\label{dihkmorethan2}
Let $g\ge1$, $k\ge3$ and let $\rho:\pi_1(S_{g,k})\longrightarrow \pslc$ be a non-trivial dihedral (but not affine) and degenerate representation such that at least one puncture has trivial monodromy. Then there is a projective structure on $S_{g,k}$ with monodromy $\rho$.
\end{cor}

\begin{proof}
Suppose there are more than two punctures, \textit{i.e.} $k>2$. Let $\rho:\pi_1(S_{g,k})\longrightarrow \pslc$ be a dihedral (but not affine) degenerate representation. Let $A_1,A_2,\dots,A_k$ be the monodromies of the punctures. Since $\rho$ is a degenerate representation, we can assume without loss of generality that $A_1=\text{Id}$. Let $\rho_0:\pi_1(S_{g,2}) \to \pslc$ be the restriction of $\rho$ to the subsurface of $S_{g,k}$ homeomorphic to $S_{g,2}$ that contains one puncture with trivial monodromy. Let $C=\rho_0(\gamma_2)$ the monodromy of the other puncture of $S_{g,2}$. We can notice that $A_2\,A_3\,\cdots\,A_k\,C^{-1}=\text{Id}$. Let $S_{0,k-1}$ be the $(k-1)$-punctured sphere and let $\delta_i$ denotes a curve enclosing the $i$-th puncture. Similarly, we define $\rho_1:\pi_1(S_{0,k-1}) \longrightarrow \pslc$ to be the representation such that $\rho_1(\delta_i)=A_i$ for any $i=1,\dots,k-2$ and $\rho_1(\delta_{k-1})=A_k\,C^{-1}$. Note that the representation $\rho_1$ is by itself an affine representation.

\smallskip

\noindent Let us consider first the representation $\rho_0$. Our previous Lemma \ref{dihk2} applies and $\rho_0$ can be realized as the monodromy of some complex projective structure on $S_{g,2}$. Let us now denote by $\Sigma$ the surface $S_{g,2}$ equipped with such a structure. It follows by construction that there exists arcs between the punctures (in fact infinitely many) that develop onto rays in $\C$. Recall that one of these punctures is an apparent singularity and any neighborhood of it is locally modelled on a punctured disk centered at some point $p\in\mathbb{C}$. Let us fix any such arc $r\subset \Sigma$ joining the punctures and denote by $\overline{r}$ its developed image, which is an infinite ray on $\mathbb{C}$ leaving $p$ towards the infinity.

\smallskip

\noindent Let us now consider the affine representation $\rho_1$. According to the proof of our Proposition \ref{affg0}, after an appropriate choice pf a  base-point, $\rho_1$  appears as the holonomy of some branched affine structure (and hence a branched projective structure) on $S_{0,k-1}$.  We denote by $\Sigma'$ the surface $S_{0,k-1}$ equipped the branched affine structure we obtain by choosing $p$ as the base-point, where $p$ is the point we saw above. Let $r^\prime$ be an arc from the unique branch-point, say $q$, to the puncture with holonomy $A_k\,C^{-1}$. This ray develops on a ray $\overline{r}'$ leaving from $p$. Note that $\overline{r}$ and $\overline{r}'$ are two rays based at $p$; in particular, they intersect only at $p$ if they do not coincide.

\smallskip

\noindent It finally remains to glue together structures $\Sigma$ and $\Sigma'\setminus \{q\}$, along the rays $r$ and $r'$ defined above, by the gluing construction described in Definition \ref{glue-new}. 
\noindent The resulting surface, after the gluing, is homeomorphic to $S_{g,k}$ and carries a complex projective structure with holonomy $\rho$ as desired.
\end{proof}

\noindent This concludes the construction of the general case and indeed the proof of Proposition \ref{dihcase}. $\,\,\,\,\qed$


\section{Corollaries}\label{coros}

\subsection{Infinite fibers}\label{fib} Here we provide a proof of Corollary \ref{cor:fiber}. Let $\rho:\Pi \to \pslc$ be a representation that satisfies the requirements of Theorem \ref{thm1}, so that there exists a projective structure in $\mathcal{P}_g(k)$ with monodromy $\rho$. Here, we shall describe how the proof of Theorem \ref{thm1} shows that in fact, the set of such projective structures with monodromy $\rho$ is infinite. That is, for any such $\rho$, the fiber of the monodromy map $\Psi^{-1}(\rho)$ is infinite in cardinality. 

\smallskip

\noindent For this, we recall the following surgery, well-known in the context of branched projective structures, see \cite[Section \S12.1]{GKM} or \cite[Definition 2.5]{CDF} and references therein. 

\begin{defn}[Bubbling]\label{bubb} Let $S$ be a surface equipped with a projective structure, and let $\gamma$ be an embedded arc on $S$ with from one puncture to another, such that the developing image is an embedded arc $\widehat{\gamma}$ in $\cp$. We shall call such an arc $\gamma$ an \textit{admissible arc} for the $\cp$-structure on $S$. Take a copy of $\cp$ slit along $\widehat{\gamma}$, and let $\widehat{\gamma}_+$ and $\widehat{\gamma}_-$ be the resulting sides of the slit. Cut $S$ along $\gamma$, and identify the resulting sides with $\widehat{\gamma}_\pm$ so that the resulting surface $S^\prime$ is homeomorphic to $S$, and acquires a projective structure. The developing map of this new projective structure, when restricted to a fundamental domain, now wraps an additional time around $\cp$; however, the monodromy remains unchanged. Notice that we have already implicitly used this fact in Definition \ref{bubbhand}.  Moreover, a computation exactly as in \S\ref{sdpunct} shows that the resulting projective structure is also in $\mathcal{P}_g(k)$, that is, the Schwarzian derivative of developing map has a pole of order at most two at the punctures. 
\end{defn}

\noindent Indeed, once we have an admissible arc as in the definition above, then we can perform the bubbling operation $m$ times for any $m$, each time adding a new copy of $\cp$ along $\gamma$, thus obtaining infinitely many projective structures with the same monodromy. 

\smallskip 

\noindent It only remains to show that there exists admissible arcs in any of the $\cp$-structures we construct in the course of the proof of Theorem \ref{thm1}. If the representation $\rho$ is non-degenerate, then recall from \S\ref{consprojstrucnd} that the projective structure on a surface $S$ with monodromy $\rho$ is obtained by considering a $\rho$-equivariant pleated plane $\Psi:\widetilde{S}\to \mathbb{H}^3$, and then taking its ``shadow" at the conformal boundary at infinity $\partial_\infty \mathbb{H}^3 = \cp$. It follows from that construction that any of the pleating lines of $\Psi$ is the lift of an admissible arc on $S$; see also \cite[Theorem 1.3]{GupMon1} and its proof.

\noindent For a degenerate representation $\rho$, note that
\begin{itemize}
    \item in the case of the trivial representation handled in \S\ref{triv}, either $g=0$, in which case any arc between punctures is admissible, or else $g>0$, in which case the projective structure is obtained by taking a branched cover of $\cp \setminus \{0,1,\infty\}$. Since we can obtain infinitely many $\cp$-structures on the latter by bubbling along any arc between the three punctures, their pullbacks of under the same topological branched cover defines an infinite set of points in the fiber, as desired. 
    \item in all remaining constructions in \S\ref{proof2}-\S\ref{ssdih}, there is a handle-generator that develops onto an edge of a polygonal curve in $\cp$, and is hence admissible.
\end{itemize}

\noindent This completes the proof of Corollary \ref{cor:fiber}. We note that the above argument proves that each non-empty fiber is at least \textit{countably} infinite; however, as noted at the end of \S1, there are representations with connected (and hence uncountably infinite) fibers. 


\subsection{Spherical cone-metrics}\label{spher} Here we provide a proof of Corollary \ref{cor:spher}. Since a spherical cone-metric is also a $\cp$-structure on the punctured surface obtained by deleting the cone-points, the ``only if" direction  is an immediate consequence of Theorem \ref{ab61}, and Lemmata \ref{trivex} and \ref{excase}. Namely, it follows from these results that the holonomy of such a structure satisfies conditions (i) and (ii) of Theorem \ref{thm1}.  In what follows, we shall prove the "if" direction by handling the cases of non-degenerate and degenerate holonomy separately. 

\medskip 
\noindent Let $\rho: \Pi \to \text{SO}(3,\mathbb{R})$ be a non-degenerate representation. Note that this can be thought of as a representation into $\pslc$ that is unitary. In particular, note that the monodromy around any puncture is either elliptic or the identity element. By Proposition \ref{prop:nondegen} one can construct a $\cp$-structure $P$ on $S_{g,k}$ with monodromy $\rho$.  By virtue of the holonomy lying in the isometry group of the round metric on $\cp$, the punctured surface acquires a spherical metric.  It only remains to verify that the punctures are cone-points (or regular points if the cone-angle is $2\pi$). This is a consequence of our construction in \S\ref{consprojstrucnd}, see also \cite[Section 3]{GupMon1}. In what follows we describe briefly how the developing map for $P$ extends to each puncture as a branch-point.

\smallskip

\noindent Consider the $\rho$-equivariant pleated plane $\Psi$ in $\mathbb{H}^3$, see \S\ref{pleatedplane}; recall that the image of $\Psi$ comprises totally-geodesic ideal triangles with vertices in the image of the framing map $\beta:F_\infty \to \cp$. The edges of these totally-geodesic ideal triangles form an equivariant collection of \textit{pleating lines}, which are geodesic lines, each with a weight in $(0,2\pi)$ which equals the dihedral angle between the two adjacent  ideal triangles adjacent at the pleating line.

\noindent Let $\bar{p}$ be a puncture on $S_{g,k}$. Since the monodromy around $\bar{p}$ is elliptic, from \cite[Lemma 3.2]{GupMon1} it follows that up to the equivariance, there will be finitely many pleating lines incident on any lift  $p \in F_\infty$, and the  sum of their weights will be positive. Interpreting this in terms of our construction of $P$, this implies that, in the language of \S\ref{sdpunct}, the "total bending angle" $\alpha$ around $\beta(p)$ is positive. The developing map $f$ of the projective structure $P$ then takes a neighborhood of the puncture into the portion of a lune $L_\alpha$ in $\cp$ that lies in a neighborhood of one of its endpoints $\beta({p})$. The developing map can thus be extended to ${p}$ by mapping it to $\beta({p})$; as explained in \S3.4, in a conformal coordinate $w$ on the surface in a neighborhood of the puncture $\bar{p}$, if we take $\beta({p}) = 0\in \cp$ the developing map has the form $w\mapsto w^{\alpha/2\pi}$. This differs slightly from the map $\tilde{f}$ in \S\ref{sdpunct} since there we took $\beta({p}) = \infty \in \cp$. The puncture $\bar{p}$ is thus a cone-point of angle $\alpha$ (and a regular point if $\alpha = 2\pi$), as desired.

\medskip

\noindent Now let $\rho: \Pi \to \text{SO}(3,\mathbb{R})$ be a degenerate representation satisfying condition (ii) of Theorem \ref{thm1}. The constructions of \S4 apply to produce a $\cp$-structure on $S_{g,k}$ with monodromy $\rho$.  Away from the punctures, the charts to $\cp$ for this projective structure can be considered as charts to the round sphere $\mathbb{S}^2$. So as observed above, since the monodromy of any curve is an element of $\text{SO}(3,\mathbb{R})$, i.e.\ an isometry of  $\mathbb{S}^2$, the pullback of the spherical metric defines a spherical metric on the punctured surface. 
The key observation is that our constructions in \S\ref{proof2} always produce projective structures where the punctures are cone-points. Indeed, a puncture that is an apparent singularity is necessarily a branch-point (i.e.\ with cone-angle an integer-multiple of $2\pi$)  or a regular point (when the cone-angle is exactly $2\pi$). A puncture with non-trivial monodromy around it, say an elliptic rotation of angle $\alpha$, has a cone-angle $\alpha + 2\pi n$ for some integer $n\geq 0$. In particular, the developing map always extends to the puncture and has the form $z\mapsto z^{\alpha/2\pi}$ in a coordinate disk centered at the puncture. This defines a spherical cone-metric on $S_{g,k}$ with monodromy $\rho$ and cone-points at the punctures, as desired. 

\medskip

\subsection{Branched projective structures}\label{branch} We finally provide a proof of Corollary \ref{cor:branch}. Our main Theorem \ref{thm1} already covers all representations except those that are ``exceptional" in the following sense.

\begin{defn}
An exceptional representation $\rho:\pi_1(S_{g,k})\longrightarrow\pslc$ is necessarily degenerate and satisfies one of the following additional conditions
\begin{itemize} \item[-] $\rho$ does not have any apparent singularity (in the sense of Definition \ref{appsing}), or \item[-] $\rho$ is trivial when $g>0$ and $k=1$ or $2$, or \item[-] $\rho$ has an apparent singularity, but the image of $\rho$ is a group of order two, when $g>0$ and $k=1$. \end{itemize} 
\end{defn}

\noindent Our proof of Corollary \ref{cor:branch} is then an immediate consequence of the following Lemmata, that deals with each of these possibilities.

\begin{lem}\label{noappsing}
Let $\rho:\pi_1(S_{g,k})\longrightarrow \pslc$ be a degenerate representation without apparent singularity. Then $\rho$ arises as the monodromy of a branched $\cp$-structure on $S_{g,k}$ with a single branch-point.
\end{lem}

\begin{proof}[Proof of Lemma \ref{noappsing}]
Let $\rho:\pi_1(S_{g,k})\longrightarrow \pslc$ be a degenerate representation without any apparent singularity. Note that $\rho$ cannot be trivial. Then we may regard $\rho$ as a representation $\overline{\rho}:\pi_1(S_{g,k+1})\longrightarrow\pslc$ such that the monodromy of the extra puncture is trivial. Notice that $k+1\ge2$. Our Theorem \ref{thm1} applies and the representation $\overline{\rho}$ arises as the monodromy of a complex projective structure with one apparent singularity. We eventually fill the apparent singularity with a, necessarily branched, complex projective chart. The resulting structure is therefore a branched projective structure with a single branch-point and monodromy $\rho$.
\end{proof}

\begin{lem}\label{z2im}
Let $\rho:\pi_1(S_{g,1})\to \pslc$ be a degenerate representation with the puncture having trivial monodromy such that $\text{\emph{Im}}(\rho)\cong\mathbb{Z}_2$. Then $\rho$ arises as the monodromy of a branched $\cp$-structure on $S_{g,1}$ with a single branch-point.
\end{lem}

\begin{proof}[Proof of Lemma \ref{z2im}]
Let $\rho:\pi_1(S_{g,1})\to \pslc$ be a degenerate representation such that its image $\text{Im}(\rho)\cong\mathbb{Z}_2$. Assume the puncture has trivial monodromy. In this case we regard $\rho$ as a representation $\overline{\rho}:\pi_1(S_{g,2})\longrightarrow\pslc$ such that the monodromy of the extra puncture is trivial. Note that both punctures have trivial monodromy. Let us consider first the case $g=1$. Let $e=\overline{p\,q}\subset \C$ be any segment such that $p,q\notin\{0,\infty\}$. Slit $\cp$ along $e\,\cup\,-e$ and denote the resulting sides as $e^{\pm}$ and $-e^{\pm}$. Then glue $e^+$ with $-e^+$ and $e^-$ with $-e^-$ to obtain a half-translation structure $\Sigma$ on a torus and two branch-points of magnitude $4\pi$. By removing one of them we obtain a branched projective structure on $S_{1,1}$ with monodromy $\rho$. Assume now $g\ge2$. By construction, we can always find a geodesic segment $r$ joining the two branch-points on $\Sigma$. Suppose $\Sigma_1,\dots,\Sigma_g$ are $g$ copies of $\Sigma$. For any $i=1,\dots,g$, we slit $\Sigma_i$ along $r_i$ and denote the resulting segments $r_i^+$ and $r_i^-$. We then glue the $\Sigma_i$'s together by identifying $r_i^-$ with $r_{i+1}^+$. The resulting surface is homeomorphic to $S_g$ and carries a branched projective structure with two branch-points each one of magnitude $4g\pi$. By removing one of them we obtain a branched projective structure on $S_{g,1}$ with a single branch-point and desired monodromy.
\end{proof}

\begin{lem}\label{trivk12}
Let $k=1,\,2$ and let $\rho:\pi_1(S_{g,k})\longrightarrow \pslc$ be the trivial representation. Then $\rho$ arises as the monodromy of a branched $\cp$-structure on $S_{g,k}$ with a single branch-point if $k=2$ or two branch-points if $k=1$.
\end{lem}

\begin{proof}[Proof of Lemma \ref{trivk12}]
Let $k=1,2$ and let $\rho:\pi_1(S_{g,k})\longrightarrow \pslc$ be the trivial representation. We can regard $\rho$ as the trivial representation $\overline{\rho}:\pi_1(S_{g,3})\longrightarrow \pslc$. Our Lemma \ref{triv} applies and hence $\overline{\rho}$ appears as the monodromy of a complex projective structure on $S_{g,3}$. We eventually fill one or two punctures with a (necessarily) branched projective chart depending on whether $k=2$ or $k=1$ respectively. In both cases, we obtain a branched projective structure on $S_{g,k}$, with $k=1,2$, with trivial monodromy.
\end{proof}

\bibliographystyle{amsalpha}
\bibliography{monodromy-aug-new}

\end{document}